\documentclass[11pt]{amsart}
\usepackage{amsmath,amssymb,amsthm,comment,bm,setspace,mathtools,appendix}
\usepackage[utf8]{inputenc}
\usepackage[margin=1in]{geometry}
\usepackage{paralist,tikz-cd}
\usepackage{euscript}
\usepackage{mathrsfs}
\usepackage[pagebackref]{hyperref}
\hypersetup{colorlinks=true,linkcolor=purple,anchorcolor=green,citecolor=blue,filecolor=black,menucolor=black,urlcolor=brown}

\usetikzlibrary{decorations.pathmorphing}
\usepackage{BOONDOX-uprscr}
\usepackage{mathrsfs}
\usepackage[scr=rsfs,cal=boondox]{mathalfa}

\theoremstyle{plain}

\newtheorem{theorem}{Theorem}[section]
\newtheorem{lemma}[theorem]{Lemma}

\newtheorem{proposition}[theorem]{Proposition}
\newtheorem{corollary}[theorem]{Corollary}

\theoremstyle{remark}
\newtheorem{defn}{Definition}[section]
\newtheorem{defnprop}{Definition-Proposition}[section]

\newtheorem{remark}[theorem]{Remark}
\newtheorem{example}[theorem]{Example}

\newtheorem{conjecture}[theorem]{Conjecture}

\newcommand{\gr}{{\mathrm{gr}}}

\title{On a conjecture of Hosono--Lee--Lian--Yau}
\author{Andrew Harder and Sukjoo Lee}

\date{}

\begin{document}
\begin{abstract}
    We extend the mirror construction of singular Calabi–Yau double covers, introduced by Hosono, Lee, Lian, and Yau, to a broader class of singular Calabi–Yau $(\mathbb{Z}/2)^k$-Galois covers, and prove Hodge number duality for both the original and extended mirror pairs. A main tool in our approach is an analogue of the Cayley trick, which relates the de Rham complex of the branched covers to the twisted de Rham complex of certain Landau–Ginzburg models. In particular, it reveals direct relations between the Hodge numbers of the covers and the irregular Hodge numbers of the associated Landau--Ginzburg models. This construction is independent of mirror symmetry and may be of independent interest.
\end{abstract}
\maketitle

\section{Introduction}
\subsection{Double covers and Landau--Ginzburg models}\label{s:intro doublecover}
Let $X$ denote a smooth, projective variety. All varieties in this paper will be defined over $\mathbb{C}$. The Hodge theory of a branched double cover
\[
\rho: \widehat{X} \to X
\]
is a classical topic that has been extensively studied; for instance in the foundational work of Esnault and Viehweg \cite{Esnault1992vanishing} (see also \cite{Arapura2014cyclic, Schnell2016vanishing}). Without loss of generality, we may assume that the double cover $\widehat{X}$ is determined by a global section $\sigma \in \Gamma(X, L^{ 2})$ of a line bundle $L$ on a smooth variety $X$. The branch locus is the zero locus of $\sigma$, denoted by $B = \{\sigma = 0\}$, and the ramification locus is given by the preimage of $B$, $R:=\rho^{-1}(B)$.  When $B$ is smooth (or more generally, a simple normal crossings divisor), the de Rham complex $\mathrm{DR}({\widehat{X}}) := (\Omega^\bullet_{\widehat{X}}, \mathrm{d})$, equipped with the usual Hodge filtration $F^\bullet$, admits a natural decomposition induced by the canonical $(\mathbb{Z}/2)$-action on $\widehat{X}$:
\begin{equation}\label{e:intro DR Z/2}
\mathrm{R}\rho_*\mathrm{DR}({\widehat{X}}) = \mathrm{DR}({\widehat{X}})^{(+)} \oplus \mathrm{DR}({\widehat{X}})^{(-)}
\end{equation}
where the superscripts $(+)$ and $(-)$ indicate the invariant and coinvariant part, respectively. The invariant part is isomorphic to the de Rham complex of $X$, $\mathrm{DR}(X)$, while the coinvariant part is naturally identified with the following twisted de Rham complex on the open subset $X \setminus B$, 
\[
\mathrm{DR}(X \setminus B, \log \sqrt{\sigma}) := (\Omega_X^\bullet(\log B) \otimes L^{-1},\mathrm{d} + \tfrac{1}{2}\mathrm{d}\log \sigma).
\]
We often denote its cohomology by $H^*(\widehat{X}\setminus R)^{(-)}$ and Hodge filtration by $F^\bullet$, which is simply defined as the stupid filtration as usual. Then the identification in \eqref{e:intro DR Z/2} respects the Hodge filtration $F^\bullet$. 

An interesting feature of the setup above is that the data $(L, \sigma)$ gives rise to two distinct Landau--Ginzburg models. The first one is a pair $(\mathrm{Tot}(L^{-2}), g_{1,\sigma})$ where the potential $g_{1,\sigma}:=\langle \sigma, -\rangle$ is the linear map induced by the section $\sigma$. The second one is a pair $(\mathrm{Tot}(L^{-1}), g_{2,\sigma})$ where the potential $g_{2,\sigma}=\langle \sigma, (-)^2 \rangle$ is the quadratic function induced by $\sigma$. One of our goals is to study Hodge theory of these Landau--Ginzburg models and verify the relationship with Hodge theory of the double cover $\widehat{X}$ (see Theorem \ref{t:intro hyps}).


The Hodge theory of Landau--Ginzburg models has been developed in recent years \cite{katzarkov2017bogomolov, esnault20171, yu2014irregular, shamoto2018hodge, harder, Lee2024mirrorPW}. Recall that for any Landau--Ginzburg model $(Y,f: Y \to \mathbb{A}^1)$, the relevant de Rham cohomology is given by the cohomology of the twisted de Rham complex $\mathrm{DR}(Y, f):=(\Omega_Y^\bullet, \mathrm{d} + \mathrm{d}f)$ where we simply write $\mathrm{d}f:=\mathrm{d}f\wedge$. The Hodge filtration on cohomology was constructed by Yu \cite{yu2014irregular}, which is called the \emph{irregular Hodge filtration}, denoted by $F^\bullet_{\mathrm{irr}}$. This filtration is a finite $\mathbb{Q}$-filtration where the non-integrality of the grading reflects the multiplicities of the polar divisor of $f$. We review some basic properties of the irregular Hodge filtration in Section \ref{s:coverLG}.

According to the general philosophy of the Cayley method, the Hodge-theoretic data of the Landau--Ginzburg model $(\mathrm{Tot}(L^{-2}), g_{1,\sigma})$ reflects the Hodge theory of the branch hypersurface $B \subset X$. More precisely, there is a filtered isomorphism
\[
(\mathrm{DR}_B(X), F^\bullet) \cong (\mathrm{R}\pi_*\mathrm{DR}(\mathrm{Tot}(L^{-2}), g_{1,\sigma}), F^\bullet_{\mathrm{irr}}),\qquad \pi : \mathrm{Tot}(L^{-2}) \longrightarrow X.
\]
where $\mathrm{DR}_B(X)$ denotes the de Rham complex of $X$ localized to $B$. A natural question then arises: \emph{what kind of Hodge-theoretic information is encoded in the Landau--Ginzburg model $(\mathrm{Tot}(L^{-1}), g_{2,\sigma})$}?

In the first part of this article, we answer this question. The key observation is that the Landau--Ginzburg model $(\mathrm{Tot}(L^{-1}), g_{2,\sigma})$ is the double cover of $(\mathrm{Tot}(L^{-2}), g_{1,\sigma})$ branched along the zero section, so that there is a $(\mathbb{Z}/2)$-decomposition similar to \eqref{e:intro DR Z/2}. 
\begin{theorem}[Theorem \ref{t : single direct sum decomposition}]\label{t:intro hyps}
Let the notation be as above. There is a filtered isomorphism
\[
\begin{aligned}
    H^*(\mathrm{Tot}(L^{-1}), g_{2,\sigma}) & \cong H^*(\mathrm{Tot}(L^{-2}), g_{1,\sigma}) \oplus H^{*-1}(\widehat{X} \setminus R)^{(-)}(1/2) \\
    &\cong H_B^*(X) \oplus H^{*-1}(\widehat{X} \setminus R)^{(-)}(1/2).
\end{aligned}
\]
Here, $H^*(\widehat{X}\setminus R)^{(-)}$ and $H_B^*(X)$ carry their usual pure Hodge structures, and the notation $(1/2)$ indicates a rational shift by $1/2$ (see \S \ref{sss:lg} for details).
\end{theorem}

\begin{remark}
The framework discussed here extends naturally to the setting of cyclic covers of higher degree. However, for the purposes of this article, particularly in view of an application to be addressed later, we focus exclusively on double covers. We also note that a more general and abstract form of such decompositions has appeared in the work of Sabbah and Yu \cite{Sabbah2023cyclic} in the affine case, formulated using the language of exponential mixed Hodge modules. We believe that a similar formalism should hold in the global (non-affine) case as well, which would be of independent interest.
\end{remark}

Finally, one can generalize this construction to iterated or fibered double covers (cf. \cite[Section 4]{Lee2025CYdouble}). Fix an index set $[k]=\{1, \cdots, k\}$. Let $\{(L_i, \sigma_i)|i\in [k]\}$ be a collection of pairs, where each $L_i$ is a line bundle on $X$ and $\sigma_i \in \Gamma(X, L_i^{\otimes 2})$. Each pair defines a branched double cover $\widehat{X}_i$ branched over $B_i := \{\sigma_i = 0\}$ as before. For any subset $I=\{i_1, ..., i_r\} \subset [k]$, the fiber product
\[
\widehat{X}_I:=\widehat{X}_{i_1} \times_X \cdots \times_X \widehat{X}_{i_r}
\]
is a $(\mathbb{Z}/2)^{|I|}$-Galois cover of $X$ branched over the divisor $B_I := \bigcup_{j\in I} B_j$. Let $\rho_I:\widehat{X}_I \to X$ be the covering map. Moreover, this cover can be realized as the complete intersection
\[
\{ y_i^2 - \sigma_i = 0 \mid i\in I\} \subset \operatorname{Tot}\left( \bigoplus_{i\in I} L_i \right)
\]
where $y_i$ is the fiber coordinate of $L_i$. On the other hand, for any subset $J$, we consider the Landau--Ginzburg model 
\[
\mathbb{V}_J:=\mathrm{Tot}\left(\bigoplus_{i\in J} L^{-1}_i \oplus \bigoplus_{i\notin J} L^{-2}_i \right), \quad g_{2,J}+g_{1,J^c}
\]
where $g_{2,J}=\sum_{i \in J} g_{2,\sigma_i}$ and $g_{1,J^c}=\sum_{i \notin J} g_{1,\sigma_i}$. Our first main theorem finds the Hodge theoretic relationship between this Landau--Ginzburg model and the Galois covers $\widehat{X}_I$'s for $I \subseteq J$.

\begin{theorem}[Theorem \ref{t : multiple direct sum decomposition}]\label{t:intro com.int}
    For any subset $J \subseteq [k]$, there is a filtered isomorphism 
    \begin{equation}\label{e:mdsd}
    H^*(\mathbb{V}_{J}, g_{2,J}+g_{1,J^c}) \cong \bigoplus_{I \subseteq J}H^{*-|I|}_{Z^I}(\widehat{X}_I \setminus R_I)^{(I)}(|I|/2)
    \end{equation}
    where $Z^{I}=\rho_I^{-1}(B^{I^c})$ and $B^{I^c} = \cap_{i \notin I} B_i$. Here the superscript $(I)$ indicates the $(-1)^{|I|}$-eigenspace of $(\mathbb{Z}/2)^{|I|}$ action on $\widehat{X}_I\setminus R_I$. 
\end{theorem}
The summands on the right hand side of \eqref{e:mdsd} are equipped with fractionally shifted versions of the usual Hodge filtration, while the left hand side carries Yu's irregular Hodge filtration. In the case where $I=[k]$, the right hand side of \eqref{e:mdsd} has a factor of $H^{*-k}(\widehat{X}_{[k]}\setminus R_{[k]})^{([k]^-)}(k/2)$. We show in Proposition \ref{p : double cover interpretation of the first factor} that this cohomology group is a direct summand of the cohomology of the double cover $\widetilde{X}\rightarrow X$ branched along $B_{[k]}$. Therefore, there is a concrete cohomological relationship between the irregular Hodge numbers of $(\mathbb{V}_{[k]},g_{2,{[k]}})$ and those of the branched double cover $\widetilde{X}\rightarrow X$. This is a key observation for the applications described in the following section.
\subsection{Applications to mirror symmetry}
In mirror symmetry, one of the foundational achievements is the work of Batyrev and Borisov, who established a striking duality between certain Calabi–Yau complete intersections in toric varieties \cite{batyrev1996mirror}. As we will rely on this construction, let us first recall it. Their setup begins with a reflexive Gorenstein cone or, more concretely, with a pair of dual reflexive polytopes equipped with nef partitions, 
\[
(\Delta, \{\Delta_i\}_{i=1}^k) \quad \text{and} \quad (\check{\Delta}, \{\check{\Delta}_i\}_{i=1}^k)
\]
where $\Delta=\Delta_1\cup\cdots\cup\Delta_k$ and $\check{\Delta}=\check{\Delta}_1\cup\cdots\cup\check{\Delta}_k$. Let $T(\Sigma_{\Delta})$ (resp. $T(\Sigma_{\check{\Delta}})$ be the toric variety associated to the spanning fan $\Sigma_\Delta$ of $\Delta$ (resp. $\Sigma_{\check{\Delta}}$) of $\Delta$ (resp. $\check{\Delta}$). Throughout this article, we assume that both admit maximal projective crepant smooth (MPCS, for short) resolutions. This corresponds to the existence of projective unimodular triangulations of both $\Delta$ and $\check{\Delta}$. Thus, we choose such resolutions and denote them by $T_\Delta$ and $T_{\check{\Delta}}$, respectively. 

The nef partition induces the decomposition of the anticanonical line bundle:
\[
K^{-1}_{T_\Delta} = \bigotimes_{i=1}^k \mathcal{O}(E_{\Delta_i}),
\]
where each $E_{\Delta_i}$ is the toric divisor corresponding to the chosen projective unimodular triangulation of $\Delta_i$ and $\mathcal{O}(E_{\Delta_i})$ is the corresponding line bundle. The lattice points of $\check{\Delta}_i$ give a basis of torus invariant sections of $\mathcal{O}(E_{\Delta_i})$, and we denote by $\sigma_{i,{gen}}$ a generic section of $\mathcal{O}(E_{\Delta_i})$. The choice of generic sections determines the Calabi--Yau complete intersection that is the common vanishing locus of these generic sections,
\[
X_\Delta := \bigcap_{i=1}^k \{ \sigma_{i,{gen}} = 0 \} \subset T_\Delta.
\]
A parallel construction yields the mirror candidate, denoted by $X_{\check{\Delta}}$. Throughout the article, we use the notation $\check{(-)}$ to indicate the mirror counterpart. The main result of \cite{batyrev1996mirror} is the \emph{stringy Hodge number duality}, which asserts that the stringy Hodge numbers of the mirror pair are related by
\[
h^{p,q}_{\mathrm{st}}(X_\Delta) = h^{d-p,q}_{\mathrm{st}}(X_{\check{\Delta}}),
\]
where $d = \dim X_\Delta = \dim X_{\check{\Delta}}$. 
In particular, since we assume that $T_\Delta$ and $T_{\check{\Delta}}$ are MPCS resolutions, the stringy Hodge numbers are the same as the usual Hodge numbers. 

Building on the same initial data, Hosono, Lee, Lian, and Yau introduced a new family of \emph{singular} Calabi–Yau varieties \cite{Hosono2024doubleCY}. Alongside a generic section $\sigma_{i,{gen}}$ of each $\mathcal{O}(E_{\Delta_i})$, there is a canonical toric section, denoted by $\sigma_{i,{tor}}$. They consider the product
\[
\sigma_i := \sigma_{i,{gen}} \cdot \sigma_{i,{tor}},
\]
which becomes a section of $\mathcal{O}(2E_{\Delta_i})$. The collection of $\sigma_i$'s determines a double cover $\widetilde{T}_\Delta \to T_\Delta$,
branched along the divisor $B_\Delta=\{\sigma:= \prod_{i=1}^k \sigma_i = 0\}$. The resulting space $\widetilde{T}_\Delta$ is a singular Calabi–Yau variety with at worst orbifold singularities (Lemma \ref{l:CY}). A mirror candidate $\widetilde{T}_{\check{\Delta}}$ is obtained in the same way from the dual data. In \textit{loc.cit.}, the following conjecture is implicit and the case when $d=3$ is verified.

\begin{conjecture}[HLLY conjecture]\label{conj:intro-HLLY}
    A pair of $d$-dimensional singular Calabi--Yau varieties $(\widetilde{T}_\Delta, \widetilde{T}_{\check{\Delta}})$ satisfies the Hodge number duality: For $p,q \in \mathbb{Z}$,
    \[
    h^{p,q}(\widetilde{T}_\Delta)=h^{d-p,q}(\widetilde{T}_{\check{\Delta}}).
    \]
\end{conjecture}

We prove Conjecture \ref{conj:intro-HLLY}. To do this, we introduce other pairs of singular Calabi--Yau varieties. As introduced in Section \ref{s:intro doublecover}, consider a $(\mathbb{Z}/2
)^k$-Galois cover:
\[
\widehat{T}_\Delta := \bigcap_{i=1}^k \left\{ y_i^2 = \sigma_i \right\}\subset \mathrm{Tot} \left( \bigoplus_{i=1}^k \mathcal{O}(E_{\Delta_i}) \right)
\]
where $y_i$ denotes the fiber coordinate in the line bundle $\mathcal{O}(E_{\Delta_i})$. The variety $\widehat{T}_\Delta$ is again a singular Calabi–Yau variety (Lemma \ref{l:CY}). Our second main theorem is to establish Hodge number duality for the mirror pair $(\widehat{T}_\Delta, \widehat{T}_{\check{\Delta}})$ first, and use it to verify Conjecture \ref{conj:intro-HLLY}.

\begin{theorem}\label{t:intro mirror}
    Let the notation be as above. Then the following identities of Hodge numbers hold:
    \begin{enumerate}
        \item (Theorem \ref{t:toricmirror}) $h^{p,q}(\widehat{T}_\Delta)=h^{d-p,q} (\widehat{T}_{\check{\Delta}})$ for $p,q \in \mathbb{Z}$.
        \item (Theorem \ref{t:HLLY}) $h^{p,q}(\widetilde{T}_\Delta)=h^{d-p,q}(\widetilde{T}_{\check{\Delta}})$ for $p,q \in \mathbb{Z}$.
    \end{enumerate} 
\end{theorem}

The main tool in the proof is our previous work on Hodge number duality for stacky Clarke mirror pairs of Landau–Ginzburg models~\cite{HL2025}. As described in the previous section, the Hodge numbers of branched covers can be computed from certain Landau–Ginzburg models (see Theorems~\ref{t:intro hyps} and~\ref{t:intro com.int}). Our strategy is therefore to verify Hodge number duality for the corresponding Landau–Ginzburg models. For further background, we refer the reader to Section~\ref{s:toric}.

Among the various classes of singular mirror pairs that can be reinterpreted through the lens of Clarke duality—and thereby lead to Hodge number dualities for singular Calabi–Yau varieties—we highlight one particular example: the so-called \emph{toric extremal transition}.

For instance, there is a classical notion of a conifold transition in the mathematics and physics literature where the setup is as follows. Let $X$ be a Calabi--Yau manifold so that $X$ degenerates to a Calabi--Yau variety  $X'$ with a number of isolated $\mathrm{A}_1$ (``conifold'') singularities. If $X'$ admits a crepant resolution, ${X}''$, then we say that $X$ and ${X}''$ are related by conifold transition. 

It is expected that, in certain circumstances, the mirrors, $\check{X}$ and $\check{X}''$ to $X$ and ${X}''$ respectively, are also related by a conifold transition, where $\check{X}''$ degenerates to a conifold Calabi--Yau variety $\check{X}'$, and $\check{X}$ is a crepant resolution of $\check{X}'$. In other words, we have the following diagrams

 \[
    \begin{tikzcd}
      & {X}''\ar[d,dashed]  \\
      X \ar[r,rightsquigarrow] & X' 
    \end{tikzcd} \hspace{2cm} 
    \begin{tikzcd}
      & \check{X} \ar[d,rightsquigarrow]  \\
      \check{X}'' \ar[r,dashed] & \check{X}'
    \end{tikzcd}
    \]
This construction was generalized to the notion of an {\em extremal transition} in the work of Morrison in the late 1990s~\cite{MR1673108}. Furthermore, Morrison provides a combinatorial description of a class of extremal transitions arising from inclusions between reflexive integral polytopes, $\Delta_{\mathrm{II}} \subseteq \check{\Delta}_{\mathrm{I}}$, which correspond to a dual inclusion $\Delta_{\mathrm{I}} \subseteq \check{\Delta}_{\mathrm{II}}$ of polar-dual polytopes:
 \[
    \begin{tikzcd}
      & X_{\Delta_{\mathrm{II}}} \ar[d]  \\
      X_{\check{\Delta}_{\mathrm{I}}} \ar[r,rightsquigarrow] & X'_{\Delta_{\mathrm{II}}} 
    \end{tikzcd} \hspace{2cm} 
    \begin{tikzcd}
      & X_{\check{\Delta}_{\mathrm{II}}} \ar[d,rightsquigarrow]  \\
      X_{\Delta_{\mathrm{I}}} \ar[r] & X'_{\Delta_{\mathrm{I}}}
    \end{tikzcd}
    \]
It is a consequence of \cite{batyrev1996mirror} that
\[
h^{p,q}(X_{\Delta_{\mathrm{II}}}) = h^{d-p,q}(X_{\check{\Delta}_{\mathrm{II}}}),\qquad h^{p,q}(X_{\Delta_{\mathrm{I}}}) = h^{d-p,q}(X_{\check{\Delta}_{\mathrm{I}}}).
\]
We prove the Hodge number duality between $X'_{\Delta_{\mathrm{I}}}$ and $X'_{\Delta_{\mathrm{II}}}$ (see Theorem \ref{t:transition}): For $p,q \in \mathbb{Z}$,
\begin{equation}\label{e:intro-transition}
    h^{p,q}(X_{\Delta_{\mathrm{I}}}') = h^{d-p,q}(X_{\Delta_{\mathrm{II}}}')
\end{equation}
Note that both $X_{\Delta_{\mathrm{I}}}'$ and $X_{\Delta_{\mathrm{II}}}'$ are singular in general, and we use the notation $h^{p,q}$ to denote the dimension of the $p-$th Hodge graded piece of Deligne's canonical mixed Hodge structures. In particular, if we take the following polytopes with a choice of projective unimodular triangulations, 
\[
    \begin{aligned}
        \Delta_{\mathrm{II}}&:=\mathrm{Conv}(\check{\Delta}_i\times e_i \cup 0\times -e_i|i=1, \cdots, k) \subset M_\mathbb{R}\times \mathbb{R}^k, \\
        \check{\Delta}_{\mathrm{I}}&:=\mathrm{Conv}(2\check{\Delta}_i\times e_i \cup 0\times -e_i|i=1, \cdots, k) \subset M_\mathbb{R}\times \mathbb{R}^k,
    \end{aligned}
\]
then the Hodge number duality \eqref{e:intro-transition} becomes the first item in Theorem \ref{t:intro mirror}.


For the second item in Theorem \ref{t:intro mirror}, we make use of a certain yoga involving several distinct stacky Clarke mirror pairs including pairs related to those that appear in Theorem \ref{t:intro hyps}. Moreover, the first part of Theorem \ref{t:intro mirror} plays a key role in the proof of the second part. We refer the reader to Section \ref{s:HLLY} for further details.

\begin{remark}
    Unlike the proof in \cite{Hosono2024doubleCY}, our proof of the HLLY conjecture does not use  ampleness of the anti-canonical divisor of MPCS resolution of toric varieties, which forces the Hodge numbers $h^{p,q}$ to vanish for $p+q \neq d$ due to \cite{Esnault1992vanishing}. 
    
\end{remark}

\subsection*{Acknowledgements}
AH was supported by a Simons Foundation Collaboration Grant for Mathematicians. SL was supported by the Institute for Basic Science (IBS-R003-D1).


\section{Cyclic covers and Landau--Ginzburg models}\label{s:coverLG}

This section develops technical results relating the irregular Hodge numbers of certain Landau--Ginzburg models to the irregular Hodge numbers of a class of $(\mathbb{Z}/2)^k$-covers of a variety $X$. The reader interested mostly in mirror symmetry applications might find it convenient to skip Section \ref{s : proof of t: single direct sum decomposition} upon first reading.

\subsection{Landau--Ginzburg models and the irregular Hodge filtration}

In this article, a Landau--Ginzburg (LG) model will denote a pair $(Y,w)$ where $Y$ is a quasiprojective variety and $w$ is a regular function on $Y$. A {\em compactified LG model} will be a triple $(X,D,f)$ where $X$ is a smooth projective variety, and $D$ is a simple normal crossings (snc) hypersurface in $X$, and the map $f$ is a rational function on $X$ whose pole divisor is contained in $D$. We say a compactified LG model $(X,D,f)$ \emph{nondegenerate} if the vanishing locus $Z(f)$ is irreducible and $Z(f) \cup D$ is snc. Let $P$ denote the pole divior of $f$. For simplicity, we also assume that $Z(f) \cap P = \emptyset.$  Note that given any LG model for which $Y$ and $\{w=0\}$ are smooth,  one can always find a nondegenerate compactified LG model $(X,D,f)$ for which $Y = X\setminus D$ and $w = f|_Y$ by Hironaka’s theorem.

\subsubsection{Twisted de Rham cohomology} Given a nondegenerate compactified LG model $(X,D,f)$, we have a twisted de Rham complex $(\Omega_X^\bullet(\log D)(*P), \mathrm{d} + \mathrm{d}f )$. We denote its cohomology by $H^*(X\setminus D,f)$. Given a snc hypersurface $E$ for which $E\cup D \cup Z(f)$ is also snc, we define the local twisted cohomology of $E$, denoted by $H^*_E(X\setminus D,f)$, to be the hypercohomology of the complex,
\[
\mathrm{DR}_{E}(X\setminus D,f):=\left(\dfrac{\Omega^\bullet_X(\log D + E)(*P)}{\Omega^\bullet_X(\log D)(*P)}, \mathrm{d} + \mathrm{d}f \right). 
\]
Suppose $E$ is irreducible. Then there is a residue morphism
\[
\mathrm{Res}_E:  \mathrm{DR}(X\setminus (D\cup E),f) \longrightarrow \mathrm{DR}(E \setminus (D\cap E), f|_E)[1].
\]
This morphism is induced by the usual residue morphism (see e.g. \cite[\S 4.2]{peters2008mixed}) on sheaves of $p$-forms. The kernel of $\mathrm{Res}_E$ is $\mathrm{DR}(X\setminus D,f)$ so that we have an isomorphism between $(\Omega_X^\bullet(\log D,E)(*P),\mathrm{d} + \mathrm{d}f)$ and $(\Omega^{\bullet -1}_E(\log E\cap D)(*P\cap D), \mathrm{d} +\mathrm{d}f|_E)$ when $E$ is irreducible. 

More generally, if $E_1,\dots, E_r$ are snc hypersurfaces (not necessarily irreducible) such that $E_1\cup \dots \cup E_r$ is also snc, then the twisted de Rham complex, localized  to $E^{[r]}:=E_1\cap \dots \cap E_r$ is defined by
\[
\mathrm{DR}_{E^{[r]}}(X\setminus D,f):=\left(\Omega^\bullet_X(\log D, E^{[r]}):=  \dfrac{\Omega_X^\bullet(\log D + \sum_{i} E_i)(*P)}{\sum_i \Omega^\bullet(\log D + \sum_{j \neq i} E_j)(*P)}, \mathrm{d} + \mathrm{d}f\right).
\]
If $r=0$, then we take $E^{[r]}$ to mean all of $X$, and the local twisted cohomology coincides with the twisted cohomology of $X$. The hypercohomology of this complex is denoted by $H^*_{E^{[r]}}(X\setminus D,f)$. It is an exercise in local coordinates to check that there is a short exact sequence of complexes,
\begin{equation}\label{e : local cohomology les 1 - sheaf}
0 \longrightarrow \mathrm{DR}_{E^{[r-1]}}(X\setminus D,f) \longrightarrow \mathrm{DR}_{E^{[r-1]}}(X\setminus (D\cup E_r),f) \xlongrightarrow{\mathrm{Res}_{E_r}} \mathrm{DR}_{E^{[r]}}(X\setminus D,f) \longrightarrow 0
\end{equation}
where $E^{[r-1]} = E_1\cap \dots \cap E_{r-1}$ and $\mathrm{Res}_{E_r}$ is the residue morphism. If we write $E_r = R_1\cup \dots \cup R_m$ for smooth, irreducible hypersurfaces $R_i$ and define $R^I= \bigcap_{i\in I} R_i$, $B^I = R^I \cap D \cap \bigcup_{j \notin I} R_{j}$, there is a residue resolution
\begin{equation}\label{e : residue resolution}
\begin{array}{ll}
0 \longrightarrow \Omega_X^\bullet(\log D, E^{[r]})(*P) & \longrightarrow \bigoplus_{|I| =1}\Omega_{R^I}^{\bullet-1}(\log B_I, E^{[r]}\cap R^I)(*P) \\ & \longrightarrow \bigoplus_{|I| = 2}\Omega_{R^I}^{\bullet-2}(\log B_I, E^{[r]}\cap R^I)(*P) \longrightarrow \dots 
\end{array}
\end{equation}
The morphisms in \eqref{e : residue resolution} are alternating sums of the residue morphisms whose precise definition is standard (see e.g. \cite[\S 8.2]{HL2025}), but we will suppress here because it only plays a minor role in our computations.

\subsubsection{Irregular Hodge filtration} 
There is a filtration on the twisted de Rham complex, called the irregular Hodge filtration, which was first defined by Yu \cite{yu2014irregular} as follows. Let $\lambda \in \mathbb{Q}$.
\[
F_{\mathrm{irr}}^\lambda \mathcal{O}_X(*P) = \begin{cases} 0 & \text{ if } \lambda < 0 \\ 
\mathcal{O}_X(\lfloor \lambda P\rfloor) & \text{ if } \lambda \geq 0 \end{cases},\quad F_{\mathrm{irr}}^\lambda\Omega_X^p(\log D)(*P) = \Omega_X^p(\log D)\otimes F_{\mathrm{irr}}^{p-\lambda}\mathcal{O}_X(*P).
\]
Exactness of the tensor product with the line bundle $\mathcal{O}_X(\lfloor  \lambda P\rfloor)$ implies that this filtration also induces an irregular Hodge filtration $\mathrm{DR}_{E^{[r]}}(X\setminus D,f)$. 
Note that the irregular Hodge filtration is a $\mathbb{Q}$-filtration. The induced filtration on $H^p_{E^{[r]}}(X\setminus D,f)$ is defined as usual by letting 
\begin{equation}\label{e : definition of the irregular filtration on cohomology}
F_\mathrm{irr}^\lambda H_{E^{[r]}}^p(\Omega^\bullet_X(\log D)(*P)) = \mathrm{im}\left(H_{E^{[r]}}^p(F^\lambda_\mathrm{irr}\Omega^\bullet_X(\log D)(*P)) \rightarrow H_{E^{[r]}}^p(\Omega^\bullet_X(\log D)(*P) ) \right).
\end{equation}
The irregular filtration is exhaustive. In fact, $F^0_\mathrm{irr}H^p_{E^{[r]}}(\Omega_X^\bullet(\log D)(*P)) = H^p_{E^{[r]}}(\Omega_X^\bullet(\log D)(*P))$ \cite{yu2014irregular} and $F_\mathrm{irr}^\lambda \Omega_X^\bullet(\log D)(*P) = 0$ if $\lambda > \dim X$. Furthermore, if 
\begin{align*}
F^{>\lambda}_\mathrm{irr}H^p_{E^{[r]}}(X\setminus D,f) & = \bigcup_{\lambda' >\lambda} F^{\lambda'}_\mathrm{irr}H^p_{E^{[r]}}(X\setminus D,f), \\  \gr^\lambda_{F_\mathrm{irr}}H^p_{E^{[r]}}(\Omega_X^\bullet(\log D)(*P))& = \dfrac{{F}^{\lambda}_\mathrm{irr}H^p_{E^{[r]}}(\Omega_X^\bullet(\log D)(*P))}{F_\mathrm{irr}^{>\lambda}H^p_{E^{[r]}}(\Omega_X^\bullet(\log D)(*P))},
\end{align*}
then $\gr_{F_\mathrm{irr}}^\lambda H^p_{E^{[r]}}(\Omega_X^\bullet(\log D)(*P)) = 0$ for all but finitely many $\lambda \in \mathbb{Q}_{\geq 0}$. By results of Yu \cite{yu2014irregular}, the twisted cohomology and irregular Hodge filtration of $(X,D,f)$ only depend on $Y = X\setminus D$ and $w = f|_Y$.
\begin{example}\label{ex : irregular = regular}
    If $f = 0$ and $E^{[r]}=\emptyset$, then it is not hard to see that the filtration $F_\mathrm{irr}^\bullet$ is simply the stupid filtration on the log de Rham complex $\Omega^\bullet_X(\log D)$, in which case ${F}^\bullet_\mathrm{irr}$ agrees with the usual Hodge filtration on the cohomology of the noncompact variety $X\setminus D$.
\end{example}

\begin{example}\label{e : t^n LG model}
    The affine LG model $(\mathbb{A}^1,t^k)$ has twisted cohomology computed using the complex
    \[
    \mathbb{C}[t] \xrightarrow{\mathrm{d} + \mathrm{d}(t^k)} \mathbb{C}[t]\cdot \mathrm{d}t
    \]
    It is an exercise to check that 
    \[F^{(k-a)/k}_\mathrm{irr}H^1(\mathbb{A}^1,t^k)=\begin{cases} 0 & a <1 \\
        \mathrm{span}_\mathbb{C}\{\mathrm{d}t, t\mathrm{d}t,\dots, t^{a-1}\mathrm{d}t\} & 1 \leq a \leq k-1 \\
        \mathrm{span}_\mathbb{C}\{\mathrm{d}t, t\mathrm{d}t,\dots, t^{k-2}\mathrm{d}t\} & a=k
    \end{cases}
    \]
    and that $H^0(\mathbb{A}^1,t^k) \cong 0$ if $k \geq 1$. Therefore, $\dim \gr_F^{\lambda}H^1(\mathbb{A}^1,t^k) = 1$ if $\lambda =1/k,\dots, (k-1)/k$ and $0$ otherwise.
\end{example}

An important result in \cite{esnault20171} is that the morphism in \eqref{e : definition of the irregular filtration on cohomology} is injective for all $\lambda$ and as a consequence,
\begin{equation}\label{e : ESY}
\gr_{F_\mathrm{irr}}^\lambda H_{E^{[r]}}^{p}(\Omega^\bullet_X(\log D)(*P)) \cong H_{E^{[r]}}^{p}(\gr_{F_\mathrm{irr}}^\lambda \Omega^\bullet_X(\log D)(*P)).
\end{equation}
There are filtered long exact sequences whenever $E$ is an snc hypersurface in $X$, coming from \eqref{e : local cohomology les 1 - sheaf},
\begin{equation}\label{e : local cohomology les}
\dots \longrightarrow H_{E^{[r-1]}}^p(X\setminus D,f) \longrightarrow H^p_{E^{[r-1]}}(X\setminus (D\cup E_r),f) \longrightarrow H_{E^{[r]}}^p(X\setminus D,f) \longrightarrow \dots 
\end{equation}
The following fact can be deduced directly from the results of \cite{esnault20171} (in particular \eqref{e : ESY}).
\begin{proposition}\label{p : strict morphism of filtered complexes}
    The morphisms in the long exact sequence \eqref{e : local cohomology les} are strict with respect to $F^\bullet_\mathrm{irr}$. In other words, for each $\lambda$ there are long exact sequences 
    \[
\dots \longrightarrow \gr^\lambda_{F_\mathrm{irr}}H_{E^{[r-1]}}^p(X\setminus D,f) \longrightarrow \gr^\lambda_{F_\mathrm{irr}}H_{E^{[r-1]}}^p(X\setminus (D\cup E_r),f) \longrightarrow \gr^\lambda_{F_\mathrm{irr}}H_{E^{[r]}}^p(X\setminus D,f) \longrightarrow \dots 
\]
\end{proposition}
\begin{remark}[On irregular filtrations]
    We will often suppress discussion of the irregular Hodge filtration in this section, as all morphisms of interest will either be of the type described in Proposition \ref{p : strict morphism of filtered complexes}, and hence will strictly preserve $F^\bullet_\mathrm{irr}$, or they will be filtered morphisms which are isomorphisms of the underlying vector spaces, hence they are filtered isomorphisms. 
\end{remark}

There is also a K\"unneth formula for the irregular Hodge filtration, which appears in work of Chen and Yu \cite{chen2018kunneth} and Sabbah and Yu \cite{Sabbah2023cyclic}, which we state in slightly greater generality than \cite[Theorem 1]{chen2018kunneth} and significantly less generality than \cite[Corollary 3.34]{Sabbah2023cyclic}.
\begin{theorem}\label{t : Kuenneth theorem}
    Suppose $(X_1,D_1,f_1)$ and $(X_2,D_2,f_2)$ are nondegenerate compactified LG models and that $E$ is a subvariety of $X_1$. There is a filtered isomorphism
    \[
    \bigoplus_{p+q=v} H_E^p(X_1\setminus D_1,f_1) \otimes H^q(X_2\setminus D_2,f_2)\longrightarrow H_{E\times X_2}^v((X_1\setminus D_1)\times(X_2 \setminus D_2),f_1\boxplus f_2).
    \]
    Here, $f_1 \boxplus f_2 := \pi_1^*f_1 + \pi_2^*f_2$ where $\pi_i : X_1\times X_2 \rightarrow X_i$ is the projection map. 
\end{theorem}

The Kontsevich sheaves of $(X,D,f)$ are defined to be:
    \[
    \Omega_f^p : = \{ \omega \in \Omega_X^p(\log D) \mid \mathrm{d}f \wedge \omega \in \Omega_{X}^{p+1}(\log D)\}.    
    \]
    Equipping $\Omega_f^\bullet$ with the differential $\mathrm{d} + \mathrm{d}f$, we obtain the Kontsevich complex, and equipping $\Omega_f^\bullet$ with the differential $\mathrm{d}$ we obtain the Kontsevich--de Rham complex.
    
    There is an inclusion of complexes 
    \begin{equation}\label{e: inlcusion kontsevich}
    (\Omega_f^\bullet, \mathrm{d} +\mathrm{d}f ,\sigma^\bullet) \hookrightarrow (\mathrm{DR}(X\setminus D,f),F^{\mathrm{Yu},\bullet}_0).
    \end{equation}
    Here $\sigma^\bullet$ denotes the stupid filtration and where, for any $\alpha \in \mathbb{Q}$, $F^{\mathrm{Yu},p}_\alpha$ is the $\mathbb{Z}$-filtration defined so that $F^{\mathrm{Yu},p}_\alpha = F_\mathrm{irr}^{p+\alpha}$ for $p\in \mathbb{Z}$. The inclusion \eqref{e: inlcusion kontsevich} a quasi-isomorphism of filtered complexes  (\cite[Corollary 1.4.5]{esnault20171}). In particular, if $\mathrm{gr}_{F_\mathrm{irr}}^\lambda H^*(X\setminus D, f)=0$ for $\lambda \in \mathbb{Q} \setminus \mathbb{Z}$ (as will be assumed in Proposition \ref{p : yu filtration versus hodge filtration} below) the irregular filtration on $H^*(X\setminus D,f)$ can be identified with the filtration induced by $\sigma^\bullet$ on $\mathbb{H}^*(\Omega_f^\bullet, \mathrm{d} + \mathrm{d}f)$. 
    
    According to \cite[Theorem 1.3.2]{esnault20171}, we have equalities,
    \[
    \dim \mathbb{H}^n(X,(\Omega_f^\bullet, \mathrm{d} +\mathrm{d}f)) = \dim \mathbb{H}^n(X,(\Omega_f^\bullet, \mathrm{d})) = \sum_{p+q = n} \dim H^p(X,\Omega_f^q).
    \]
It is known, for instance by work of Hien \cite{Hien2009period}, that $\dim H^*(X\setminus D,f) = \dim H^*(X\setminus D,f^{-1}(t);\mathbb{C})$ for a generic value $t$. The next result is presumably well known but does not seem to be stated anywhere in the literature. It says that if $t$ is a generic point in $\mathbb{A}^1$, the usual Hodge numbers of the pair $(X\setminus D,f^{-1}(t))$ can be computed using the irregular Hodge filtration under certain conditions.
\begin{proposition}\label{p : yu filtration versus hodge filtration}
    If $(X,D,f)$ is a nondegenerate compactified LG model for which $\mathrm{gr}_{F_{\mathrm{irr}}}^\lambda H^{*}(X\setminus D,f) = 0$ for $\lambda \in \mathbb{Q} \setminus \mathbb{Z}$, then for any smooth, generic fiber $f^{-1}(t)$ of $f$ and for any $p,q\in \mathbb{Z},$
    \[
    \dim \gr_{F_0^\mathrm{Yu}}^p H^{p+q}(X\setminus D,f) = \dim \gr_{F}^p H^{p+q}(X\setminus D,f^{-1}(t);\mathbb{C}).
    \]
\end{proposition}
\begin{proof}
 The proof is a slight generalization of the proofs of \cite[Claim 2.22]{katzarkov2017bogomolov} and \cite[Theorem 3.1]{harder}. 

 Let $\mathfrak{X} = X \times \mathbb{D}, \mathfrak{D} = D \times \mathbb{D}$ where $\mathbb{D}$ denotes a small disc near $\infty \in \mathbb{P}^1$, and let $\mathfrak{P} = \{ (x,\varepsilon) \in \mathfrak{X} \mid f(x) = \varepsilon\}$. By construction, $\mathfrak{P}\cup \mathfrak{D}$ is a normal crossings divisor in $\mathfrak{X}$.  There are complexes of sheaves $(\Omega_{\mathfrak{X}/\mathbb{D}}^\bullet(\log \mathfrak{D}), \mathrm{d})$ and $(\Omega^\bullet_{\mathfrak{P}/\mathbb{D}}(\log \mathfrak{D}\cap \mathfrak{P}),\mathrm{d})$ and a restriction morphism, so, following notation of \cite{katzarkov2017bogomolov}, we may define
    \begin{equation}\label{e: stalk 1}
    \Omega_{\mathfrak{X}/\mathbb{D}}^\bullet(\log \mathfrak{D},\mathrm{rel}\,\mathfrak{P}):= \ker\left[\Omega_{\mathfrak{X}/\mathbb{D}}^\bullet(\log \mathfrak{D}) \longrightarrow \Omega^\bullet_{\mathfrak{P}/\mathbb{D}}(\log \mathfrak{D}\cap \mathfrak{P})\right].
    \end{equation}
    We now check that $\Omega_{\mathfrak{X}/\mathbb{D}}^\bullet(\log \mathfrak{D},\mathrm{rel}\, \mathfrak{P})|_{X_\varepsilon}$ is isomorphic to $\Omega_{X}^\bullet(\log D,\mathrm{rel}\, {f}^{-1}(\varepsilon))$ if $\varepsilon \neq \infty$ and $\Omega_f^\bullet$ if $\varepsilon = \infty$. We write $f$ locally as $1/(x^{e_1}_1\dots x_k^{e_k})$ for positive exponents $e_1,\dots, e_k$. As a general fact, if $V$ is smooth in $X$, $V\cup D$ has normal crossings, and $g$ is a function in local coordinates for which $V= \{g = t\}$ for some value $t$, we may write
    \begin{equation}\label{e : stalk 2}
    \Omega_X^p(\log D,\mathrm{rel}\, V)_{x} = d\log g \wedge \left( \bigwedge^{p-1} W\right) \oplus (g-t)\cdot \left(\bigwedge^p W\right)
    \end{equation}
    where $W$ is the $\mathcal{O}_{X,x}$ span of a collection of forms so that $\mathcal{O}_{X,x}\cdot d\log g \oplus W = \Omega_{X}^1(\log D)_x$. We may write $f^{-1}(\varepsilon)$ locally as $\{x_1^{e_1}\dots x_k^{e_k} - t = 0\}$ where $t = 1/\varepsilon$. As in \cite[Proof of Claim 2.22]{katzarkov2017bogomolov}, we may write 
    \begin{equation}\label{e : stalk 3}
    \Omega_{\mathfrak{X}/\mathbb{D}}^p(\log \mathfrak{D},\mathrm{rel}\, \mathfrak{P})_x = (x_1^{e_1}\dots x_k^{e_k} - t) \left( \bigwedge^p \mathcal{W} \right) \oplus d\log(x_1^{e_1}\dots x_k^{e_k})\wedge \left(\bigwedge^{p-1} \mathcal{W}\right)
    \end{equation}
    where $\mathcal{W}$ is a free $\mathcal{O}_{\mathfrak{X},x}$ submodule of $\Omega_{\mathfrak{X}/\mathbb{D}}^1(\log \mathfrak{D})$ so that $\mathcal{O}_{\mathfrak{X},x} \cdot (d\log (x_1^{e_1}\dots x_k^{e_k})) \oplus\mathcal{W}   = \Omega_{\mathfrak{X}/\mathbb{D}}^1(\log \mathfrak{D})_x$. Precisely, we may let 
    \[
    \mathcal{W} = \mathrm{span}_{\mathcal{O}_{\mathfrak{X},x}}\{ \mathrm{d}\log x_2,\dots, \mathrm{d}\log x_k, \mathrm{d}\log y_1,\dots, \mathrm{d}\log y_m, \mathrm{d}z_1,\dots, \mathrm{d} z_n\}
    \]
    where, locally, $\mathrm{supp}(P) = \{x_1 \dots x_k = 0\}, D = \{x_1\dots x_ky_1\dots y_m= 0\}$ and $z_1,\dots, z_n$ are the remaining coordinates. Note that we ignore $\mathrm{d}\log t$ because we are considering relative differential forms. Now by comparing \eqref{e : stalk 2} and \eqref{e : stalk 3} we see that $\Omega_{\mathfrak{X}/\mathbb{D}}^\bullet(\log \mathfrak{D},\mathrm{rel}\, \mathfrak{P})|_{X_\varepsilon}$ is isomorphic to $\Omega_{X}^\bullet(\log D,\mathrm{rel}\, {f}^{-1}(\varepsilon))$ if $\varepsilon \neq \infty$. The local form for $\Omega_f^\bullet$ given in \cite[Proof of Proposition 1]{chen2018kunneth} is identical to $\Omega_{\mathfrak{X}/\mathbb{D}}^\bullet(\log \mathfrak{D},\mathrm{rel}\, \mathfrak{P})|_{X_0}$.

     The relative de Rham complex $\Omega_{X}^\bullet(\log D,\mathrm{rel}\,f^{-1}(\varepsilon))$ equipped with its stupid filtration $\sigma^\bullet$ underlies a mixed Hodge structure, therefore the filtration induced by $\sigma^\bullet$ degenerates at the $E_1$ term and 
    \[
   \dim H^n(X\setminus D,f^{-1}(\varepsilon);\mathbb{C}) =  \dim \mathbb{H}^n(X_\varepsilon, \Omega_{X}^\bullet(\log D,\mathrm{rel}\,f^{-1}(\varepsilon))) = \sum_{p+q = n} H^p( \Omega_{X}^q(\log D,\mathrm{rel}\,f^{-1}(\varepsilon)))
    \]
    and by \cite[Theorem 1.3.2]{esnault20171} and the results of Hien \cite{Hien2009period} mentioned above,
    \[
    \dim H^n(X\setminus D,f^{-1}(\varepsilon);\mathbb{C}) =  \dim \mathbb{H}^n(X, (\Omega_{f}^\bullet,\mathrm{d})) = \sum_{p+q = n} H^p(X, \Omega_{f}^q)
    \]
    Now we invoke Grauert's base change theorem (e.g. \cite[Theorem 10.30]{peters2008mixed}) to see that rank of $H^p(X_\varepsilon,\Omega^q_{\mathfrak{X}/\mathbb{D}}(\log \mathfrak{D}, \mathrm{rel}\, \mathfrak{P})|_{X_\varepsilon})$ is upper semicontinuous on $\mathbb{D}$ to conclude that $\dim H^p(X,\Omega_f^q) = \dim H^p(X,\Omega_X^q(\log D,\mathrm{rel}\, f^{-1}(\varepsilon)))$ for all $\varepsilon$ in $\mathbb{D}$ as required.
\end{proof}
\begin{corollary}
     Suppose $(X,D,f)$ is a nondegenerate compactified LG model for which $\mathrm{gr}_{F_{\mathrm{irr}}}^\lambda H^{*}(X\setminus D,f) = 0$ when $\lambda \in \mathbb{Q} \setminus \mathbb{Z}$. For $p,q \in \mathbb{Z}$ and a generic value of $t$,
    \[
    \dim \gr_{F_\mathrm{irr}}^pH^{p+q}(X\setminus D,f) = \dim \gr_F^pH^{p+q}(X\setminus D, f^{-1}(t);\mathbb{C}).
    \]
    
\end{corollary}
\begin{proof}
    If the irregular Hodge numbers of $(X,D,f)$ vanish away from integer indices, the filtration induced by $F_\mathrm{irr}$ on twisted cohomology is the same as the filtration induced by $F_0^{\mathrm{Yu}}$. The result then follows from Proposition \ref{p : yu filtration versus hodge filtration}.
\end{proof}

We present a particular case, which we will discuss in more detail in the next section. Suppose $\sigma$ is a global section of a vector bundle $\pi: \mathcal{E}\rightarrow X$ and let $Z_\sigma$ denote the vanishing locus of $\sigma$ in $X$. There is an induced regular map $g=g_\sigma : \mathbb{V}:=\mathrm{Tot}(\mathcal{E}^\vee) \rightarrow \mathbb{A}^1$, and we may view $(\mathbb{V},g)$ as a LG model. Completely explicitly, if we trivialize $\mathbb{V}$ on an open set $U$, $\varphi : \mathbb{C}^k \times U \rightarrow \mathbb{V}|_U $ and we may write  
\begin{equation}\label{e : local representation for potential}
\varphi^*g = t_1f_1(\underline{x}) + \dots + t_k f_k(\underline{x})
\end{equation}
where $t_1,\dots, t_k$ are coordinates on $\mathbb{C}^k$ and $(f_1, \dots, f_k)$ is the pullback of $\sigma$. Results closely related to Proposition \ref{p : well known} below appear several places (e.g. \cite{dimca2000dwork,yu2014irregular,Fresan2022}) but we are not aware of a statement in the literature where the precise form used below can be found.
\begin{proposition}\label{p : well known}
    If $\gr_{F_\mathrm{irr}}^\lambda H^*(\mathbb{V},g) = 0$ for all $\lambda \in \mathbb{Q}\setminus \mathbb{Z}$, then 
    \[
    \dim \gr_{F_\mathrm{irr}}^p H^{p+q}(\mathbb{V},g) = \dim \gr_F^pH^{p+q}_{Z_\sigma}( X)
    \]
    for $p,q \in \mathbb{Z}$.
\end{proposition}
\begin{proof}
    Under the assumptions of the proposition, we see that 
    \[
    \dim \gr_{F_\mathrm{irr}}^p H^{p+q}(\mathbb{V},g) = \dim \gr_F^pH^{p+q}( \mathbb{V}, g^{-1}(\varepsilon))
    \]
    for sufficiently generic $\varepsilon$. From \eqref{e : local representation for potential} we can see that the projection map induces an isomorphism between  $g^{-1}(\varepsilon)$ and a $\mathbb{C}^{k-1}$ bundle over $X \setminus Z_\sigma$. There is a commutative diagram in cohomology whose vertical maps are isomorphisms and the horizontal maps are restrictions,
    \[
    \begin{tikzcd}
        H^*(X) \ar[r] \ar[d,"\pi^*"]& H^*(X\setminus Z_\sigma) \ar[d, "\pi|_{g^{-1}(\varepsilon)}"] \\
        H^*(\mathbb{V}) \ar[r] & H^*(g^{-1}(\varepsilon))
    \end{tikzcd}
    \]
    Completing the rows of this diagram to long exact sequences, we obtain an isomorphism of mixed Hodge structures between $H^*(\mathbb{V},g^{-1}(\varepsilon))$ and $H^*(X,X\setminus Z_\sigma) = H^*_{Z_\sigma}(X)$.
\end{proof}

By Poincar\'e--Lefschetz duality (e.g. \cite[Theorem B.29]{peters2008mixed} or \cite[(1.7.1)]{Fujiki1980duality}) there is an isomorphism of mixed Hodge structures between $H^*_{Z_\sigma}(X)$ and $H^{*-2c}(Z_\sigma)$ where $c = \mathrm{codim}_\mathbb{C}\, Z_\sigma$. Therefore, Proposition \ref{p : well known} identifies the twisted cohomology of $(\mathbb{V},g)$ with the usual cohomology of $Z_\sigma$ after an appropriate Tate twist.

\subsubsection{Affine bundles} The following result is a rather straightforward generalization of known results for the cohomology of projective varieties. Let $\pi:\mathbb{P}\rightarrow X$ denote a $\mathbb{P}^1$ bundle, which is the projectivization of the total space of a line bundle $\mathbb{L} = \mathrm{Tot}(L^{-1})$. Let $\mathbb{X}$ denote the hypersurface in $\mathbb{P}$ so that $\mathbb{L} = \mathbb{P}\setminus \mathbb{X}$. Let $E_1,\dots, E_r$ denote hypersurfaces in $X$ as above and let $\mathbb{E}_i:=\pi^{-1}E$. If $D$ is a hypersurface in $X$, let $\mathbb{D}:=\pi^{-1}D$, and if $f$ is a regular function on $X\setminus D$, let $\mathsf{f} := \pi^*f$.

\begin{proposition}\label{p : Affine bundle cohomology agrees with the base}
    Let notation be as above and suppose $(X,D,f)$ is a nondegenerate compactified LG model. Then $(\mathbb{P},\mathbb{D}\cup \mathbb{X},\mathsf{f})$ is a nondegenerate compactified LG model, $\mathrm{DR}_{\mathbb{E}^{[r]}}(\mathbb{P}\setminus (\mathbb{D} + \mathbb{X}),\mathsf{f})$ is $\pi$-acyclic,  and there is a quasi-isomorphism of complexes, 
    \[
    \mathrm{R}\pi_*\mathrm{DR}_{\mathbb{E}^{[r]}}(\mathbb{P}\setminus  (\mathbb{D}+\mathbb{X}),\mathsf{f}) \longrightarrow \mathrm{DR}_{E^{[r]}}(X \setminus D,f).
    \]
\end{proposition}
\begin{proof}
    The first statement, that the compactification $(\mathbb{P},\mathbb{D} \cup \mathbb{X},\mathsf{f})$ is nondegenerate if $(X,D,f)$ is a nondegenerate compactified LG model, is straightforward so we omit its proof. Let $P_\mathsf{f}$ denote the pole divisor of $\mathsf{f}$. There is a short exact sequence of sheaves for each $q$ (see e.g. \cite[\S 10.4]{peters2008mixed}),
    \begin{equation}\label{e: relative de Rham sequence}
    0 \longrightarrow \pi^*(\Omega_X^q(\log D)(*P)) \longrightarrow \Omega_{\mathbb{P}}^q(\log \mathbb{D}+ \mathbb{X})(*P_\mathsf{f}) \longrightarrow \Omega^1_{\mathbb{P}/X}(\log \mathbb{X}) \otimes \pi^*(\Omega_X^{q-1}(\log D)(*P)) \longrightarrow 0.
    \end{equation}
    Now we may apply $\mathrm{R}\pi_*$ to \eqref{e: relative de Rham sequence} to obtain a long exact sequence of sheaves on $X$. 
    \begin{align*}
    0 &\rightarrow \Omega_X^q(\log D)(*P) \otimes \pi_*\mathcal{O}_\mathbb{P} \rightarrow \pi_*\Omega_{\mathbb{P}}^q(\log \mathbb{D}+ \mathbb{X})(*P_\mathsf{f}) \rightarrow \pi_*\Omega^1_{\mathbb{P}/X}(\log \mathbb{X}) \otimes (\Omega_X^{q-1}(\log D)(*P)) \\ & \rightarrow \Omega_X^q(\log D)(*P) \otimes \mathrm{R}^1\pi_*\mathcal{O}_\mathbb{P}\rightarrow \mathrm{R}^1\pi_*\Omega_{\mathbb{P}}^q(\log \mathbb{D}+ \mathbb{X})(*P_\mathsf{f}) \rightarrow \mathrm{R}^1\pi_*\Omega^1_{\mathbb{P}/X}(\log \mathbb{X}) \otimes \Omega_X^{q-1}(\log D)(*P) \\
    & \rightarrow \dots 
    \end{align*}
    Here we have applied the projection formula several times. Since $\pi :\mathbb{P}\rightarrow X$ is a $\mathbb{P}^1$ bundle, $\mathrm{R}\pi_*\mathcal{O}_{\mathbb{P}} \cong \pi_*\mathcal{O}_\mathbb{P}\cong \mathcal{O}_X$ and $\Omega^1_{\mathbb{P}/X}(\log \mathbb{X}) \cong \mathcal{O}_{\mathbb{P}}(-1)$, which means that $\mathrm{R}\pi_* \Omega^1_{\mathbb{P}/X}(\log \mathbb{X}) \cong 0$.  From this, it follows that $\mathrm{R}\pi_*\Omega_{\mathbb{P}}^q(\log \mathbb{D}+ \mathbb{X})(*P_\mathsf{f}) \cong \Omega_X^q(\log D)(*P)$. A straightforward local computation shows that $\pi_*(\mathrm{d} + \mathrm{d}\mathsf{f}) = \mathrm{d} + \mathrm{d}f$ under the isomorphism above. Therefore, 
    \[
    \mathrm{R}\pi_*(\Omega_\mathbb{P}^\bullet(\log \mathbb{D} + \mathbb{X})(*P_\mathsf{f}), \mathrm{d} + \mathrm{d}\mathsf{f}) \cong    \pi_*(\Omega_\mathbb{P}^\bullet(\log \mathbb{D} + \mathbb{X})(*P_\mathsf{f}), \mathrm{d} + \mathrm{d}\mathsf{f}) \cong (\Omega_X^\bullet(\log D)(*P),\mathrm{d} +\mathrm{d}f).
    \]
    This proves the result when $r = 0$. If $r > 0$, we have the following diagram commutes for each $j$
    \[
    \begin{tikzcd}
    \pi_*(\Omega_{\mathbb{P}}^\bullet(\log \mathbb{D} + \mathbb{X} + \sum_{i\neq j} \mathbb{E}_i)(*P_\mathsf{f}),\mathrm{d} + \mathrm{d}\mathsf{f}) \ar[r,"\cong"] \ar[d,hookrightarrow] &  (\Omega_X^\bullet(\log D + \sum_{i\neq j} E_i)(*P),\mathrm{d} + \mathrm{d}f) \ar[d,hookrightarrow]\\ 
    \pi_*(\Omega_{\mathbb{P}}^\bullet(\log \mathbb{D} + \mathbb{X} + \sum_{i} \mathbb{E}_i)(*P_\mathsf{f}),\mathrm{d} + \mathrm{d}\mathsf{f}) \ar[r,"\cong"] &  (\Omega_X^\bullet(\log D + \sum_i E_i)(*P),\mathrm{d} +\mathrm{d}f)
    \end{tikzcd}
    \]
    The horizontal isomorphisms follow from the argument above. Applying an induction argument, we see that there is also a commutative diagram
    \[
    \begin{tikzcd}
    \sum_j\pi_*(\Omega_{\mathbb{P}}^\bullet(\log \mathbb{D} + \mathbb{X} + \sum_{i\neq j} \mathbb{E}_i)(*P_\mathsf{f}),\mathrm{d} + \mathrm{d}\mathsf{f}) \ar[r,"\cong"] \ar[d,hookrightarrow] & \sum_j (\Omega_X^\bullet(\log D + \sum_{i\neq j} E_i)(*P),\mathrm{d} + \mathrm{d}f) \ar[d,hookrightarrow]\\ 
    \pi_*(\Omega_{\mathbb{P}}^\bullet(\log \mathbb{D} + \mathbb{X} + \sum_{i} \mathbb{E}_i)(*P_\mathsf{f}),\mathrm{d} + \mathrm{d}\mathsf{f}) \ar[r,"\cong"] &  (\Omega_X^\bullet(\log D + \sum_i E_i)(*P),\mathrm{d} +\mathrm{d}f)
    \end{tikzcd}
    \]
    The cokernels of the vertical morphisms are the complexes of sheaves that we wish to compare. Basic homological algebra tells us that there is an isomorphism of complexes
    \[
    (\mathrm{R}\pi_*\Omega_\mathbb{P}^\bullet(\log \mathbb{D} + \mathbb{X},\mathbb{E}^{[r]}),\mathrm{d} + \mathrm{d}\mathsf{f}) \longrightarrow (\Omega_X^\bullet(\log D, E^{[r]}),\mathrm{d} + \mathrm{d}f)
    \]
    as required.
\end{proof}
\begin{remark}\label{r : commutes with residue maps}
    Let notation be as in Proposition \ref{p : Affine bundle cohomology agrees with the base}. Suppose $R$ is a smooth, irreducible hypersurface in $X$ and $\mathbb{R} = \pi^{-1}R$. Then $\mathbb{R}$ is a $\mathbb{P}^1$ bundle over $R$ as well. A local computation is enough to show that the isomorphism above commutes with the residue morphisms. Precisely,
    \[
    \begin{tikzcd}
         \pi_*\Omega^\bullet_\mathbb{P}(\log \mathbb{D} + \mathbb{R} + \mathbb{X},\mathbb{E}^{[r]}, \mathrm{d} + \mathrm{d}\mathsf{f}) \ar[d,"\mathrm{Res}_\mathbb{R}"] \ar[r,"\cong"]& (\Omega_X^\bullet(\log D+R, E^{[r]}),\mathrm{d} + \mathrm{d}f) \ar[d,"\mathrm{Res}_R"]  \\
        \pi_*\Omega^{\bullet-1}_\mathbb{R}(\log \mathbb{R}\cap (\mathbb{D} + \mathbb{X}),\mathbb{R}\cap\mathbb{E}^{[r]}, \mathrm{d} + \mathrm{d}\mathsf{f}|_\mathbb{R}) \ar[r,"\cong"] & (\Omega_R^{\bullet-1}(\log D\cap R, R\cap E^{[r]}),\mathrm{d} + \mathrm{d}f|_R)
    \end{tikzcd}
    \]
    commutes. This property will be used in the following discussion.
\end{remark}

\subsection{Statement of Theorem \ref{t : multiple direct sum decomposition} and discussion} \label{s : statement}

Let $X$ be a smooth projective variety and suppose $B_1,\dots, B_k$ are snc hypersurfaces chosen so that $B = \bigcup_{i=1}^k B_i$ is also snc. We assume that there are line bundles $L_1,\dots, L_k$ so that $\mathcal{O}_X(B_i) \cong L_i^{\otimes 2}$ for all $i$. To this data, there are two classes of geometric objects that are related to one another. 

\subsubsection*{Cyclic covers} For each subset $I \subseteq [k]$, we have a $(\mathbb{Z}/2)^{|I|}$-cover $\widehat{X}_I$ of $X$ with branch divisors $\bigcup_{i \in I} B_i$. Precisely, if $\widehat{X}_i$ is the double cover of $X$ branched at ${B}_i$, and $I = \{i_1,\dots,i_j\} \subseteq [k]$
\[
\widehat{X}_I = \widehat{X}_{i_1} \times_X \dots \times_X \widehat{X}_{i_j}.
\]
We let ${R}_I$ denote the ramification divisor of $\rho_I : \widehat{X}_I\rightarrow X$, which is the preimage of $B_I:=\bigcup_{i \in I}B_i$ in $\widehat{X}_I$. We also let $B^J = \cap_{j\in J} B_j$ and  ${Z}^{I} = \rho_I^{-1}B^{I^c}$. The local cohomology $H^*_{{Z}^{I}}(\widehat{X}_I\setminus {R}_I)$ admits a linear $(\mathbb{Z}/2)^{|I|}$ action which respects the Hodge filtration because the $\mathbb{Z}/2$ actions in question are by automorphisms preserving $Z^I$. We use the notation $H^*_{{Z}^{I}}(\widehat{X}_I\setminus {R}_I)^{(I)}$ to denote the intersection of the $(-1)$-eigenspaces of each of the $|I|$ different $\mathbb{Z}/2$ actions inherited from the automorphisms $\tau_{i_j} : \widehat{X}_{i_j} \rightarrow \widehat{X}_{i_j}$ with $i_j \in I$. $H_{Z^{I}}^*(\widehat{X}_I\setminus {R}_I)^{(I)}$ is a direct summand of the mixed Hodge structure $H_{Z^I}^*(\widehat{X}_I\setminus {R}_I)$. When $I$ is a single element set, we often simplify notation to $(-)$ instead of $(I)$.

\subsubsection*{LG models}\label{sss:lg} For each subset $J \subseteq [k]$ we denote 
\[
\mathbb{V}_J = \mathrm{Tot}\left( \bigoplus_{i \in J} L_i^{-1} \oplus \bigoplus_{i \notin J} L_i^{-2}\right).
\]
Let $\pi_J : \mathbb{V}_J \rightarrow X$ denote the canonical morphism. For each $i$, there are tautological sections $s_i$ for $\pi_J^*L_i^{-1}$ (if $i \in J$) and $t_i$ of $\pi_J^*L_i^{-2}$ (if $i \notin J$). There are sections $\sigma_i$ of $L_i^{\otimes 2}$ so that  $\{\sigma_i = 0\} = B_i$. When $i \in J$, $g_{2,i} := s_i^2\cdot \pi_J^*\sigma_i$ is a regular function on $\mathbb{V}_J$ and when $i \notin J$, $g_{1,i} := t_i\cdot \pi_J^*\sigma_i$ is a regular function on $\mathbb{V}_J$. Let $g_{2,J} := \sum_{i \in J}g_{2,i}$ and $g_{1,J^c} := \sum_{i \notin J} g_{1,i}$.

The main result of this section relates the cohomology of these two objects. For a $\mathbb{Q}$-filtered vector space $(V,F^\bullet)$,  and $\alpha \in \mathbb{Q}$ we let $(V(\alpha),F^\bullet)$ denote the vector filtered vector space $F^\lambda(V(\alpha)) = F^{\lambda-\alpha}V$. This notation is meant to be reminiscent of the Tate twist in usual Hodge theory.

\begin{theorem}\label{t : multiple direct sum decomposition}
    There is a filtered direct sum decomposition
    \[
    H^*(\mathbb{V}_J, g_{2,J} + g_{1,J^c}) \cong \bigoplus_{I \subseteq J} H^{*-|I|}_{Z^{I}}(\widehat{X}_I\setminus {R}_I)^{(I)}(|I|/2).
    \]
\end{theorem}

\begin{remark}\label{rem:notation change}
    In Section \ref{s:HLLY} we will use the notation
\begin{equation}\label{eq: setup notation for section 6}
H^*(\mathbb{V}_J,g_{2,J} + g_{1,J^c})^{(I)} := H^{*-|I|}_{Z^{I}}(\widehat{X}_I\setminus {R}_I)^{(I)}(|I|/2),
\end{equation}
to focus more on the irregular Hodge numbers of Landau--Ginzburg model $(\mathbb{V}_J, g_{2,J}+g_{1,J^c})$. 
\end{remark}

The remainder of Section \ref{s:coverLG} is devoted to proving Theorem \ref{t : multiple direct sum decomposition}. The remainder of Section \ref{s : statement} is meant to explain the geometric meaning of Theorem \ref{t : multiple direct sum decomposition}     to the reader.

In Theorem \ref{t : multiple direct sum decomposition}, we may either view objects on the right hand side as LG models equipped with potential $f =0$ or, by Example \ref{ex : irregular = regular}, the objects on the right hand side may be viewed as cohomology groups equipped with their canonical Hodge filtrations.

The following remarks and examples describe the content of Theorem \ref{t : multiple direct sum decomposition} in a few particular cases.

\begin{example}\label{ex : simple examples}
    If $r= 1$, Theorem \ref{t : multiple direct sum decomposition} says:
    \[
    H^*(\mathbb{V}_{\{1\}},g_{2,\{1\}}) \cong H^*_B(X) \oplus H^{*-1}(\widehat{X}\setminus {R})^{(-)}(1/2).
    \]
    If $r= 2$, Theorem \ref{t : multiple direct sum decomposition} says:
    \begin{align*}
    H^*(\mathbb{V}_{\{1,2\}},g_{2,\{1,2\}}+g_{1,\{1,2\}^c}) \cong H^{*-2}(\widehat{X}_{\{1,2\}}\setminus {R}_{\{1,2\}})^{(\{1,2\}^-)}(1) &\oplus H^{*-1}_{Z^{\{1\}}}(\widehat{X}_1 \setminus {R}_1)^{(-)}(1/2) \\ & \oplus H^{*-1}_{Z^{\{2\}}}(\widehat{X}_2 \setminus {R}_2)^{(-)}(1/2) \\ & \oplus H^{*}_{{B}_1\cap B_2}(X).
    \end{align*}
\end{example}

\begin{example}
    Suppose we have a hyperelliptic curve of genus $g$ with its usual double cover $C \rightarrow \mathbb{P}^1$ of the projective line whose branch locus $B$ is a collection of $2g+2$ points, thus $L = \mathcal{O}_{\mathbb{P}^1}(g+1)$. Then $\mathbb{V}_{\{1\}}$ is the total space of $\mathcal{O}_{\mathbb{P}^1}(-(g+1))$.  We see that ${H}^2_B(X) \cong \mathbb{Q}(-1)^{\oplus 2g+2}$ and $\dim \gr_F^0H^1(E)^{(-)} = \dim \gr_F^1H^1(E)^{(-)} = g$ and all other Hodge numbers vanish. Therefore Theorem \ref{t : multiple direct sum decomposition} claims that:
    \[
    \dim \gr_{F_{\mathrm{irr}}}^{\lambda}H^2(\mathbb{V}_{\{1\}},g_{2,\{1\}}) = \begin{cases}
        g & \text{ if } \lambda = 1/2 \text{ or } 3/2,\\
        2g+2 & \text{ if } \lambda = 1.
    \end{cases} 
    \]
    All other irregular Hodge numbers are zero.
\end{example}
\begin{remark}
    Assume that $B_1,\dots, B_k$ are smooth and the union $\bigcup_{i=1}^k B_i$ is snc. Then we may reinterpret Theorem \ref{t : multiple direct sum decomposition} in terms of the cohomology of the varieties $Z^I$. Note that in this case, each variety $Z_I$ is a codimension $|I^c|$ smooth complete intersection in $\widehat{X}_I$ and a  $(\mathbb{Z}/2)^{|I|}$-cover of the smooth variety  $B^{I^c}=\cap_{i\in I^c}B_i$ ramified along $Z^I\cap B_I$. The local cohomology $H_{Z^I}^{*-|I|}(\widehat{X}_I\setminus R_I)$ is isomorphic to $H^{*-k-|I^c|}(Z^I \setminus (Z^I\cap R_I))$ if $I\neq \emptyset$ by Poincar\'e--Lefschetz duality. Furthermore, one can check that $H^{*-k-|I^c|}(Z^I\setminus (Z^I\cap R_I))^{(I)} \cong H^{*-k-|I^c|}(Z^I)^{(I)}$ (see Proposition \ref{p : double cover interpretation of the first factor} below). Therefore, in this case, Theorem \ref{t : multiple direct sum decomposition} takes the form
    \[
    H^*(\mathbb{V}_J, g_{2,J} + g_{1,J^c}) \cong \bigoplus_{I \subseteq J} H^{*-k-|I^c|}(Z^I)^{(I)}(|I|/2).
    \]
\end{remark}

The following statement is an obvious consequence of Theorem \ref{t : multiple direct sum decomposition}, which we include because the underlying principle is important in Section \ref{s:HLLY}.
\begin{corollary}\label{cor:another description}
     If $J \subseteq K$ then $H^*(\mathbb{V}_J, g_{2,J} + g_{1,J^c})$ is a filtered direct summand of $H^*(\mathbb{V}_K,g_{2,K} + g_{1,K})$.
\end{corollary}

Suppose notation is as above. There is also a double cover $\widetilde{X}_{[k]}$ of $X$ whose branch locus is $B_{[k]} = \bigcup_{i=1}^k B_i$. Let $\widetilde{R}_{[k]}$ be the preimage of $B_{[k]}$ in $\widetilde{X}_{[k]}$. 

\begin{proposition}\label{p : double cover interpretation of the first factor}
    There is an isomorphism of filtered vector spaces:
    \begin{equation}\label{eq: cohomological decomposition}
    H^*(\widehat{X}_{[k]}) \cong \bigoplus_{I\subseteq [k]} H^*(\widehat{X}_{I})^{(I)} \cong \bigoplus_{I\subseteq [k]}H^*(\widehat{X}_{I}\setminus R_I)^{(I)}.
    \end{equation}
    Furthermore, $H^*(\widehat{X}_{[k]})^{([k])} \cong H^*(\widetilde{X}_{[k]})^{(-)}$ and $H^*(\widehat{X}_{[k]})^{(\emptyset)} \cong H^*(X).$
\end{proposition}
\begin{proof}
We may describe the $\mathbb{Z}/2$ coinvariant cohomology of $\widetilde{X}_{[k]}$ using \cite[Theorem 3.2]{Esnault1992vanishing} (see also Section \ref{s:intro doublecover}). For concreteness, we recall the statement we need. For any $I \subseteq [k]$, let $L_I = \bigotimes_{i \in I} L_i, \sigma_I = \prod_{i\in I}\sigma_i$. If $\widetilde{\rho} : \widetilde{X}_{[k]} \rightarrow X$ is the double covering map then: 
\begin{align*}
\mathrm{R}\widetilde{\rho}_*\Omega_{\widetilde{X}_{[k]}}^\bullet(\log \widetilde{R}_{[k]}) & \cong \Omega^\bullet_X(\log B_{[k]}) \oplus (\Omega_{X}^\bullet(\log B_{[k]})\otimes L_{[k]}^{-1}), \\ 
\mathrm{R}\widetilde{\rho}_*\Omega_{\widetilde{X}_{[k]}}^\bullet & \cong \Omega^\bullet_X \oplus (\Omega_{X}^\bullet(\log B_{[k]})\otimes L_{[k]}^{-1}).
\end{align*}
Here, $\Omega_{X}^\bullet(\log B_{[k]})$ is the $\mathbb{Z}/2$ coinvariant subsheaf and is equipped with the twisted differential $\mathrm{d} + \tfrac{1}{2}\mathrm{d}\log \sigma_{[k]}$ and the other summand in each expression, $\Omega_X^\bullet(\log B_{[k]})$ resp. $\Omega_X^\bullet$, is $\mathbb{Z}/2$ invariant and is equipped with the de Rham differential $\mathrm{d}$. Therefore  
\[
H^*(\widetilde{X}_{[k]} \setminus R_{[k]})^{(-)} \cong H^*(\widetilde{X}_{[k]})^{(-)}.
\]
Furthermore, the Hodge filtration  on $H^*(\widetilde{X}_{[k]})$ is identified with the filtration induced by truncation $\tau_{\leq p}$ filtration on $(\Omega_{X}^\bullet(\log B_{[k]})\otimes L_{[k]}^{-1},\mathrm{d} + \tfrac{1}{2}\mathrm{d}\log \sigma_{[k]})$. Now we prove that the same complex and filtration provide the Hodge filtration on $H^*(\widehat{X}_{[k]}\setminus R_{[k]})^{(-)}$ by induction on $k$. We note that there is a tower of double covers
\[
\widehat{X}_{[k]} \longrightarrow \widehat{X}_{[k-1]} \longrightarrow \dots \longrightarrow \widehat{X}_{[1]} \longrightarrow X.
\]
Let ${\rho}_{[k-1]} : \widehat{X}_{[k-1]} \rightarrow X$ denote the induced degree $2^{k-1}$ map. We may assume by induction that
\[
\mathrm{R}{\rho}_{[k-1]*}\Omega^\bullet_{\widehat{X}_{[k-1]}}(\log {R}_{[k-1]}) \cong \bigoplus_{I \subseteq [k-1]} \left(\Omega_{X}^\bullet(\log B_{[k-1]})\otimes L^{-1}_{I}, \mathrm{d} + \tfrac{1}{2}\mathrm{d}\log \sigma_{I}, \tau_{\leq \bullet}\right)
\]
and each factor corresponding to $I \subseteq [k-1]$ is coinvariant with respect to the subgroup $(\mathbb{Z}/2)^{I} \subseteq (\mathbb{Z}/2)^{k-1}$. The case $k-1 = 1$  is the case described by Esnault and Viewheg \cite{Esnault1992vanishing}.

Let $\widehat{\rho}_k : \widehat{X}_{[k]} \rightarrow \widehat{X}_{[k-1]}$ denote the double covering map. There is a decomposition into $\mathbb{Z}/2$ invariant and coinvariant parts, respectively,
\[
\mathrm{R}\widehat{\rho}_{k*}\Omega_{\widehat{X}_{[k]}}^\bullet(\log {R}_{[k]}) \cong \Omega^\bullet_{\widehat{X}_{[k-1]}}(\log \widehat{\rho}_k({R}_{[k]})) \oplus \left( \Omega^\bullet_{\widehat{X}_{[k-1]}}(\log \widehat{\rho}_k({R}_{[k]})) \otimes {\rho}_{[k-1]}^*L_k^{-1}\right)
\]
Applying the projection formula and induction, we get
\begin{align}\label{eq : eigenspace decomp}
 \mathrm{R}\rho_{[k]*}\Omega_{\widehat{X}_{[k]}}^\bullet(\log {R}_{[k]}) & \cong \bigoplus_{I \subseteq [k]} \left(\Omega_{X}^\bullet(\log B_{I})\otimes L^{-1}_{I}, \mathrm{d} + \tfrac{1}{2}\mathrm{d}\log \sigma_{I}\right).
\end{align}
If $\alpha_1,\dots, \alpha_k$ are canonical generators of $(\mathbb{Z}/2)^k$ then this induction argument applies to show that $\Omega_{X}^\bullet(\log B_{I})\otimes L^{-1}_{I}$ is the intersection of the $(-1)$-eigensheaves of $\alpha_i, i \in I$. Taking (hyper)cohomology of both sides of \eqref{eq : eigenspace decomp} one obtains \eqref{eq: cohomological decomposition}. Comparing with the first paragraph of the proof, one obtains the fact that $H^*(\widehat{X}_{[k]})^{([k])} \cong H^*(\widetilde{X}_{[k]})^{(-)}$ and $H^*(\widehat{X}_{[k]})^{(\emptyset)} \cong H^*(X)$.
\end{proof}

As a consequence, we see that $H^*(\widetilde{X}_{[k]})^{(-)}(k/2)$ is a filtered direct summand of $H^*(\mathbb{V}_{[k]},g_{2,[k]})$. This observation is important in Section \ref{s:HLLY} because it allows us to recast part of the cohomology of the double cover $\widetilde{X}_{[k]}$ in terms of the cohomology of the LG model $(\mathbb{V}_{[k]},g_{2,[k]})$. 


\subsection{Proof of Theorem \ref{t : multiple direct sum decomposition}, assuming Theorem \ref{t : single direct sum decomposition}}\label{s : mdsdecom}
Let $(X,D,f)$ denote a nondegenerate compactified LG model and assume that $E^{[r]} \subseteq X$ is the intersection of a collection of snc hypersurfaces, $E_1,\dots, E_r$.  There are three auxiliary LG models that we study in relation to $(X,D,f)$:
\begin{itemize}[\quad --]
    \item Suppose $\rho : \widehat{X}\rightarrow X$ is a double cover with normal crossings branch divisor $B$. Let $\widehat{D}$ denote $\rho^{-1}D$ and $\widehat{f} = \rho^{*}f$.  Let $\widehat{E}^{[r]}$ denote $\rho^{-1}E^{[r]}$ and let ${R} := \rho^{-1}B(\cong B)$. We are interested in understanding the nondegenerate compactified LG model $(\widehat{X},\widehat{D},\widehat{f})$.
\end{itemize}
Suppose there exists a line bundle $L$ so that $L^2 \cong \mathcal{O}_X(B)$.
\begin{itemize}[\quad --]
    \item Let $\mathbb{L}_{1} = \mathrm{Tot}(L^{-1})$ and let $\pi_{1} :\mathbb{L}_{1}\rightarrow X$ denote the usual projection morphism. Let $\mathbb{D}_{1} = \pi_{1}^{-1}D, \mathbb{E}^{[r]}_{1} = \pi_{1}^{-1}E^{[r]}$, and let $s$ denote the tautological section of $\pi_1^*L$. Then $s^2 \in \Gamma(\mathbb{L}_1,\pi_1^*L^{-2})$ and $\pi_1^*\sigma \in \Gamma(\mathbb{L}_2,\pi_1^*L^2)$. Let $\mathsf{g}_{2} = s^2\cdot (\pi_1^*\sigma), \mathsf{f}_{1} = \pi^{*}_{1}f$, which are regular functions on $\mathbb{L}_1\setminus \mathbb{D}_1$. We consider the LG model $(\mathbb{L}_{1}\setminus \mathbb{D}_{1}, \mathsf{g}_{2} + \mathsf{f}_{1})$.
   
    \item Let $\mathbb{L}_{2} = \mathrm{Tot}(L^{-2})$ and let $\pi_{2} :\mathbb{L}_{2}\rightarrow X$ denote the usual projection morphism. Let $\mathbb{D}_{2} = \pi_{1}^{-1}D, \mathbb{E}^{[r]}_{2} = \pi_{2}^{-1}E^{[r]}$ and let $t$ denote the tautological section of $\pi_2^*L^{-2}$.  Then $t \in \Gamma(\mathbb{L}_2,\pi_2^*L^{-2})$ and $\pi_2^*\sigma \in \Gamma(\mathbb{L}_2,\pi_1^*L^2)$. Let $\mathsf{g}_{1} = t\cdot (\pi^*\sigma), \mathsf{f}_{2} = \pi^{*}_{2}f$, which are regular functions on $\mathbb{L}_2 \setminus \mathbb{D}_2$. We consider the LG model $(\mathbb{L}_{2}\setminus \mathbb{D}_{2}, \mathsf{g}_{1} + \mathsf{f}_{2})$.

\end{itemize}
Section \ref{s : proof of t: single direct sum decomposition} will be devoted to proving the following theorem which is a global variation on results proved by Sabbah and Yu \cite{Sabbah2023cyclic} and Fresán, Sabbah, and Yu \cite{Fresan2022} respectively. 
\begin{theorem} \label{t : single direct sum decomposition}
    For every $r \geq 0$, there are filtered isomorphisms 
    \begin{align}
    H_{\mathbb{E}_{1}^{[r]}}^*(\mathbb{L}_{1}\setminus \mathbb{D}_{1},\mathsf{g}_{2} + \mathsf{f}_{1}) & \cong H^{*-1}_{\widehat{E}^{[r]}}(\widehat{X}\setminus (\widehat{D} \cup R),\widehat{f})^{(-)}(1/2) \oplus H^{*}_{B \cap E^{[r]}}(X\setminus D,f),\label{e : situation for the g reduction} \\ 
    H_{\mathbb{E}^{[r]}_{2}}^*(\mathbb{L}_{2}\setminus \mathbb{D}_{2},\mathsf{g}_{1} + \mathsf{f}_{2}) &\cong  H^{*}_{B \cap E^{[r]}}(X\setminus D,f). \label{e : situation for the f reduction}
    \end{align}
\end{theorem}
Theorem \ref{t : multiple direct sum decomposition} follows from Theorem \ref{t : single direct sum decomposition}. Let us start with a brief remark to make the relation between Theorem \ref{t : multiple direct sum decomposition} and Theorem \ref{t : single direct sum decomposition} clear; if $r = 1$ then $I = \emptyset$ or $\{1\}$. In either case $g_{2,\emptyset} = 0, g_{1,\emptyset^c} = \mathsf{g}_{1}$ and $\mathbb{V}_{\emptyset} = \mathbb{L}_{2}$ or $g_{2,\{1\}} = \mathsf{g}_{2}, g_{1,\{1\}^c} =  0$ and $\mathbb{V}_{\{1\}} = \mathbb{L}_{1}$. Applying \eqref{e : situation for the f reduction} and \eqref{e : situation for the g reduction} respectively, with $\mathbb{E}^{[r]} = \mathbb{D} = \emptyset$ and $\mathsf{f} = 0$ yields, 
    \[
    H^*(\mathbb{V}_{\emptyset}, g_{1,\emptyset^c}) = H^{*-1}_{B}(X),\qquad H^*(\mathbb{V}_{\{1\}},g_{2,\{1\}}) = H^{*-1}(\widehat{X}\setminus R)^{(-)}(1/2) \oplus H^{*}_{B}(X)
    \]
    as desired. This recovers the first part of Example \ref{ex : simple examples} immediately.

    Taking Theorem \ref{t : single direct sum decomposition} for granted, we prove Theorem \ref{t : multiple direct sum decomposition}.
\begin{proof}[Proof of Theorem \ref{t : multiple direct sum decomposition}]
    We prove a slightly stronger fact: If $E_1,\dots, E_r$ are snc divisors in $X$, then 
    \begin{equation}\label{e : inductive hypothesis}
    H_{\mathbb{E}^{[r]}}^*(\mathbb{V}_J,g_{2,J} + g_{1,J^c}) \cong \bigoplus_{I \subseteq J} H^{*-|I|}_{\widehat{E}^{[r]} \cap Z^I}(\widehat{X}_I\setminus {R}_I)^{(I)}(|I|/2).
    \end{equation}
    Assume \eqref{e : inductive hypothesis} is true for any $X$, $E_1,\dots, E_r$, and any collection of $k-1$ line bundles $L_1,\dots, L_{k-1}$ for $k>1$. The $k=2$ case is Theorem \ref{t : single direct sum decomposition}. We may view $\mathbb{V}_J = \mathrm{Tot}(\bigoplus_{i \in J} L^{-1}_i \oplus \bigoplus_{i \notin J} L_i^{-2})$ as the total space of a line bundle over 
    \[
    \mathbb{V}_J' = \mathrm{Tot}\left(\bigoplus_{\substack{i \in J \\ i \neq k}} L^{-1}_i \oplus \bigoplus_{\substack{i \notin J\\ i \neq k}} L_i^{-2}\right).
    \]
    Precisely, let $\pi' : \mathbb{V}_J'\rightarrow X$ be the usual projection morphism, then $\mathbb{V}_J$ is the total space of the line bundle $(\pi')^*L_k^{-1}$ if $k \in J$ and the total space of $(\pi')^*L_k^{-2}$ if $k \notin J$. We use the notation $g_{2,J}' = \sum_{i\in J,i \neq k} s_i^2(\pi')^*\sigma_i$ and $g_{1,J^c}' = \sum_{i \notin J, i\neq k} t_i(\pi')^*\sigma_i$.
    
    We address these two situations separately, applying induction on $k$ in both cases. The base case follows directly from Theorem \ref{t : single direct sum decomposition}.
    \vspace{0.2cm}
    
    \noindent \textit{Case 1: Assume  $k\in J$.} We may write $g_{2,J} = g_{2,k} + g_{2,J\setminus k}$ where  $g_{2,J\setminus k} = \sum_{j\in J\setminus k} s_j^2\cdot\pi_J^*\sigma_j$.  If $\pi : \mathbb{V}_J \rightarrow \mathbb{V}_J'$ then $\pi^*g_{2,J}' = g_{2,J\setminus k}$ and $\pi^*g_{1,J^c}' = g_{1,J^c}$. We apply \eqref{e : situation for the g reduction} directly to see that 
    \begin{equation}\label{e : first decomposition}
    H^*_{\mathbb{E}^{[r]}}(\mathbb{V}_J,g_{2,k} + g_{2,J\setminus k}+ g_{1,J^c}) = H_{\widehat{\mathbb{E}}^{'[r]}}^{*-1}(\widehat{\mathbb{V}}_J'\setminus {\mathbb{R}}_k', \widehat{g}_{2,J}' + \widehat{g}_{1,J^c}')^{(-)}(1/2) \oplus H^{*}_{\mathbb{E}^{'[r]}\cap \mathbb{B}'_k}( \mathbb{V}_J',  g_{2,J}' + g_{1,J^c}')
    \end{equation}
    where $\widehat{\mathbb{V}}_J'$ denotes the double cover of $\mathbb{V}_J'$ ramified along  the vanshing locus of $(\pi_J')^*\sigma_k$, denoted $\mathbb{B}_k'$, and where ${\mathbb{R}}_k'$ is the preimage of $\mathbb{B}_k'$ in  $\widehat{\mathbb{V}}_J'$. Following similar conventions, let $\widehat{\mathbb{E}}^{'[r]}$ and $\mathbb{E}^{'[r]}$ be the preimages of $E^{[r]}$ in $\widehat{\mathbb{V}}'_J$ and $\mathbb{V}_j'$ respectively. The notation $\widehat{g}_{2,J}', \widehat{g}_{1,J^c}'$ indicates the pullback of ${g}_{2,J}',{g}_{1,J^c}'$ to $\widehat{\mathbb{V}}_J'$. In \eqref{e : first decomposition} the notation $(-)$ indicates coinvariance with respect to the $\mathbb{Z}/2$ action $\widehat{\mathbb{V}}_J'\rightarrow \mathbb{V}'_J$.
    
    Observe that $(\mathbb{V}_J',g_{2,J}'+g_{1,J^c}') = (\mathbb{V}_{J\setminus k}, g_{2,J\setminus k} + g_{1,(J\setminus k)^c})$ where, on the right hand side, $(J\setminus k)^c$ is viewed as a subset of $\{1,\dots, k-1\}$. By induction on $k$ and Theorem \ref{t : single direct sum decomposition}, \eqref{e : inductive hypothesis} says that\footnote{This step uses the full force of \eqref{e : inductive hypothesis}. A direct proof of the weaker statement in Theorem \ref{t : multiple direct sum decomposition} would require more subtle induction at this point.}
    \begin{align}
    H_{{\mathbb B}_k'\cap \mathbb{E}^{'[r]}}^*(\mathbb{V}_J', g_{2,J}' + g_{1,J^c}') & \cong \bigoplus_{I \subseteq (J\setminus k)} H_{{R}_k \cap {Z'}^{I} \cap \widehat{E}^{[r]}}^{*-|I|}(\widehat{X}_I \setminus {R}_I)^{(I)}(|I|/2)\label{e : second decomposition} \\ 
     & \cong \bigoplus_{I \subseteq (J\setminus k)} H_{{Z}^{I}\cap \widehat{E}^{[r]}}^{*-|I|}(\widehat{X}_I \setminus {R}_I)^{(I)}(|I|/2) \label{e : second decomposition 2}
    \end{align}
    In \eqref{e : second decomposition} we have used the notation ${Z'}^I$ to denote the preimage of $\bigcap_{i \in (I^c \setminus k)} B_i$ in $\widehat{X}_I$ and in \eqref{e : second decomposition 2} we use the fact that ${Z'}^{I}\cap R_k = Z^I$. This tells us that the factors of the right hand side of \eqref{e : inductive hypothesis} corresponding to sets $I$ which are contained in $J\setminus k$ correspond to the second summand of \eqref{e : first decomposition}. We now use a similar argument to deal with the first summand of \eqref{e : first decomposition}

    Since $\widehat{\mathbb{V}}_J'  \rightarrow \mathbb{V}_J'$ is a cyclic double cover of $\mathbb{V}_J'$ ramified along the divisor $\mathbb{B}_{k}' = (\pi'_J)^{-1}B_{k}$, there is a cartesian diagram
    \[
    \begin{tikzcd}
        \widehat{\mathbb{V}}_J' \setminus {\mathbb{R}}_k' \ar[r,"2:1"] \ar[d,"\widehat{\pi}'_J"] & {\mathbb{V}}_J' \setminus {\mathbb{B}}_{k}' \ar[d,"\pi_J'"] \\
        \widehat{X}_{\{k\}} \setminus R_{k} \ar[r,"\eta_{\{k\}}"] & X \setminus B_{k}
    \end{tikzcd}
    \]
    where 
    \[
    \widehat{\mathbb{V}}_J' = \mathrm{Tot}_{\widehat{X}_{\{k\}}}\left(\bigoplus_{i \in (J\setminus k)} \eta_{\{k\}}^*L_i^{-1} \oplus \bigoplus_{i \notin J}\eta_{\{k\}}^*L_i^{-2}  \right).
    \]
We will now use the notation $\widehat{X}'_I$ to denote the $(\mathbb{Z}/2)^{|I|}$-cover of $\widehat{X}_{\{k\}}$ coming from the collection of sections $\{\eta_{\{k\}}^*\sigma_i\mid i \in I\}$ whenever $k \notin I$. Let ${R}'_{I}$ denote the preimage of the corresponding ramification divisor. Observe that  $\widehat{X}_I' = \widehat{X}_{I\cup k}$ and the union of $R'_I $ and the preimage of $B_k$ in $\widehat{X}_I'$ is $R_{I\cup k}$. 
    
    Applying induction a second time,
    \begin{align*}
        H^{*-1}_{\widehat{\mathbb{E}}^{'[r]}}(\widehat{\mathbb{V}}_J'\setminus \widehat{\mathbb{R}}_k', \widehat{g}_{2,J}'+ \widehat{g}_{1,J^c}') & \cong \bigoplus_{I\subseteq (J\setminus k)} H^{*-|I|-1}_{Z^{I\cup k}\cap \widehat{E}^{[r]}}((\widehat{X}'_{I}\setminus {R}'_k) \setminus {R}_I')^{(I)}(|I|/2) \\ &
        \cong \bigoplus_{\substack{I \cup k \subseteq J \\ k \notin I}} H^{*-|I\cup k|}_{Z^{I\cup k}\cap \widehat{E}^{[r]}}(\widehat{X}_I'\setminus {R}'_{I\cup k})^{(I)}(|I|/2).
    \end{align*}
     The exponent $(I)$ in the displayed equation above indicates only coinvariance with respect to $(\mathbb{Z}/2)^{\{i\}}$ action inducing the quotient $\widehat{X}_I \rightarrow \widehat{X}_{\{k\}}$. Therefore, taking coinvariance with respect to the $\mathbb{Z}/2$ action on the $k$-th factor, we get 
    \begin{align}\label{e : third decomposition}
        H_{\widehat{\mathbb{E}}^{'[r]}}^{*-1}(\widehat{\mathbb{V}}_J'\setminus \widehat{\mathbb{R}}_k', \widehat{g}_{2,J}'+ \widehat{g}_{1,J^c}')^{(-)} 
        \cong \bigoplus_{\substack{(I\cup k) \subseteq J\\ k \notin I}} H^{*-|I\cup k|}_{Z^{I\cup k}\cap \widehat{E}^{[r]}}(\widehat{X}_{I\cup k}\setminus {R}_{I\cup k})^{(I\cup k)}(|I|/2).
    \end{align}
    Here $(-)$ on the left indicates coinvariance under the $\mathbb{Z}/2$ action and $(I\cup k)$ on the right indicates $(\mathbb{Z}/2)^{|I\cup k|}$ coinvariance. Combining \eqref{e : first decomposition}, \eqref{e : second decomposition}, \eqref{e : second decomposition 2}, and \eqref{e : third decomposition} we obtain \eqref{e : inductive hypothesis}.
\vspace{0.2cm}

    \noindent \textit{Case 2: Assume $k \notin J$.} We may write $g_{1,J^c} = g_{1,\{k\}^c} + g_{1,(J\setminus k)^c}$, $g_{1,(J\setminus k)^c} = \sum_{j\in [k-1]\setminus J} t_i\cdot\pi_J^*\sigma_j$. We apply \eqref{e : situation for the f reduction} to see that 
    \begin{equation}\label{e : final decomposition}
    H^*_{\mathbb{E}^{[r]}}(\mathbb{V}_J,g_{1,\{k\}^c} + g_{1,(J\setminus k)^c} + g_{2,J}) \cong  H^{*}_{\mathbb{R}'_k\cap \mathbb{E}^{'[r]}}( \mathbb{V}_J',  g_{1,J^c}' + g_{2,J}').
    \end{equation}
    We use the notation developed above in \eqref{e : final decomposition}. By induction on $k$, we may write 
    \[
    H^{*}_{\mathbb{R}'_k\cap {\mathbb{E}}^{'[r]}}( \mathbb{V}_{J}',  g_{2,J}' + g_{1,J^c}') \cong \bigoplus_{I\subseteq J} H_{\widehat{E}^{[r]}\cap {R}_k \cap Z_0^{I}}^{*-|I|}(\widehat{X}_I\setminus {R}_I)^{(-)}(|I|/2)
    \]
    where $Z_0^I$ indicates the preimage of $\cap_{i \in I^c\setminus k} B_i$ in $\widehat{X}_I$. Since ${R}_k \cap Z^I_0 = Z^I$, this completes the proof.
\end{proof}

\subsection{Proof of Theorem \ref{t : single direct sum decomposition}} \label{s : proof of t: single direct sum decomposition}

The entirety of the section is devoted to proving Theorem \ref{t : single direct sum decomposition}. There are two basic morphisms that play a significant role in this proof. The first will be explained now, and the second will be explained later.  

 Observe that $\mathbb{L}_{1}:=\mathrm{Tot}(L^{-1})$ is a double cover of $\mathbb{L}_{2}:=\mathrm{Tot}({L}^{-2})$ ramified along the zero section. Precisely, if $U \times \mathbb{C}$ is a trivializing chart on $\mathbb{L}_1$, then the covering map takes $(\underline{x},t) \in U\times \mathbb{C}\subseteq \mathbb{L}_1$ to $(\underline{x},t^2) \in \mathbb{L}_2$. We let $\rho : \mathbb{L}_{1} \rightarrow \mathbb{L}_{2}$ denote this morphism, and we let $\pi_{2}:\mathbb{L}_{2} \rightarrow X$ denote the canonical morphism. Recall that $\mathsf{g}_{1} = t\cdot \pi_{2}^*\sigma$ where $t$ denotes the canonical section of $\pi_{2}^*L^{-2}$. The following diagram commutes:
\[
\begin{tikzcd}
    \mathbb{L}_{1}\ar[rd,"\pi_{1}",swap] \ar[r,"\rho"] & \mathbb{L}_{2} \ar[d,"\pi_{2}"] \\
    & X 
\end{tikzcd}
\]
We notice that $\rho^*\mathsf{g}_{1} = \mathsf{g}_{2}$ and if $\mathbb{E}_{2}^{[r]} = \pi_{2}^{-1}E^{[r]}$ then $\rho^{-1}\mathbb{E}_{2}^{[r]} = \mathbb{E}_{1}^{[r]}$. There is a $\mathbb{Z}/2$ action on $H_{\mathbb{E}_{1}^{[r]}}^*(\mathbb{L}_{1}\setminus \mathbb{D}_{1},\mathsf{g}_{2} +\mathsf{f}_{1})$ which preserves the irregular Hodge filtration. Therefore, we have a filtered direct sum decomposition into eigenspaces,
\[
H^*_{\mathbb{E}_{1}^{[r]}}(\mathbb{L}_{1}\setminus \mathbb{D}_{1},\mathsf{g}_{2} +\mathsf{f}_{1}) = H^*_{\mathbb{E}_{1}^{[r]}}(\mathbb{L}_{1}\setminus \mathbb{D}_{1},\mathsf{g}_{2} +\mathsf{f}_{1})^{(+)} \oplus H^*_{\mathbb{E}_{1}^{[r]}}(\mathbb{L}_{1}\setminus \mathbb{D}_{1},\mathsf{g}_{2} +\mathsf{f}_{1})^{(-)}.
\]
We may identify the $(+)$ eigenspace using the following  result which is a direct generalization of a classical result about the Hodge numbers of cyclic coverings.  
\begin{proposition}\label{p : invariant cohomology and covers of LG models}
    Let $(X,D,f)$ be a nondegenerate compactified  LG model and suppose $\rho : \widehat{X} \rightarrow X$ is a $(\mathbb{Z}/2)$-cover with smooth ramification divisor $R$ whose union with $D$ is snc. Then $(\widehat{X},\widehat{D} = \rho^{-1}D ,\widehat{f} = \rho^*f)$ is a nondegenerate compactified LG model and $H_{\widehat{E}^{[r]}}^*(\widehat{X}\setminus \widehat{D},\widehat{f})^{(+)} \cong H^*_{E^{[r]}}(X\setminus D,f)$.
\end{proposition}
\begin{proof}
Let us first address the case where $r=0$. Then we know, classically, on the level of sheaves, that 
\begin{equation}\label{e : eigensheaf decomposition}
\rho_*\Omega^p_{\widehat{X}}(\log \widehat{D}\cup {R}) \cong \mathrm{R}\rho_*\Omega^p_{\widehat{X}}(\log \widehat{D}\cup {R}) \cong \Omega^p_X(\log D\cup B) \oplus \left(\Omega^p_X(\log D\cup B)\otimes \mathcal{O}_X(-\tfrac{1}{2}B)\right)
\end{equation}
and that $\mathrm{R}\rho_*\mathrm{d}$ restricts to the second  summand of the right hand side of \eqref{e : eigensheaf decomposition} as $\mathrm{d} + \left(\tfrac{1}{2}\right)\mathrm{d}\log \sigma$ and $\mathrm{R}\rho_*\mathrm{d}$ restricts to the first summand of \eqref{e : eigensheaf decomposition} as $\mathrm{d}$. The first summand of the right hand side of \eqref{e : eigensheaf decomposition} is the $(+1)$ eigensheaf under the $\mathbb{Z}/2$ action and the second summand is the $(-1)$ eigensheaf. Since the pole divisor of $\rho^*f$, satisfies $\widehat{P} = \rho^{-1}P$, we have $\mathcal{O}_{\widehat{X}}(*\widehat{P}) \cong \pi^*\mathcal{O}_X(*P)$, and  the projection formula tells us that 
\begin{align*}
\rho_*(\Omega^p_{\widehat{X}}(\log \widehat{D}\cup {R})(*\widehat{P})) & \cong\mathrm{R}\rho_*(\Omega^p_{\widehat{X}}(\log \widehat{D}\cup {R})(*\widehat{P})) \\ & \cong \Omega^p_X(\log D\cup B)(*P) \oplus \left(\Omega^p_X(\log D\cup B)(*P)\otimes \mathcal{O}_X(-\tfrac{1}{2}B)\right).
\end{align*}
The differential $\mathrm{d}\widehat{f} = \rho^*\mathrm{d}f$ induces the differential $\mathrm{d}{f}$ which preserves the eigensheaf decomposition. Since no component of $P$ is contained in $B$, the irregular filtration is preserved by this decomposition.

From this, we may show that the invariant part of $H^*(\widehat{X}\setminus \widehat{D}\cup R,\widehat{f})$ is isomorphic to $H^*(X\setminus D \cup B,f)$. We now show that the invariant part of $H^*(\widehat{X}\setminus \widehat{D},\widehat{f})$ is isomorphic to $H^*(X\setminus D,f)$. To do so, take the short exact sequence of sheaves 
\begin{equation}\label{e : residue ses}
0 \longrightarrow \Omega_{\widehat{X}}^p(\log \widehat{D}) \longrightarrow \Omega_{\widehat{X}}^p(\log \widehat{D} \cup {R}) \xrightarrow{\mathrm{Res}_{{R}}} i_*\Omega^{p-1}_{{R}}(\log R\cap \widehat{D})\longrightarrow 0
\end{equation}
where $i : R\hookrightarrow \widehat{X}$ is the inclusion map. Applying pushforward to this short exact sequence, along with the fact that $\rho\cdot i$ is the usual inclusion of $j:{R}\cong B\hookrightarrow X$, we see that $\mathrm{R}\rho_*i_*\Omega^{p-1}_{R}(\log {R}\cap \widehat{D}) \cong j_*\Omega^{p-1}_{{B}}(\log {B}\cap {D})$. A local calculation then shows that the induced map is
\[
\rho_*\mathrm{Res}_{{R}} = (\mathrm{Res}_R \oplus 0): \Omega^p_X(\log D\cup B) \oplus \Omega^p_X(\log D\cup B)\otimes \mathcal{O}_X(-\tfrac{1}{2}B) \longrightarrow j_*\Omega^{p-1}_{{B}}(\log {B}\cap {D}).
\]
This map is surjective, because $\mathrm{Res}_R$ is surjective. Applying $\rho_*$ to \eqref{e : residue ses}, we get a short exact sequence,
\[
0 \longrightarrow \rho_*\Omega_{\widehat{X}}^p(\log \widehat{D}) \longrightarrow \Omega^p_X(\log D\cup B) \oplus \Omega^p_X(\log D\cup B)\otimes \mathcal{O}_X(-\tfrac{1}{2}B) \longrightarrow j_*\Omega^{p-1}_{{B}}(\log {B}\cap {D}) \longrightarrow 0.
\]
The kernel of $\mathrm{Res}_R$ is $\Omega_X^p(\log D\cup B)$ is $\Omega_X^p(\log D).$ Therefore, $\mathrm{R}\rho_*\Omega^p_{\widehat{X}}(\log \widehat{D}) \cong \Omega^p_X(\log D) \oplus \Omega^p_X(\log D\cup B)\otimes \mathcal{O}_X(-\tfrac{1}{2}B)$, where the first summand is the $(+1)$ eigensheaf and the second summand is the $(-1)$ eigensheaf. Remark that $\mathrm{R}\rho_*(\mathrm{d} + \mathrm{d}\widehat{f}) = (\mathrm{d} + \mathrm{d}f, \mathrm{d} + \mathrm{d}f + \left(\tfrac{1}{2}\right) \mathrm{d}\log \sigma)$. This proves the desired result when $E = \emptyset$. 

In the general case the proof is similar to the final step of the proof of Theorem \ref{p : Affine bundle cohomology agrees with the base}. Recall that $E^{[\ell]} = E_1\cap \dots \cap E_\ell$, and let $\widehat{E}_i = \rho^{-1}E_i$. For each $j$, we have a commutative diagram of isomorphisms and inclusions,
\[
\begin{tikzcd}
    \Omega_X^\bullet(\log D + \sum_{i\neq j}E_i)(*P) \ar[r,"\cong"] \ar[d,hookrightarrow ] & \left(\rho_*\Omega^\bullet_{\widehat{X}}(\log \widehat{D} + \sum_{i \neq j} \widehat{E}_i)(*\widehat{P})\right)^{(+)} \ar[d, hookrightarrow] \\
    \Omega_X^\bullet(\log D + \sum_{i}E_i)(*P) \ar[r,"\cong"]  & \left(\rho_*\Omega^\bullet_{\widehat{X}}(\log \widehat{D} + \sum_{i} \widehat{E}_i)(*\widehat{P})\right)^{(+)} 
\end{tikzcd}
\]
Using the vanishing of $\mathrm{R}^m\rho_*\Omega^\bullet_{\widehat{X}}(\log \widehat{D} + \sum_{i\neq j} \widehat{E}_i)(*\widehat{P})$ and $\mathrm{R}^m\rho_*\Omega^\bullet_{\widehat{X}}(\log \widehat{D} + \sum_{i} \widehat{E}_i)(*\widehat{P})$ when $m>0$, we obtain isomorphisms
\[
\dfrac{\Omega_X^\bullet(\log D + \sum_{i}E_i)(*P)}{\Omega_X^\bullet(\log D + \sum_{i\neq j}E_i)(*P)} \cong \rho_*\left(\dfrac{\Omega^\bullet_{\widehat{X}}(\log \widehat{D} + \sum_{i} \widehat{E}_i)(*\widehat{P})}{\Omega^\bullet_{\widehat{X}}(\log \widehat{D} + \sum_{i \neq j} \widehat{E}_i)(*\widehat{P})} \right)^{(+)}
\]
Taking sums and applying an induction argument, we get commutative diagrams, 
\[
\begin{tikzcd}
    \sum_j\Omega_X^\bullet(\log D + \sum_{i\neq j}E_i)(*P) \ar[r,"\cong"] \ar[d,hookrightarrow ] & \sum_j \left(\rho_*\Omega^\bullet_{\widehat{X}}(\log \widehat{D} + \sum_{i \neq j} \widehat{E}_i)(*\widehat{P})\right)^{(+)} \ar[d, hookrightarrow] \\
    \Omega_X^\bullet(\log D + \sum_{i}E_i)(*P) \ar[r,"\cong"]  & \left(\rho_*\Omega^\bullet_{\widehat{X}}(\log \widehat{D} + \sum_{i} \widehat{E}_i)(*\widehat{P})\right)^{(+)} .
\end{tikzcd}
\]
From which we obtain isomorphisms
\[
\dfrac{\Omega_X^\bullet(\log D + \sum_{i}E_i)(*P)}{\sum_j\Omega_X^\bullet(\log D + \sum_{i\neq j}E_i)(*P)} \cong \rho_*\left(\dfrac{\Omega^\bullet_{\widehat{X}}(\log \widehat{D} + \sum_{i} \widehat{E}_i)(*\widehat{P})}{\sum_j \Omega^\bullet_{\widehat{X}}(\log \widehat{D} + \sum_{i \neq j} \widehat{E}_i)(*\widehat{P})} \right)^{(+)}
\]
Applying the definition of the local twisted de Rham complex, we obtain the result.
\end{proof}

Therefore, we see that 
\[
H^*_{\mathbb{E}_{1}^{[r]}}(\mathbb{L}_{1}\setminus \mathbb{D}_1,\mathsf{g}_{2} +\mathsf{f}_{1}) = H_{\mathbb{E}^{[r]}_{2}}^*(\mathbb{L}_{2} \setminus \mathbb{D}_{2},\mathsf{g}_{1} +\mathsf{f}_{2}) \oplus H_{\mathbb{E}_{1}^{[r]}}^*(\mathbb{L}_{1}\setminus \mathbb{D}_{1},\mathsf{g}_{2} +\mathsf{f}_{1})^{(-)}.
\]
Proposition \ref{p : reducing local cohomology to the base} below is surely well known to experts but it does not seem to be proven in the generality that we need (see \cite{yu2014irregular, dimca2000dwork,Fresan2022} for similar results when $f =0$, $\mathbb{B}$ is smooth, or $X$ is affine, respectively). It is a consequence of the following Lemma.
\begin{lemma}\label{l : two isomorphisms}
    There are filtered isomorphisms:
    \begin{align*}
        H^*_{\mathbb{E}_{2}^{[r]}\cap \mathbb{B}_{2}}(\mathbb{L}_{2}\setminus \mathbb{D}_{2}, \mathsf{g}_{1}+ \mathsf{f}_{2}) \cong H^*_{E^{[r]}\cap B}(X\setminus D,f) \cong H^*_{\mathbb{E}_{1}^{[r]}\cap \mathbb{B}_{1}}(\mathbb{L}_{1}\setminus \mathbb{D}_{1}, \mathsf{g}_{2} + \mathsf{f}_{1})
    \end{align*}
\end{lemma}
\begin{proof}
  We prove the first isomorphism. The proof of the second isomorphism is essentially  identical. Write $B = \bigcup_{i=1}^\ell B_i$ where $B_i$ are smooth and irreducible, and let $B^I = \cap_{i \in I} B_i$, $\mathbb{B}_{2}^I = \pi_{2}^{-1}B^I$. We make the important observation that $\mathsf{g}_{1}|_{\mathbb{B}_{2}^I} = 0$ for all $I$. (It is also true that $\mathsf{g}_{2}|_{\mathbb{B}_{1}^I} = 0$ for all $I$, which is why the second isomorphism follows from the same argument.) Following \eqref{e : residue resolution}, there is a residue resolution of complexes,
    \begin{equation}\label{e : residue resolution 2}
\begin{array}{ll}
0 \longrightarrow \Omega_{\mathbb{P}_{2}}^\bullet(\log \mathbb{X}_{2} + \mathbb{D}_{2}, \mathbb{E}_{2}^{[r]}\cap \mathbb{B}_{2})(*P_\mathsf{f}) & \longrightarrow \bigoplus_{|I| =1}\Omega_{\mathbb{B}_{2}^I}^{\bullet-1}(\log \partial\mathbb{B}^I_{2}, \mathbb{E}_{2}^{[r]}\cap \mathbb{B}_{2}^I)(*P_\mathsf{f}) \\ & \longrightarrow \bigoplus_{|I| = 2}\Omega_{\mathbb{B}_{2}^I}^{\bullet-2}(\log \partial\mathbb{B}_{2}^I, \mathbb{E}_{2}^{[r]}\cap \mathbb{B}_{2}^I)(*P_\mathsf{f}) \longrightarrow \dots 
\end{array}
\end{equation}
The columns of this resolution of complexes are equipped with the differential $(\mathrm{d} + \mathrm{d}\mathsf{g}_{1} + \mathrm{d}\mathsf{f}_{2})|_{\mathbb{B}_{2}^I} = (\mathrm{d} + \mathrm{d}\mathsf{f}_{2})|_{\mathbb{B}_{2}^I}$. We apply $\mathrm{R}\pi_{2*}$ to  \eqref{e : residue resolution 2}. Since each $\mathbb{B}_{2}^I$ is a projective bundle over $B^I$, we may apply Proposition \ref{p : Affine bundle cohomology agrees with the base} to each complex $(\Omega_{\mathbb{B}^I_2}^\bullet(\log \partial\mathbb{B}^I_{2},\mathbb{E}_{2}^{[r]}\cap \mathbb{B}_{2}^I)(*P_\mathsf{f}), \mathrm{d} + \mathrm{d}\mathsf{f})$ in obtain a resolution
    \begin{equation}\label{e : residue resolution 3}
\begin{array}{ll}
0 \longrightarrow \mathrm{R}\pi_{2*}\Omega_{\mathbb{P}_{2}}^\bullet(\log \mathbb{X}_{2} + \mathbb{D}_{2}, \mathbb{E}_{2}^{[r]}\cap \mathbb{B}_{2})(*P_\mathsf{f}) & \longrightarrow \bigoplus_{|I| =1}\Omega_{B^I}^{\bullet-1}(\log \partial B^I, E^{[r]}\cap B^I)(*P) \\ & \longrightarrow \bigoplus_{|I| = 2}\Omega_{B^I}^{\bullet-2}(\log \partial B^I, E^{[r]}\cap B^I)(*P) \longrightarrow \dots 
\end{array}
\end{equation}
where the columns of this resolution of complexes are equipped with differentials $\mathrm{R}\pi_{2*}(d + \mathrm{d}\mathsf{g}_{1} +\mathrm{d}\mathsf{f}_{2}),\bigoplus_{|I|=1}(\mathrm{d} + \mathrm{d}f)|_{B^I},\dots$, and, as mentioned  in Remark \ref{r : commutes with residue maps}, the morphisms connecting columns in \eqref{e : residue resolution 3} are alternating sums of residue maps. On the other hand, there is also a residue resolution of complexes 
    \begin{equation}\label{e : residue resolution 4}
\begin{array}{ll}
0 \longrightarrow \Omega_X^\bullet(\log D,E^{[r]}\cap B)(*P) & \longrightarrow \bigoplus_{|I| =1}\Omega_{B^I}^{\bullet-1}(\log \partial B^I, E^{[r]}\cap B^I)(*P) \\ & \longrightarrow \bigoplus_{|I| = 2}\Omega_{B^I}^{\bullet-2}(\log \partial B^I, E^{[r]}\cap B^I)(*P) \longrightarrow \dots 
\end{array}
\end{equation}
Therefore, we obtain a quasi-isomorphism between $\mathrm{R}\pi_{2*} \mathrm{DR}_{\mathbb{E}^{[r]}_2\cap \mathbb{B}_2}(\mathbb{L}_2\setminus \mathbb{D}, \mathsf{g}_1 + \mathsf{f}_2)$ and $\mathrm{DR}_{E^{[r]}\cap B}(X\setminus D,f)$.\end{proof}

Note that \eqref{e : situation for the f reduction} is a consequence of the following proposition.
\begin{proposition}\label{p : reducing local cohomology to the base}
    There is a filtered isomorphism,
    \[
    H^*_{\mathbb{E}^{[r]}_{2}}({\mathbb{L}}_{2} \setminus {\mathbb{D}_{2}}, \mathsf{g}_{1} + \mathsf{f}_{2}) \cong H^*_{B\cap E^{[r]}}(X\setminus D,f).
    \]
\end{proposition}
\begin{proof}
    Let $\mathbb{B}_2 = \pi_2^{-1}B$ as before. Because $B$ is the vanishing locus of $\sigma$, and $\sigma$ is a section of $L^2$, $\mathbb{L}_{2}\setminus (\mathbb{D}_{2} \cup \mathbb{B}_{2}) \cong (X\setminus D) \times \mathbb{A}^1$ and under this identification, $\mathsf{g}_{1}$ is projection onto the second coordinate. By the K\"unneth formula (Theorem \ref{t : Kuenneth theorem}) and the fact that $H^*(\mathbb{A}^1,t) = 0$, we see that $H^*_{\mathbb{E}_{2}}({\mathbb{L}}_{2}\setminus ({\mathbb{D}}_{2}\cup \mathbb{B}_{2}),\mathsf{g}_{1} + \mathsf{f}_2) \cong 0$. Therefore, the long exact sequence in cohomology \eqref{e : local cohomology les} tells us there is an isomorphism between 
    \begin{equation}\label{e : isom in cohomology}
    H^*_{\mathbb{E}^{[r]}_{2}}({\mathbb{L}}_{2}\setminus {\mathbb{D}}_{2}, \mathsf{g}_{1} + \mathsf{f}_{2})\cong H^*_{\mathbb{E}^{[r]}_{2}\cap \mathbb{B}_{2}}({\mathbb{L}}_{2}\setminus {\mathbb{D}}_{2}, \mathsf{g}_{1}+\mathsf{f}_{2}).
    \end{equation}
    The proposition then follows from Lemma \ref{l : two isomorphisms}.
\end{proof}

Now we use this to decompose the cohomology of the LG model $(\mathbb{L}_{1}\setminus \mathbb{D}_{1}, \mathsf{g}_{2} + \mathsf{f}_{2})$. There is a direct sum decomposition
\[
H_{\mathbb{E}_{1}^{[r]}}^*(\mathbb{L}_{1}\setminus \mathbb{D}
_{1},\mathsf{g}_{2} +\mathsf{f}_{1}) \cong H_{\mathbb{E}_{1}^{[r]}}^*(\mathbb{L}_{1} \setminus  \mathbb{D}_{1},\mathsf{g}_{2}+\mathsf{f}_{1})^{(+)} \oplus H_{\mathbb{E}_{1}^{[r]}}^*(\mathbb{L}_{1} \setminus  \mathbb{D}_{1},\mathsf{g}_{2}+\mathsf{f}_{1})^{(-)}.
\]
By Proposition \ref{p : invariant cohomology and covers of LG models} (using the double covering map $\mathbb{L}_1 \setminus \mathbb{D}_1 \rightarrow \mathbb{L}_2 \setminus \mathbb{D}_2$) and Proposition \ref{p : reducing local cohomology to the base} this becomes 
\begin{equation}\label{e : eigenspace decomposition}
H_{\mathbb{E}_{1}^{[r]}}^*(\mathbb{L}_{1}\setminus \mathbb{D}
_{1},\mathsf{g}_{2} +\mathsf{f}_{1}) \cong H^*_{B \cap E^{[r]}}(X\setminus D,f) \oplus H_{\mathbb{E}_{1}^{[r]}}^*(\mathbb{L}_{1} \setminus  \mathbb{D}_{1},\mathsf{g}_{2}+\mathsf{f}_{1})^{(-)}.
\end{equation}
We provide an alternative description of the second summand of the right hand side of \eqref{e : eigenspace decomposition}. There is a commutative diagram, \eqref{e : commutative diagram from pullback} below, whose top horizontal arrow is an isomorphism by Proposition \ref{p : reducing local cohomology to the base} and the right vertical morphism is injective by Proposition \ref{p : invariant cohomology and covers of LG models}.
\begin{equation}\label{e : commutative diagram from pullback}
\begin{tikzcd}
    H_{\mathbb{E}^{[r]}_{2} \cap \mathbb{B}_{2}}^{*-1}({\mathbb{L}}_{2}\setminus {\mathbb{D}}_{2},\mathsf{g}_1 + \mathsf{f}_{2}) \ar[r,"\cong"]\ar[d,"\rho^*"]& H_{\mathbb{E}^{[r]}_{2}}^*(\mathbb{L}_{2}\setminus \mathbb{D}_{2},\mathsf{g}_{1} + \mathsf{f}_{2}) \ar[d,hookrightarrow,"p"] \\
    H^{*-1}_{\mathbb{E}_{1}^{[r]}\cap \mathbb{B}_1}(\mathbb{L}_{1}\setminus \mathbb{D}_{1},\mathsf{g}_{2} +  \mathsf{f}_{1}) \ar[r,"i_*"] & H_{\mathbb{E}_{1}^{[r]}}^*( \mathbb{L}_{1}\setminus \mathbb{D}_{1},\mathsf{g}_{2} + \mathsf{f}_{1}).
\end{tikzcd}
\end{equation}
By commutativity, $\rho^*$ is injective, so by Lemma \ref{l : two isomorphisms}, it must be an isomorphism.  Therefore, the image of $i_*$ in \eqref{e : commutative diagram from pullback} is identified with the image of $p$ which, by Proposition \ref{p : invariant cohomology and covers of LG models}, is $H^*_{\mathbb{E}_{1}^{[r]}}(\mathbb{L}_{1}\setminus \mathbb{D}_{1},\mathsf{g}_{2} +\mathsf{f}_{1})^{(+)}$. From \eqref{e : local cohomology les} we have a long exact sequence
\begin{align}\label{e : les that splits}
\dots \longrightarrow H^{*-1}_{\mathbb{E}^{[r]}_{1}\cap \mathbb{B}_{1}}(\mathbb{L}_{1}\setminus \mathbb{D}_{1},\mathsf{g}_{2} +  \mathsf{f}_{1}) &\xrightarrow{\,\, i_*\,\,}  H_{\mathbb{E}_{1}^{[r]}}^*( \mathbb{L}_{1}\setminus \mathbb{D}_{1}, \mathsf{g}_{2} + \mathsf{f}_{1}) \\ &\longrightarrow H_{\mathbb{E}_{1}^{[r]}}^*( \mathbb{L}_{1}\setminus (\mathbb{D}_{1}\cup \mathbb{B}_{1}),\mathsf{g}_{2} + \mathsf{f}_{1}) \longrightarrow \dots 
\end{align}
Since $i_*$ is injective the long exact sequence \eqref{e : les that splits} splits into a collection of short exact sequences:
    \begin{align*}\label{e : les that splits}
0 \longrightarrow H^{*-1}_{\mathbb{E}^{[r]}_{1}\cap \mathbb{B}_{1}}(\mathbb{L}_{1}\setminus \mathbb{D}_{1},\mathsf{g}_{2} +  \mathsf{f}_{1}) &\xrightarrow{\,\, i_*\,\,}  H_{\mathbb{E}_{1}^{[r]}}^*( \mathbb{L}_{1}\setminus \mathbb{D}_{1}, \mathsf{g}_{2} + \mathsf{f}_{1}) \\ &\longrightarrow H_{\mathbb{E}_{1}^{[r]}}^*( \mathbb{L}_{1}\setminus (\mathbb{D}_{1}\cup \mathbb{B}_{1}),\mathsf{g}_{2} + \mathsf{f}_{1}) \longrightarrow 0 
\end{align*}
As above, we have a filtered splitting,
\[
H_{\mathbb{E}_{1}^{[r]}}^*( \mathbb{L}_{1}\setminus \mathbb{D}_{1}, \mathsf{g}_{2} + \mathsf{f}_{1})  \cong H_{\mathbb{E}_{1}^{[r]}}^*( \mathbb{L}_{1}\setminus \mathbb{D}_{1}, \mathsf{g}_{2} + \mathsf{f}_{1})^{(-)}  \oplus H_{\mathbb{E}_{1}^{[r]}}^*( \mathbb{L}_{1}\setminus \mathbb{D}_{1}, \mathsf{g}_{2} + \mathsf{f}_{1})^{(+)}. 
\]
Since the image of $i_*$ is $H_{\mathbb{E}_{1}^{[r]}}^*( \mathbb{L}_{1}\setminus \mathbb{D}_{1}, \mathsf{g}_{2} + \mathsf{f}_{1})^{(+)}$, there is a filtered isomorphism
\begin{equation}\label{e : remove divisor}
H_{\mathbb{E}_{1}^{[r]}}^{*}( \mathbb{L}_{1}\setminus (\mathbb{D}_{1}\cup \mathbb{B}_{1}),\mathsf{g}_{2} + \mathsf{f}_{1}) \cong H_{\mathbb{E}_{1}^{[r]}}^*(\mathbb{L}_{1}\setminus  \mathbb{D}_{1},\mathsf{g}_{2}+\mathsf{f}_{1})^{(-)}.
\end{equation}

We will now construct a morphisms that will allow us to identify the left hand side of \eqref{e : remove divisor} with  $H^{*-1}_{\widehat{E}^{[r]}}(\widehat{X}\setminus \widehat{D}, \widehat{f})^{(-)}(1/2)$. This will complete the proof. Given the ramified double cover $\eta: \widehat{X} \rightarrow X$, there is a cartesian diagram
\[
\begin{tikzcd}
    \widehat{\mathbb{L}}_{1} \ar[r,"\eta"] \ar[d,"\widehat{\pi}_{1}"] & \mathbb{L}_{1} \ar[d,"\pi_{1}"] \\
    \widehat{X} \ar[r,"\rho"] & X
\end{tikzcd}
\]
where $\widehat{\mathbb{L}}_1$ is $\mathrm{Tot}(\rho^*L)$. Let $\widehat{\mathbb{E}}_{1}^{[r]} = \rho^{-1}\mathbb{E}_{1}^{[r]}$ and let ${\mathbb{R}}_{1} = \rho^{-1}\mathbb{B}_{1}$. By Proposition \ref{p : invariant cohomology and covers of LG models} there is a filtered isomorphism 
\[
H^*_{\widehat{\mathbb{E}}_{1}^{[r]}}(\widehat{\mathbb{L}}_{1}\setminus (\widehat{\mathbb{D}}_{1}\cup {\mathbb{R}}_{1}),\widehat{\mathsf{g}}_{2}+\mathsf{f}_{1})^{(+)}\cong H^*_{\mathbb{E}_{1}^{[r]}}(\mathbb{L}_{1}\setminus (\mathbb{D}_{1}\cup \mathbb{B}_{1}),\mathsf{g}_{2} + \mathsf{f}_{1}).
\]

\begin{remark}[On singular ramification divisors]\label{r : the double cover might not be smooth}
    The variety $\widehat{X}$ is not smooth in general. It acquires singularities in the preimages of singular points of $R$. Therefore $\widehat{\mathbb{L}}$ is not smooth either, in general. We deal with this by repeatedly blowing up $X$ in $R$ until the proper transform of $R$ is smooth. Since we are only concerned with the cohomology of the complement of ${\mathbb{R}}_{1}$. To avoid complicating notation, we will ignore this blow up below. 
\end{remark}

Let $\tau:  \widehat{X}\rightarrow \widehat{X}$ be the automorphism so that $\widehat{X}/\tau \cong X$. The final step in the proof of Theorem \ref{t : multiple direct sum decomposition} is to show that 
\begin{equation}\label{e : final isomorphism}
H^*_{\widehat{\mathbb{E}}_{1}^{[r]}}(\widehat{\mathbb{L}}_{1}\setminus (\widehat{\mathbb{D}}_{1}\cup  {\mathbb{R}}_{1}),\widehat{\mathsf{g}}_{2}+\mathsf{f}_{1})^{(+)}\cong H^{*-1}(\widehat{X}\setminus (\widehat{D}\cup R),f)^{(-)}(1/2).
\end{equation}
This is deduced from the following geometric calculation and the K\"unneth formula. Details appear below.
\begin{proposition}
    There is a commutative diagram
    \[
    \begin{tikzcd}
    \widehat{\mathbb{L}}_{1} \setminus (\widehat{\mathbb{D}}_{1} \cup {\mathbb{R}}_{1}) \ar[r, "\widehat{\phi}"] \ar[d,"\eta"] & \mathbb{A}^1 \times (\widehat{X}\setminus (\widehat{D} \cup {R}))   \ar[d,"\mu"] \\
    {\mathbb{L}}_{1} \setminus ({\mathbb{D}}_{1} \cup {\mathbb{R}}_{1}) \ar[r, "{\phi}"] & \mathbb{A}^1 \times (X\setminus (D \cup {R}))  
    \end{tikzcd}
    \]
    where $\widehat{\phi}$ and $\phi$ are isomorphisms. The following properties hold.
    \begin{enumerate}
        \item If $q :(\widehat{X}\setminus (\widehat{D} \cup {R})) \times \mathbb{A}^1\rightarrow \mathbb{A}^1$ is the regular function obtained by projection onto the second coordinate, $\phi^*(q^2) = \widehat{\mathsf{g}}_2$.
        \item The map $\mu$ is the quotient by the automorphism $\widehat{\tau}: (q,z)\mapsto (-q,\tau(z))$.
    \end{enumerate}
\end{proposition}
\begin{proof}
    On $\widehat{\mathbb{L}}_1 \setminus (\widehat{\mathbb{D}}_1\cup {\mathbb{R}}_1)$, both $\widehat{\pi}_1^*\rho^*L^{- 2}$ and $\widehat{\pi}_1^*\rho^*L^{-1}$ are trivial bundles because $L^2$ and $L$ are trivial over $X\setminus D$. This means that $\hat{s}$ and  $\widehat{\pi}_1^*\sigma$ are regular functions on $\widehat{\mathbb{L}}_1 \setminus (\widehat{\mathbb{D}}_1 \cup \mathbb{R}_1)$. Now we explain that $\widehat{\pi}_1^*\sigma$ has square root $\hat{y}$. To see this, recall that $\sigma$ is a regular function on $X\setminus B$. Moreover, abusing notation to consider $\sigma$ and $y$ as functions on $\mathbb{A}^1_y \times X\setminus B$, $\widehat{X} \setminus R$ can be written as the hypersurface $y^2 = \sigma$ in $\mathbb{A}^1_y \times (X\setminus B)$ and $\sigma|_{\widehat{X}\setminus R} = \rho^*\sigma$. Therefore, $\rho^*\sigma$ has square root $y$, and we may let $\hat{y} = \widehat{\pi}_1^*y$. Therefore, $\widehat{\mathsf{g}}_2 = \rho^*\mathsf{g}_2 = \rho^*(s^2\cdot \pi_1^*\sigma) = (\rho^*s\cdot \hat{y})^2$. Note that $\widehat{y}(\tau(z)) = - \widehat{y}(z)$.
    
    We may define an isomorphism 
    \[
    \widehat{\phi}: \widehat{\mathbb{L}}_{1} \setminus (\widehat{\mathbb{D}}_{1}\cup {\mathbb{R}}_{1}) \longrightarrow \mathbb{A}^1 \times (\widehat{X}\setminus (\widehat{D}\cup {R}))  ,\quad x \longmapsto (\hat{s}(x)\cdot \hat{y}(x), \widehat{\pi}_1(x)).
    \]
    The function $\hat{y}$ is nonvanishing which means that $\widehat{\phi}$ is an isomorphism. Furthermore, $\phi^*(q^2) = \widehat{\mathsf{g}}_2$, which is statement (1) in the proposition.

    The trivialization $\phi$ is constructed similarly. If $s$ denotes the tautological section of $\pi_1^*L_1$ on $\mathbb{L}_1$ then $s$ is a regular function on $\mathbb{L}_1 \setminus \mathbb{D}_1$ and we may define
    \[
    {\phi}: {\mathbb{L}}_{1} \setminus ({\mathbb{D}}_{1}\cup {\mathbb{R}}_{1}) \longrightarrow \mathbb{A}^1 \times ({X}\setminus ({D}\cup {R}))  ,\quad x \longmapsto ({s}(x), {\pi}_1(x)).
    \]
    We obtain the following diagram
    \[
    \begin{tikzcd}
    \widehat{\mathbb{L}}_{1} \setminus (\widehat{\mathbb{D}}_{1} \cup {\mathbb{R}}_{1})  \ar[d,"\eta"] & \mathbb{A}^1 \times (\widehat{X}\setminus (\widehat{D} \cup {R}))\ar[l, swap, "\widehat{\phi}^{-1}"]  \\
    {\mathbb{L}}_{1} \setminus ({\mathbb{D}}_{1} \cup {\mathbb{R}}_{1}) \ar[r, "{\phi}"] & \mathbb{A}^1 \times (X\setminus (D \cup {R}))  
    \end{tikzcd}
    \]
    The composition $\mu:= \phi \cdot \eta \cdot \widehat{\phi}^{-1}$ is equal to the map
    \[
    (q,z) \longmapsto (q/\widehat{y}(z), \rho(z)).
    \]
    By construction, $\mu(q_1,z_1) = \mu(q_2,z_2)$ if and only if either $z_1 = z_2$ and $q_1=q_2$, or  $z_1 = \tau(z_2)$ and $q_1 = -q_2$ (because $\widehat{y}(\tau(z)) = - \widehat{y}(z)$). This proves the second statement in the proposition.
\end{proof}
We apply the K\"unneth formula for the irregular filtered twisted cohomology (Theorem \ref{t : Kuenneth theorem}) to deduce \eqref{e : final isomorphism}, which will complete the proof of Theorem \ref{t : single direct sum decomposition}. The K\"unneth formula can be interpreted as usual by taking the external tensor product of twisted de Rham cohomology classes. Specializing Example \ref{e : t^n LG model} we see that the classes $[\alpha\cdot \mathrm{d}q]$, $\alpha\in\mathbb{C}$ give $\gr^{1/2}_{F_\mathrm{irr}}H^1(\mathbb{A}^1_q,q^2) \cong H^1(\mathbb{A}^1_q,q^2)$, so the invariant cohomology classes with respect to $\widehat{\tau}$ are $[\omega] \boxtimes [\alpha \cdot \mathrm{d}q] \in H_{\widehat{E}^{[r]}}^*(\widehat{X}\setminus \widehat{D},\widehat{f}) \boxtimes H^{1}(\mathbb{A}^1,q^2)$ where $[\omega] \in H_{\widehat{E}^{[r]}}^*(\widehat{X}\setminus \widehat{D},\widehat{f})^{(-)}$. Therefore,
\begin{align*}
H_{\widehat{\mathbb{E}}_{1}^{[r]}}^*(\widehat{\mathbb{L}}_{1} \setminus (\widehat{\mathbb{D}}_{1}\cup {\mathbb{B}}_{1}), \mathsf{g}_2 + \mathsf{f}_{1})^{(+)} & \cong (H_{\widehat{E}^{[r]}}^{*-1}(\widehat{X}\setminus \widehat{D},\widehat{f}) \boxtimes H^{1}(\mathbb{A}^1,q^2))^{(+)} \\
& \cong H_{\widehat{E}^{[r]}}^{*-1}(\widehat{X}\setminus \widehat{D},\widehat{f})^{(-)}(1/2).
\end{align*}
This completes the proof of Theorem \ref{t : single direct sum decomposition}.

\begin{remark}
    Theorem \ref{t : single direct sum decomposition} can likely be extended to a quasi-isomorphism of filtered complexes of  sheaves. However the discussion following Proposition \ref{p : reducing local cohomology to the base} is completely cohomological, so our proof is only valid on the level of filtered vector spaces. 
\end{remark}

\section{Hodge number duality for Clarke mirror pairs}\label{s:toric}
In this section, we recall background on toric varieties, Clarke duality, and one of the main results of our previous work \cite{HL2025}. We first recollect some background about toric varieties to set up notation. 

Let $N$ and $M$ be dual lattices of rank $d$ with the natural bilinear pairing $\langle -,- \rangle:N \times M \to \mathbb{Z}$. We write $N_\mathbb{K}:=N \otimes_\mathbb{Z} \mathbb{K}$ and $M_\mathbb{K}=M \otimes_\mathbb{Z} \mathbb{K}$ for $\mathbb{K}=\mathbb{Q}, \mathbb{R}, \mathbb{C}$. A fan $\Sigma \in N_\mathbb{R}$ is a collection of strongly convex polyhedral cones such that each face of a cone in $\Sigma$ is also a cone in $\Sigma$, and the intersection of two cones in $\Sigma$ is a face of each cone.

Let $\Sigma[1]=\{\rho_i\mid i=1, \dots, n\}$ be the set of primitive generators of $\Sigma$. Consider the monomial ideal of $\mathbb{C}[x_1, \cdots, x_n]$, $J_\Sigma:=\langle \prod_{\rho_i \subsetneq c} x_i \mid c \in \Sigma \rangle$ and the induced quasi-affine variety $\mathbb{C}^n \setminus V(J_\Sigma)$. We also have the morphism of the lattices $\beta:\mathbb{Z}^n \to N$ that sends the standard basis $e_i$ to $\rho_i$, which induces the morphism of tori $T_{\beta}:(\mathbb{C}^*)^n \to (\mathbb{C}^*)^d$. We let $G_{\beta}$ be its kernel.
\begin{defn}
    A toric variety $T(\Sigma)$ is the quotient 
    \[
    (\mathbb{C}^n \setminus V(J_\Sigma))/ G_\beta
    \]
     where $G_{\beta}$ acts freely via the action of $(\mathbb{C}^*)^n$.
\end{defn}

Any toric variety is stratified by tori $T_c$ of dimension $(d-\dim c)$ corresponding to the cone $c \in \Sigma$. We let $T(\Sigma)_c$ denote the closure of $T_c$ in $T(\Sigma)$. In particular, each one dimensional cone $c$, or its ray generator $\rho$ determines the torus-invariant divisor, which we also denote by $E_\rho = T(\Sigma)_c$. 

\begin{example}\label{eg:cayley}
     Suppose we have a line bundle $L = \mathcal{O}_{T(\Sigma)}(-\sum_{i=1}^n a_iE_i)$ where $E_1,\dots, E_n$ are toric boundary divisors of $T(\Sigma)$ and $a_i \in \mathbb{N}$. For each cone $c$ of $\Sigma$ there is a cone
\begin{equation*}
\widetilde{c} = \mathrm{Cone}(\{(\rho_i, a_i) \mid \rho_i \in c[1]\} \cup (0,1)) \in \Sigma_L.
\end{equation*}
Then $\mathrm{Tot}(L)$ is a toric variety whose fan, denoted $\Sigma_L$, is the union of the cones $\widetilde{c}$ and their faces.
\end{example}

Below, we list some properties of the fans and their relations to the associated toric variety:
\begin{enumerate}
    \item A cone $c$ is called \emph{unimodular} if the primitive integral ray generators of $c$ form a basis of $N$. A fan $\Sigma$ is \emph{unimodular} if every maximal cone is unimodular. Then, the associated toric variety $T(\Sigma)$ is a manifold.  
    \item A cone $c$ is called \emph{simplicial} if the primitive integral ray generators of $c$ form a basis of $N$. A fan $\Sigma$ is \emph{simplicial} if every maximal cone is simplicial. Then, the associated toric variety $T(\Sigma)$ is an orbifold.  
    \item A cone $c$ in $N_\mathbb{R}$ is called \emph{Gorenstein} if there is some $m_c\in M$ so that the integral collection of points $c \cap \{n \mid \langle n, m_c \rangle = 1\}$ generates the cone $c$. A fan $\Sigma$ is \emph{Gorenstein} if all cones are Gorenstein. Then, the associated toric variety $T(\Sigma)$ is Gorenstein. 
    \item A fan $\Sigma$ is \emph{quasiprojective} if there is a convex function on $\mathrm{Supp}(\Sigma)$ which is linear on each cone of $\Sigma$ and takes integral values on $N \cap \mathrm{Supp}(\Sigma)$. Then, the associated toric variety $T(\Sigma)$ is quasiprojective. 
\end{enumerate}

Suppose that we have a finite collection of integral points $A$ in $M$. We consider a Laurent polynomial $w \in \mathbb{C}[M]$ of the form 
\[
w=\sum_{m \in A}u_m\underline{x}^m
\]
for some $u_m \in \mathbb{C}^*$.
For any toric variety $T(\Sigma)$, one can see that $w$ becomes a regular function on $T(\Sigma)$ if $\langle m, \rho \rangle \geq 0$ for every primitive integral ray generator $\rho \in \Sigma[1] $ and $m \in A$.


For later use, we also define the notion of toric Deligne--Mumford stack introduced in \cite{borisov2005orbifold}. A \emph{stacky fan} is a simplicial fan $\Sigma$ equipped with additional data: a positive integer $\beta_\rho$ assigned to each primitive integral ray generator $\rho \in \Sigma[1]$. We denote this as ${\bf \Sigma} = (\Sigma, \beta)$, where $\Sigma$ is referred to as the underlying fan of ${\bf \Sigma}$. The assignment $\beta$ can be viewed as a lattice morphism $\beta:\mathbb{Z}^n \to N$ that sends the standard basis $e_i$ to $\beta_\rho \cdot \rho$. Similar to the global construction of toric varieties, it induces the morphism of tori $T_{\beta}:(\mathbb{C}^*)^n \to (\mathbb{C}^*)^d$. Since $\beta$ has finite cokernel, $T_{\beta}$ is surjective. Let $G_{\beta}$ be its kernel. 

\begin{defnprop}
    A toric Deligne--Mumford (DM) stack $T(\mathbf{\Sigma})$ is the quotient stack 
    \[
    \left[(\mathbb{C}^n \setminus V(J_\Sigma))/G_{\beta}\right]
    \]
    where $G_{\beta}$ acts via the action of $(\mathbb{C}^*)^n$.
\end{defnprop} 

The underlying toric variety $T(\Sigma)$ is the coarse moduli space for $T({\bf \Sigma})$ so that we have a canonical morphism $\pi_{\bf\Sigma}:T({\bf \Sigma}) \to T(\Sigma)$. Any simplicial fan $\Sigma$ has a canonical stacky fan structure where $\beta_\rho = 1$ for all $\rho \in \Sigma[1]$. For a stacky fan ${\bf \Sigma}$, let ${\bf \Sigma}[1]$ denote the set $\{\beta_\rho \rho \mid \rho \in \Sigma[1]\}$. We often write $c \in {\bf \Sigma}$ which means that we take ray generators of the cone $c$ to be extended ones $\{\beta_\rho \rho \mid \rho \in c\}$.   
\begin{defn}
    We say that ${\bf \Sigma}$ is \emph{convex} if the polyhedral complex 
     \[
        \Delta_{\bf\Sigma} = \bigcup_{{c}\in \bf{\Sigma}}\mathrm{Conv}(c[1] \cup \{0\})
    \]
    has convex support. 
\end{defn}

\begin{example}\label{eg:P1}
    Let $\Sigma$ be a fan for $\mathbb{P}^1$ with ray generator $\rho_1=(1,0)$ and $\rho_2=(-1,0)$ and $E_1$ and $ E_2$ be corresponding toric divisors, respectively. As in Example \ref{eg:cayley}, consider a line bundle $L=\mathcal{O}(-2E_1-2E_2)$. The corresponding fan $\Sigma_{L}$ is generated by $\tilde{\rho}_0=(0,1), \tilde{\rho}_1=(1,2), \tilde{\rho}_2(-1,2)$. Note that  this fan is not convex. 
    
    On the other hand, we can introduce a stacky structure on $\Sigma_L$ to get a convex fan as follows: define 
    \[
    \beta(\tilde{\rho}_i)=\begin{cases}
        2 & i=0 \\
        1 & i=1,2
    \end{cases}.
    \]
    Then ${\bf \Sigma}_L=(\Sigma_L, \beta)$ becomes a convex stacky fan and the corresponding toric DM stack is the root stack of $\mathrm{Tot}(L)$ along the zero section, which we denote $\sqrt{\mathrm{Tot}(L)/0_{T(\Sigma)}}$. 
\end{example}
Some of the properties of a fan naturally extend to a stacky fan as properties of the underlying fan: ${\bf \Sigma}$ is called unimodular, simplicial, or quasiprojective if its underlying fan $\Sigma$ has these properties. An exception is the Gorenstein property, where we say ${\bf \Sigma}$ is Gorenstein if every cone $c$ in ${\bf \Sigma}$ is Gorenstein.

\begin{defn}\label{d:adjectives}
    Let ${\bf \Sigma} \subseteq N_\mathbb{R}, {\bf \check{\Sigma}}\subseteq M_\mathbb{R}$ be a pair of quasiprojective, simplicial stacky fans. We say that $({\bf \Sigma}, {\bf \check{\Sigma}})$ is \emph{Clarke dual} if the following two conditions hold:
    \begin{enumerate}
        \item (Regularity) $\langle n, m \rangle \geq 0 \text{ for all } n \in \mathrm{Supp}({\bf{\Sigma}}), m \in \mathrm{Supp}(\bf{\check{{\Sigma}}}).$
        \item (Convexity) Both ${\bf \Sigma}$ and ${\bf \check{\Sigma}}$ are convex. 
    \end{enumerate}
\end{defn}

\begin{example}\label{ex : reflexive polytope}
    A convex polytope $\Delta \subseteq N_\mathbb{R}$ is {\em reflexive} if its vertices are located at points in $N$, and its polar dual,
    \[
    \Delta^\circ = \{ m \in M_\mathbb{R} \mid m(n) \geq -1  \, \, \forall \,\, n \in \Delta\}
    \]
    also has vertices located at points in $M$. Let $C = \mathrm{cone}(\Delta \times 1) \subseteq N_\mathbb{R}\times \mathbb{R}$ and $\check{C} = \mathrm{cone}(\Delta^\circ \times 1) \subseteq M_\mathbb{R}\times \mathbb{R}$. By results of Batyrev \cite{batyrev1994dual2} we may choose simplicial quasiprojective fans $\Sigma, \check{\Delta}$ refining $C$ and $\check{C}$ whose ray generators are points in $(\Delta \times 1) \cap (N\times 1)$ and $(\Delta^\circ \times 1) \cap( M\times 1)$ respectively. Then $\Sigma$ and $\check{\Sigma}$ form a Clarke dual pair. 
\end{example}

Given a Clarke dual pair of stacky fans, the regularity condition ensures that a Laurent polynomial
\[
w({\bf \Sigma}) = 1 + \sum_{n \in {\bf \Sigma}[1]} u_n \underline{x}^n \in \mathbb{C}[N]
\]
where $u_n \in \mathbb{C}^*$ are chosen generically,
defines a regular function on $T(\check{\Sigma})$. By composing it with the canonical morphism $\pi_{\mathbf{\Sigma}}$, we obtain a regular function on $T(\check{\bf \Sigma})$ which we also write as $w({\bf \Sigma})$ by abuse of notation. Thus, we have an induced pair of LG models
\[
(T({\bf \Sigma}),w(\check{\bf \Sigma})),\qquad (T(\check{\bf \Sigma}), w({\bf \Sigma}))
\]
which is called a \emph{Clarke mirror pair of stacky LG models}. In \cite{HL2025}, we have justified this terminology, particularly the term ``mirror", by verifying an irregular version of Hodge number duality (see Theorem \ref{t:Cdual}). To state this, we briefly recall the notion of orbifold cohomology \cite{ChenRuan2004orb}. 

Let $Y$ be a toric DM stack (or an orbifold). The orbifold cohomology of $Y$ is defined as the usual cohomology of the inertia stack of $Y$ with a particular twist. Each component of the inertia stack is referred to as a \emph{twisted sector}, and the twist assigned to each component is called its \emph{age}. To simplify the discussion, we only present  a description for the toric case, which suffices for the purposes of this article.  

Let ${\bf \Sigma}\subset N$ be a simplicial stacky fan and $T({\bf \Sigma})$ be the associated toric DM stack. For each cone ${c}$ in $\bf{\Sigma}$, we let
\begin{equation*}
    \mathrm{Box}^\circ({c}) := \left\{ \left. \sum_{\rho\in {c}[1]} a_\rho \rho \,\, \right| \,\, a_\rho \in (0,1) \right\}\cap N.
\end{equation*}
\begin{theorem}[{\cite[Theorem 1]{poddar2003orbifold}},{\cite[Proposition 5.2]{borisov2005orbifold}}]\label{t:poddar} 
    The components of the inertia stack of $T({\bf \Sigma})$ are parametrized by the union of $\mathrm{Box}^\circ({c})$ over all $c \in {\bf \Sigma}$. For each $g=\sum a_\rho \rho \in \mathrm{Box}^\circ({c})$, the corresponding twisted sector is the closed substack $T({\bf{\Sigma}})_c$, whose coarse moduli space is the closed torus orbit of $T(\Sigma)$ associated with the cone $c$. The orbifold cohomology of $T(\mathbf{\Sigma})$ is given by 
\[
H_\mathrm{orb}^*(T(\mathbf{\Sigma}))=\bigoplus_{c \in {\bf\Sigma}} \bigoplus_{g \in \mathrm{Box}^\circ({c})} H^{*-2\iota(g)}(T({\bf\Sigma})_c)(-\iota(g)) 
\]
where $\iota(g)=\sum a_i$. 
\end{theorem}

We note that the age grading $\iota(g)$ is rational in general. When $\mathbf{\Sigma}$ has a trivial stacky structure, equivalently $\mathbf{\Sigma}=\Sigma$, then $\iota(g) \in \mathbb{Z}$ if and only if $\Sigma$ is Gorenstein. 

Analogously, for a LG model $(T({\bf \Sigma}),w)$, we can naturally define the orbifold twisted cohomology and the orbifold irregular Hodge filtration. We denote the orbifold twisted cohomology by $H^*_{\mathrm{orb}}(T({\bf \Sigma}), w)$, and the associated graded pieces of the irregular Hodge filtration is given by 
\[
H_\mathrm{orb}^{\lambda, \mu}(T(\mathbf{\Sigma}), w)=\bigoplus_{c \in {\bf\Sigma}} \bigoplus_{g \in \mathrm{Box}^\circ({c})} H^{\lambda-\iota(g),\mu-\iota(g)}(T({\bf\Sigma})_c, w) 
\]
for $\lambda, \mu \in \mathbb{Q}$. When the rational gradings happen to be integral, we adopt the more familiar notation $(p,q)$ in place of $(\lambda, \mu)$ to denote Hodge components. Therefore, the corresponding orbifold irregular Hodge numbers\footnote{In \cite{HL2025}, we used the letter $f$ to denote Hodge numbers in order to distinguish them from the Deligne--Hodge numbers. However, since we primarily work with smooth and proper varieties in this article, such a distinction is unnecessary, and we will use the letter $h$ throughout the article.} are
\[
h^{\lambda,\mu}_{\mathrm{orb}}(T(\mathbf{\Sigma}), w)=\sum_{c \in {\bf \Sigma}}\sum_{g \in \mathrm{Box}^\circ(c)}h^{\lambda-\iota(g), \mu-\iota(g)}(T({\bf \Sigma})_c, w). 
\]

\begin{theorem}\label{t:Cdual}
Let $({\bf\Sigma}, {\bf \check{\Sigma}})$ be a Clarke dual pair. For $\lambda,\mu \in \mathbb{Q}$, we have the identification of the orbifold irregular Hodge numbers 
 \[
 h_\mathrm{orb}^{\lambda,\mu}(T({\bf \Sigma}), w({\bf \check{\Sigma}}))=h_\mathrm{orb}^{d-\lambda,\mu}(T({\bf \check{\Sigma}}), w(\bf {\Sigma})).
 \]
\end{theorem}

\begin{example}
    Let ${\bf \Sigma}_L$ be the one introduced in Example \ref{eg:P1}. Let $\check{\Sigma}$ be a fan in the dual lattice $M$ corresponding to $\mathbb{P}^1$ as well. Denote the ray generators as $\check{\rho}_1=(1,0)$ and $\check{\rho}_2=(-1,0)$ whose associated toric divisor is denoted by $\check{E}_1$ and $\check{E}_2$, respectively. Consider the line bundle ${P}=\mathcal{O}(-\check{E}_1-\check{E}_2)$ and let $\check{\Sigma}_{P}$ be the corresponding fan. We impose the trivial stacky structure on $\check{\Sigma}_{P}$ and write the resulting stacky fan as ${\bf \check{\Sigma}}_{P}$. It is straightforward to check that $({\bf \Sigma}_L,{\bf \check{\Sigma}}_{P})$ form a Clarke dual pair. Also, the induced Clarke mirror pair of LG models is given by 
    \[
    (T({\bf {\Sigma}}_{L}), w({\bf \check{\Sigma}}_{P})), \qquad  (T({\bf \check{\Sigma}}_{P}), w({\bf {\Sigma}}_{L}))
    \]
    where 
    \begin{itemize}
        \item $T({\bf {\Sigma}}_{L})=\sqrt{\mathrm{Tot}(L)/0_{T(\Sigma)}}$, $w({\bf \check{\Sigma}}_{P})=(x+x^{-1}+1)y$.
        \item $T({\bf \check{\Sigma}}_{P})=\mathrm{Tot}(P)$, $w({\bf {\Sigma}}_{L})=(\check{x}+\check{x}^{-1}+1)\check{y}^2$.
    \end{itemize}
    Here we choose all the coefficients to be $1$ and describe each regular function as a Laurent polynomial where $x,y$ (resp. $\check{x}, \check{y}$) are the chosen coordinates. More invariantly, if we write $\sigma_{gen}$ (resp. $ \sigma_{tor}$) for a generic (resp. toric) section of $\mathcal{O}_{\mathbb{P}^1}(2)$, then $w({\bf \check{\Sigma}}_{P})=\sigma_{gen}\sigma_{tor}t$ where $t$ is the fiber coordinate of the line bundle $L \to T(\Sigma)$. The parallel argument allows to write the regular function $w({\bf \Sigma}_{L})$ as $\check{\sigma}_{gen}\check{\sigma}_{tor}{\check{t}^2}$. 

    Next, we compute the orbifold Hodge numbers. Denote $D=D_{tor} \cup D_{gen}$ where $D_{tor}=\{\sigma_{tor}=0\}$ and $D_{gen}=\{\sigma_{gen}=0\}$. In this case, $D$ consists of $4$ points. The parallel description is applied to the Clarke dual part. For $\lambda, \mu \in \mathbb{Q}$, 
    \[\begin{aligned}
        H^{\lambda, \mu}_{\mathrm{orb}}(T({\bf {\Sigma}}_{L}), w({\bf \check{\Sigma}}_{P}))&=H^{\lambda, \mu}(\mathrm{Tot}(L), \sigma_{gen}\sigma_{tor}t) \oplus H^{\lambda-\frac{1}{2}, \mu-\frac{1}{2}}(T(\Sigma)) \\
        &= H^{\lambda, \mu}_{D}(\mathbb{P}^1) \oplus H^{\lambda-\frac{1}{2}, \mu-\frac{1}{2}}(\mathbb{P}^1).
    \end{aligned}
    \] 
    On the mirror side, we have
    \[\begin{aligned}
        H^{\lambda, \mu}_{\mathrm{orb}}(T({\bf \check{\Sigma}}_{P}), w({\bf {\Sigma}}_{L}))&=H^{\lambda, \mu}(\mathrm{Tot}(P), \check{\sigma}_{gen}\check{\sigma}_{tor}\check{t}^2) \\
        &= H^{\lambda, \mu}_{D}(\mathbb{P}^1) \oplus H^{\mu-\frac{1}{2}}(\mathbb{P}^1, \Omega^{\lambda-\frac{1}{2}}_{\mathbb{P}^1}(\log \check{D})\otimes P)
    \end{aligned}
    \]
    where the second equality follows from Theorem \ref{t : single direct sum decomposition}. One can check the Hodge number duality in this case. 
\end{example}

\begin{corollary}\label{c:Cdualgoren}
    Let $(\Sigma, \check{\Sigma})$ be a Gorenstein Clarke dual pair. For $p,q \in \mathbb{Z}$, we have the identification of the orbifold Hodge numbers
    \[
    h^{p,q}_\mathrm{orb}(T(\Sigma),w(\check{\Sigma}))=h^{d-p,q}_\mathrm{orb}(T(\check{\Sigma}), w(\Sigma)).
    \]
\end{corollary}
\begin{proof}
    Since $\Sigma$ is Gorenstein, the age gradings are integers. Also, the Gorenstein condition on $\check{\Sigma}$ implies that $w(\check{\Sigma})$ is tame. Therefore $h^{\lambda, \mu}_\mathrm{orb}(T(\Sigma), w(\check{\Sigma}))=0$ for any $\lambda, \mu \in \mathbb{Q} \setminus \mathbb{Z}$. The conclusion follows from Theorem \ref{t:Cdual}.
\end{proof}

    \section{Mirror symmetry for Galois cover Calabi--Yau varieties}
    We introduce two different singular Calabi--Yau pairs associated to the nef partition data, which will be proved to satisfy Hodge number duality.
    
    Let $\Delta \subset M_{\mathbb{R}}$ be a reflexive polytope and let $\Sigma_\Delta$ denote its spanning fan. A nef partition  of $\Delta$ is a partition of the set of vertices of $\Delta$ into subsets $S_1,\dots, S_k$ so that for each $i$, there is a $\Sigma_\Delta$ linear convex function $\psi_i$ so that $\psi_i(n) = 1$ if $n \in S_i$ and $\psi_i(n) = 0$ if $n \in S_j, j\neq i$. We let $\Delta_i = \mathrm{conv}(S_i \cup 0_N)$. To indicate that $\Delta_1,\dots, \Delta_k$ form a nef partition of $\Delta,$ we will write $\Delta = \Delta_1\cup \dots \cup \Delta_k$. Given a nef partition, denote 
    \[
    \check{\Delta}_i = \{ m \in M_\mathbb{R} \mid m(n) \geq -1, \,\, \forall \,\, n \in \Delta_i, m(n) \geq 0 \,\, \forall \,\, n \in \Delta_j, j\neq i\}.
    \]
    Let $\check{\Delta} = \mathrm{conv}(\check{\Delta}_1,\dots, \check{\Delta}_k)$. Then $\check{\Delta}_1,\dots, \check{\Delta}_k$ forms a nef partition of $\check{\Delta}$. We have $\Delta^\circ = \check{\Delta}_1 + \dots + \check{\Delta}_k$ and $\check{\Delta}^\circ = \Delta_1 + \dots + \Delta_k$ where $+$ indicates Minkowski sum of polytopes\footnote{In \cite{Hosono2024doubleCY}, Hosono, Lee, Lian, and Yau let $\Sigma_\Delta$ denote the normal fan of $\Delta$, or equivalently the spanning fan of $\Delta^\circ$, rather than the spanning fan of $\Delta$. The notation used in this paper is consistent with that of our earlier work \cite{HL2025}, and was chosen for its compatibility with the Clarke mirror construction.}.
    
    As mentioned in the introduction, we always assume that the associated toric varieties $T(\Sigma_\Delta)$ and $T(\Sigma_{\check{\Delta}})$ admit MPCS resolutions, and denote them by $T_\Delta$ and $T_{\check{\Delta}}$, respectively. Note that this is equivalent to choosing projective unimodular triangulations of $\Delta$ and $\check{\Delta}$. Such triangulations can fail to exist if $\dim \Delta \geq 4$. We make this assumption to ensure that we only have to deal with singular varieties, rather than singular DM stacks.   

For each $i$, let $E_{\Delta_i}$ be the toric divisor on $T_{\Delta}$ corresponding to the chosen projective unimodular triangulation of $\Delta_i$ and $\mathcal{O}(E_{\Delta_i})$ be the corresponding line bundle. The nef partition $\Delta=\Delta_1\cup\cdots\cup\Delta_k$ induces the decomposition of the anticanonical line bundle $K^{-1}_{T_\Delta}=\bigotimes_{i=1}^k \mathcal{O}(E_{\Delta_i})$. Also, each integral point of $\check{\Delta}_i$ provides a section of $\mathcal{O}(E_{\Delta_i})$, and sections corresponding to integral points in $\check{\Delta}_i$ form a basis of $H^0(T_{\Delta},\mathcal{O}(E_{\Delta_i})$. In particular, we will use the notation $\sigma_{i,gen}$ to denote a generic section and $\sigma_{i, tor}$ will denote the section corresponding to the point $0$. The vanishing locus of $\sigma_{i,tor}$ is $E_{\Delta_i}$, and the vanishing locus of $\sigma_{i,gen}$ is a smooth hypersurface in $T_{\Delta}$. Then the product of these two sections $\sigma_i:=\sigma_{i,gen}\sigma_{i,tor}$ becomes a section of $\mathcal{O}(2E_{\Delta_i})$. We also write $D_i=\{\sigma_i=0\}, D_{i,gen}=\{\sigma_{i,gen}=0\}, D_{i,tor}=\{\sigma_{i,tor}=0\}$ for the vanishing loci and $D=D_1\cup\cdots\cup D_k$. Then by construction, these divisors are at worst simple normal crossings.

To this data, we associate two branched covers of $T_\Delta$:
    \begin{enumerate}
        \item $\pi:\widehat{T}_\Delta \to T_\Delta$, where $\widehat{T}_\Delta$ is the $(\mathbb{Z}/2)^k$-Galois cover of $T_{\Delta}$.
        \item $\pi:\widetilde{T}_\Delta \to T_\Delta$, where $\widetilde{T}_\Delta$ is the branched double cover of $T_{\Delta}$ whose branch locus is $D$.
    \end{enumerate}
    
\begin{lemma}\label{l:CY}
    Both $\widehat{T}_\Delta$ and $\widetilde{T}_\Delta$ are singular Calabi--Yau varieties with at worst orbifold singularities. 
\end{lemma}
	\begin{proof}
    The proof follows from \cite[Proposition A.3]{Hosono2024doubleCY}. This shows that $\widetilde{T}_\Delta$ is Calabi--Yau because $K^{-1}_{T_\Delta}\cong \bigotimes_{i=1}^k\mathcal{O}(E_{\Delta_i})$.  By iteratively applying the same argument from \textit{loc. cit.,} we conclude that 
  $\widehat{T}_\Delta$ is Calabi--Yau as well. 
	\end{proof}
    
  Applying the parallel construction on the mirror side, we obtain two singular Calabi--Yau pairs
    \[
    (\widehat{T}_\Delta, \widehat{T}_{\check{\Delta}}), \qquad (\widetilde{T}_{\Delta}, \widetilde{T}_{\check{\Delta}}).
    \]
    We will prove the following result in Section \ref{s : toric extremal transitions}. Since both $\widehat{T}_\Delta$ and $\widehat{T}_{\check{\Delta}}$ have at worst orbifold singularities, their cohomology groups carry pure Hodge structures. We emphasize that the Hodge numbers referred to in Theorem \ref{t:toricmirror} are usual (not orbifold) Hodge numbers. The fact that singularities are at worst orbifold follows from the fact that $\widehat{T}_\Delta$ is an iterated double cover branched along orbifold normal crossings divisors at each step.
    
     \begin{theorem}\label{t:toricmirror}
       Let the notation be as above. Then the Hodge number duality for $(\widehat{T}_\Delta, \widehat{T}_{\check{\Delta}})$ holds. In other words, for $p,q \in \mathbb{Z}$ and $d = \dim T_{\Delta}$,
          \[h^{p,q}(\widehat{T}_\Delta)=h^{d-k-p,q}(\widehat{T}_{\check{\Delta}}).\]
    \end{theorem}

    The second pair $(\widetilde{T}_{\Delta}, \widetilde{T}_{\check{\Delta}})$ was first introduced in \cite{Hosono2024doubleCY}, and we refer to such pairs as HLLY mirror pairs. It was conjectured that the Hodge number duality holds so that it is indeed a mirror pair. The case where $d=3$ was proven in \emph{op. cit.}
    \begin{conjecture}[Hosono--Lee--Lian--Yau \cite{Hosono2024doubleCY}]\label{conj:HLLY}
        The Calabi--Yau pair $(\widetilde{T}_{\Delta}, \widetilde{T}_{\check{\Delta}})$ is a mirror pair. 
    \end{conjecture}

    We will prove the Hodge number duality for HLLY mirror pairs in Section \ref{s:HLLY}, providing justification for Conjecture \ref{conj:HLLY}.  
    \begin{theorem}\label{t:HLLY}
         Let the notation be as above. Then the Hodge number duality for  $(\widetilde{T}_{\Delta}, \widetilde{T}_{\check{\Delta}})$ holds. In other words, for $p,q \in \mathbb{Z}$ and $d = \dim T_\Delta$,
    \[h^{p,q}(\widetilde{T}_{\Delta})=h^{d-p,q}( \widetilde{T}_{\check{\Delta}}).\]
    \end{theorem}

    \begin{remark}
        According to Proposition \ref{p : double cover interpretation of the first factor}, $H^*(\widetilde{T}_\Delta)$ is a filtered direct summand of $H^*(\widehat{T}_\Delta)$. Therefore, we may view Theorem \ref{t:HLLY} as a more refined statement than Theorem \ref{t:toricmirror}.
    \end{remark}

    \section{Toric extremal transition}\label{s : toric extremal transitions}
    We describe a general framework of toric extremal transitions \cite{MR1673108} and the Hodge number duality results. Using this framework, we will prove Theorem \ref{t:toricmirror}.

    Suppose we have a pair of reflexive polytopes ${\Delta}_{\mathrm{II}} \subseteq \check{\Delta}_{\mathrm{I}}$ of the same dimension $d$. Dually, we also have a pair of reflexive polytopes ${\Delta}_{\mathrm{I}} \subseteq \check{\Delta}_{\mathrm{II}}$. Let ${T}_{\Delta_\mathrm{I}}$ be an MPCP resolution of the toric Fano variety $T(\Sigma_{\Delta_{\mathrm{I}}})$ attached to the polytope ${\Delta}_{\mathrm{I}}$, and ${T}_{\check{\Delta}_{\mathrm{II}}}$ be an MPCP resolution of the Fano toric variety $T(\Sigma_{\check{\Delta}_{\mathrm{II}}})$ attached to $\check{\Delta}_{\mathrm{II}}$. The classic geometric picture \cite{MR1673108} is that attached to this data we have two families of singular Calabi--Yau varieties in ${T}_{\Delta_\mathrm{I}}$ and ${T}_{\check{\Delta}_\mathrm{II}}$. Viewing integral points of $\check{\Delta}_{\mathrm{I}}$ as monomial sections of $K^{-1}_{{T}_{\Delta_{\mathrm{I}}}}$, we see that integral points of ${\Delta}_{\mathrm{II}}$ determine a family of hypersurfaces in ${T}_{\Delta_\mathrm{I}}$ as well. Furthermore, for a generic choice of such a section, the corresponding hypersurface, which we may denote $X_{{\Delta}_{\mathrm{II}}}'$, is usually singular. We may view $X_{{\Delta}_{\mathrm{II}}}'$ as a degeneration of a very general anticanonical hypersurface $X_{\check{\Delta}_{\mathrm{I}}}$ in ${T}_{\Delta_\mathrm{I}}$. Each $X'_{{\Delta}_{\mathrm{II}}}$ is birational to an anticanonical hypersurface $X_{{\Delta}_{\mathrm{II}}}$ in ${T}_{\check{\Delta}_\mathrm{II}}$. The pair of operations consisting of degeneration of $X_{\check{\Delta}_{\mathrm{I}}}\rightsquigarrow X_{{\Delta}_{\mathrm{II}}}'$ along with the birational map $X_{{\Delta}_{\mathrm{II}}} \dashrightarrow X_{{\Delta}_{\mathrm{II}}}'$ is called an extremal transition. Applying polar duality, we obtain a dual extremal transition, from which we obtain the following diagrams.
    \begin{equation}\label{e : geometric transition}
    \begin{tikzcd}
      & X_{{\Delta}_{\mathrm{II}}} \ar[d,dashed]  \\
      X_{\check{\Delta}_{\mathrm{I}}} \ar[r,rightsquigarrow] & X'_{{\Delta}_{\mathrm{II}}} 
    \end{tikzcd} \hspace{2cm} 
    \begin{tikzcd}
      & X_{\check{\Delta}_{\mathrm{II}}} \ar[d,rightsquigarrow]  \\
      X_{{\Delta}_{\mathrm{I}}} \ar[r,dashed] & X'_{{\Delta}_{\mathrm{I}}}
    \end{tikzcd}
    \end{equation}
    
    We have Hodge number duality between pairs $X_{\check{\Delta}_{\mathrm{I}}}$ and $X_{\Delta_{\mathrm{I}}}$, and between   $X_{\Delta_{\mathrm{II}}}$ and $X_{\check{\Delta}_{\mathrm{II}}}$. The first goal of this section is to prove that a similar duality holds between $X_{{\Delta}_{\mathrm{II}}}'(\subset T_{\Delta_\mathrm{I}})$ and $X_{{\Delta}_{\mathrm{I}}}'(\subset T_{\Delta_\mathrm{II}})$.

    In the following statement the varieties $X'_{\Delta_\mathrm{I}}$ and $X'_{\Delta_{\mathrm{II}}}$ are singular and their cohomology admits a mixed Hodge structure which is not necessarily pure. We let $h^{p,q}(X) = \gr_F^p H^{p+q}(X)$.  In the case where $H^{p+q}(X)$ admits a pure Hodge structure, these are just the usual Hodge numbers.

    \begin{theorem}\label{t:transition}
        Assume the MPCS resolutions ${T}_{\Delta_\mathrm{I}}$ and ${T}_{\Delta_\mathrm{II}}$ are smooth. Then Hodge number duality holds between $X'_{{\Delta}_{\mathrm{II}}}$ and $X'_{{\Delta}_{\mathrm{I}}}$: For $p,q \in \mathbb{Z}$ and $d = \dim T_{\Delta}'$,
        \[
        h^{p,q}(X'_{{\Delta}_{\mathrm{II}}}) = h^{d-1-p,q}(X'_{{\Delta}_{\mathrm{I}}}).
        \]
    \end{theorem}
    
\begin{proof}
    By assumption, we have a MPCS resolution ${T}_{\Delta_\mathrm{I}}$ obtained by taking a projective unimodular triangulation of the $\Delta$, which induces a triangulation on the faces of the polytope $\Delta_{\mathrm{I}}$, as in \cite{batyrev1994dual2}. Given this data we get $\mathrm{Tot}(K_{T_{\Delta_\mathrm{I}}})$, which is a toric variety whose fan, denoted by $\Sigma_{A_{\mathrm{I}}}$, is smooth. Then $\Sigma_{A_{\mathrm{I}}}$ is a smooth refinement of the Gorenstein cone $C_{\Delta_{\mathrm{I}}} = \mathrm{Cone}(\Delta_{\mathrm{I}}\times \{1\}) \subseteq M \times \mathbb{Z}$. The parallel argument applies to get another fan ${\Sigma}_{A_{\mathrm{II}}}$. Since $\Delta_{\mathrm{II}} \subseteq \check{\Delta}_{\mathrm{I}}$, we have 
    \[
    \mathrm{Supp}({\Sigma}_{A_{\mathrm{II}}}) \subseteq \mathrm{Cone}(\check{\Delta}_{\mathrm{I}}\times \{1\}) = C_{\check{\Delta}_{\mathrm{I}}}.
    \]
    The definition of the polar dual is equivalent to $C_{\Delta_\mathrm{I}}$ and $C_{\check{\Delta}_\mathrm{I}}$ being dual cones.  Therefore, ${\Sigma}_{A_{\mathrm{I}}}$ and ${\Sigma}_{A_{\mathrm{II}}}$ form a unimodular, Gorenstein Clarke dual pair of fans (see also \cite[Proposition 6.12]{HL2025}). The corresponding Clarke mirror pair of Landau--Ginzburg models is
    \[
    T({\Sigma}_{A_{\mathrm{I}}}) = \mathrm{Tot}(K_{T_{\Delta_\mathrm{I}}}), \quad w({\Sigma}_{A_{\mathrm{II}}}) = g_{\phi}
    \]
    where $\phi$ is a section of $K^{-1}_{{T}_{\Delta_{\mathrm{I}}}}$ whose vanishing locus in $T_{\Delta_\mathrm{I}}$ is $X'_{{\Delta}_{\mathrm{II}}}$ and similarly, 
    \[
    T({\Sigma}_{A_{\mathrm{II}}}) = \mathrm{Tot}(K_{T_{\Delta_\mathrm{II}}}), \quad w({\Sigma}_{A_{\mathrm{I}}}) = g_{\mu}
    \]
    where $\mu$ is a section of  $K^{-1}_{T_{\Delta_\mathrm{II}}}$ whose vanishing locus in $T_{\Delta_\mathrm{II}}$ is $X'_{{\Delta}_{\mathrm{I}}}$. We see that: 
    \[
    \begin{aligned}
        h^{p,q}(X'_{{\Delta}_{\mathrm{II}}})&=h^{d-p,d-q}_{X'_{{\Delta}_{\mathrm{II}}}}(T_{\Delta_\mathrm{I}})\quad & (\text{Poincar\'e--Lefschetz duality, \cite[(1.7.1)]{Fujiki1980duality}}) \\
        &= h^{d-p,d-q}(T({\Sigma}_{A_{\mathrm{I}}}),w({\Sigma}_{A_{\mathrm{II}}})) \quad &(\text{Proposition \ref{p : well known}}) \\
        &= h^{1+p,d-q}(T({\Sigma}_{A_{\mathrm{II}}}),w({\Sigma}_{A_{\mathrm{I}}})) \quad & \text{(Corollary \ref{c:Cdualgoren})}\\
        &= h^{1+p,d-q}_{X'_{{\Delta}_{\mathrm{I}}}}(T_{\Delta_\mathrm{II}}) \quad & \text{(Proposition \ref{p : well known})} \\
        &= h^{d-1-p,q}(X'_{{\Delta}_\mathrm{I}})\quad  & \text{(Poincar\'e--Lefschetz duality)}.
    \end{aligned}
    \]
    Here, we have used the fact that $X'_{\Delta_\mathrm{I}},X'_{\Delta_\mathrm{II}}$ are compact in the application of Poincar\'e--Lefschetz duality. The application of Proposition \ref{p : well known} is justified by Corollary \ref{c:Cdualgoren}. 
\end{proof}

The following example demonstrates that when $k=1$, the HLLY mirror construction can be interpreted as a duality between singular varieties sitting in a diagram of the form \eqref{e : geometric transition}, therefore, Theorem \ref{t:HLLY} follows from Theorem \ref{t:transition} when $k=1$.\begin{example}\label{e:toricexthyp}
    Let $\Delta \subset M_{\mathbb{R}}$ be a reflexive polytope and $\check{\Delta} \subset N_{\mathbb{R}}$ be its dual. Choose a projective unimodular triangulation of $\Delta$ and $\check{\Delta}$ as before. We construct a pair of reflexive polytopes $\Delta_{\mathrm{II}} \subseteq \check{\Delta}_{\mathrm{I}}$ as follows:
    \[
    \begin{aligned}
        \Delta_{\mathrm{II}}&:=\mathrm{Conv}(\check{\Delta} \times \{1\} \cup 0 \times \{-1\}) \subset N_\mathbb{R} \times \mathbb{R}, \\
        \check{\Delta}_{\mathrm{I}}&:=\mathrm{Conv}(2\check{\Delta} \times \{1\} \cup 0 \times \{-1\}) \subset N_\mathbb{R} \times \mathbb{R}.
    \end{aligned}
    \]
    It is easy to see that both $\Delta_{\mathrm{II}}$ and $\check{\Delta}_{\mathrm{I}}$ are reflexive, whose reflexive duals are given by $\check{\Delta}_{\mathrm{II}}:=\mathrm{Conv}(2{\Delta} \times \{1\} \cup 0 \times \{-1\})$ and ${\Delta}_{\mathrm{I}}:=\mathrm{Conv}({\Delta} \times \{1\} \cup 0 \times \{-1\})$ in $N_\mathbb{R} \times \mathbb{R}$, respectively. 

   Recall that we always assume projective unimodular triangulations of $\Delta$ and $\check{\Delta}$. Choosing such triangulations induce projective  unimodular triangulations of $\Delta_{\mathrm{II}}$ and $\Delta_{\mathrm{I}}$, respectively, and the associated toric varieties $T_{\Delta_{\mathrm{II}}}$ and $T_{{\Delta}_\mathrm{I}}$ are smooth.
    
    Moreover, $T_{{\Delta}_\mathrm{I}}$ is the projectivization of the bundle $\mathcal{O}\oplus K_{T_{{\Delta}}}$. We let $\pi:T_{{\Delta}_\mathrm{I}} \to T_{{\Delta}}$ be the projection. Since $\mathbb{P}(\mathcal{O}\oplus K_{T_{\Delta}}) \cong \mathbb{P}(K_{T_{\Delta}}^{-1}\oplus \mathcal{O})$, the adjunction formula yields an isomorphism $K^{-1}_{T_{\Delta_\mathrm{I}}} \cong \pi^*(K_{T_{{\Delta}}}^{-2})\otimes \mathcal{O}(2)$. This means that the sections corresponding to the integral points in $\check{\Delta} \times \{1\}\subset \check{\Delta}_{\mathrm{I}}$ should be understood as sections of $K^{-2}_{T_{{\Delta}}}$ whose sum is of the form $\sigma_{gen}\sigma_{tor}$. Also, the section corresponding to $0\times\{-1\}$ can be written as $y^2$ where $y$ is the coordinate at the of $T_{\Delta_\mathrm{I}}$. By choosing generic coefficients of these sections, this defines a hypersurface of the form $\{y^2-\sigma_{gen}\sigma_{tor}=0\}$ that is a double cover of $T_{{\Delta}}$ branched over $D=\{\sigma_{gen}\sigma_{tor}=0\}$.
    As a corollary of Theorem \ref{t:transition}, this proves Conjecture \ref{conj:HLLY} when $k=1$.
\end{example}

There is a direct generalization of Theorem \ref{t:transition} to the case of complete intersections but we do not spell this out in general here. Instead, we will prove Theorem \ref{t:toricmirror} which can be considered as a complete intersection generalization of Theorem \ref{t:transition} in a very particular case. 

\begin{proof}[Proof of Theorem \ref{t:toricmirror}]
    Given a pair of reflexive polytopes $\Delta,\check{\Delta}$, a nef partition $\Delta=\Delta_1+\cdots+\Delta_k$ and its dual $\check{\Delta}=\check{\Delta}_1+\cdots+\check{\Delta}_k$, we define a pair of reflexive polytopes $\Delta_{\mathrm{II}} \subset \check{\Delta}_{{\mathrm{I}}}$:
    \[
    \begin{aligned}
        \Delta_{\mathrm{II}}&:=\mathrm{Conv}(\check{\Delta}_i\times e_i \cup 0\times -e_i\mid i=1, \cdots, k) \subset M_\mathbb{R}\times \mathbb{R}^k \\
        \check{\Delta}_{\mathrm{I}}&:=\mathrm{Conv}(2\check{\Delta}_i\times e_i \cup 0\times -e_i \mid i=1, \cdots, k) \subset M_\mathbb{R}\times \mathbb{R}^k 
    \end{aligned}
    \]
    where $\{e_1, \cdots, e_k\}$ is the standard basis of $\mathbb{R}^k$. It is clear that the origin $(0,0)\in M_\mathbb{R}\times \mathbb{R}^k$ is the unique interior integral point of the both, and a direct check shows that they are reflexive. The reflexive duals are constructed in the same way
    \[
    \begin{aligned}
        \Delta_{\mathrm{I}}&=\mathrm{Conv}({\Delta}_i\times e^*_i \cup 0\times -e^*_i\mid i=1, \cdots, k) \subset N_\mathbb{R}\times \mathbb{R}^k \\
        \check{\Delta}_{\mathrm{II}}&=\mathrm{Conv}(2{\Delta}_i\times e^*_i \cup 0\times -e^*_i\mid i=1, \cdots, k) \subset N_\mathbb{R}\times \mathbb{R}^k 
    \end{aligned}
    \]
    where $\{e^*_1, \cdots, e^*_k\}$ are the dual basis of $\{e_1, \cdots, e_k\}$. As before, we choose projective unimodular triangulations of $\Delta$ and $\check{\Delta}$, which in turn induce triangulations of the above reflexive polytopes.

    In this setup, $T_{\Delta_\mathrm{I}}$ is the fiber product of the projectivizations $\mathbb{P}(\mathcal{O}(E_{{\Delta_i}})\oplus \mathcal{O})$ of $\mathcal{O}(E_{\Delta_i})\oplus\mathcal{O}$. As explained in Example \ref{e:toricexthyp}, for each $i$, the integral points in $\Delta_{\mathrm{II}, i}:=\Delta_{\mathrm{II}} \cap (M\times e_i)$ define a hypersurface of the form $y_i^2-\sigma_i$ where $\sigma_i:=\sigma_{i,gen}\sigma_{i,tor}$ and $y_i$ is the coordinate at the infinity of $\mathbb{P}(E_{\Delta_i})$. For the moment, we denote this hypersurface in $T_{\Delta_\mathrm{I}}$ by $X_{\Delta_{\mathrm{II},i}}'$. The complete intersection of the hypersurfaces $X'_{\Delta_{\mathrm{II},i}}$ is the $(\mathbb{Z}/2)^k$ cover of $T_\Delta$ denoted $\widehat{T}_{\Delta}$ in the preceding section.

    Let $\Sigma_{A_\mathrm{I}}$ denote the fan that determines the total space of the line bundle ${\pi}^*\mathcal{O}({E}_{\Delta_1})\oplus \dots \oplus {\pi}^*\mathcal{O}({E}_{\Delta_k})$ where ${\pi} : T_{{\Delta}_\mathrm{I}} \rightarrow T_{{\Delta}}$ is the usual projection map\footnote{The total space of any split vector bundle over a toric variety is itself toric \cite{Cox2011toricbook}.} and similarly let $\Sigma_{\check{A}_\mathrm{II}}$ denote the fan determining the total space of the vector bundle $\pi^*\mathcal{O}(2{E}_{\Delta_1})\oplus \dots \oplus \pi^*\mathcal{O}(2{E}_{\Delta_k})$ where $\pi : T_{\check{\Delta}_{\mathrm{II}}} \rightarrow T_{\Delta}$ also denotes a projection map. We may define $\Sigma_{\check{A}_\mathrm{I}}$ and $\Sigma_{{A}_{\mathrm{II}}}$ similarly.    
    The pairs of fans $(\Sigma_{A_\mathrm{I}},\Sigma_{\check{A}_\mathrm{I}})$ and $(\Sigma_{\check{A}_\mathrm{II}},\Sigma_{{A}_\mathrm{II}})$ form unimodular, Gorenstein Clarke mirror pairs \cite[Proposition 6.12]{HL2025}. One may check that $\Sigma_{A_\mathrm{II}}\subseteq \Sigma_{\check{A}_\mathrm{I}}$. Therefore, by definition, the pair $(\Sigma_{A_\mathrm{I}},\Sigma_{A_\mathrm{II}})$ is also an unimodular, Gorenstein Clarke mirror pair of fans.
    
    As we put different basis element $e_i$ for each $i$, the Landau--Ginzburg model $(T(\Sigma_{A_\mathrm{I}}), w(\Sigma_{A_{\mathrm{II}}}))$ corresponds to the complete intersection of $X_{\Delta_{\mathrm{II}},1}', \cdots, X_{\Delta_{\mathrm{II}},k}'$, that is  $\widehat{T}_{\Delta}$. Precisely, if $t_1,\dots, t_k$ are coordinate functions on the fibers of $T(\Sigma_{A_\mathrm{I}}) = \mathrm{Tot}({\pi}^*\mathcal{O}({E}_{\Delta_1})\oplus \dots \oplus {\pi}^*\mathcal{O}({E}_{\Delta_k}))$, then
    \[
    w(\Sigma_{A_\mathrm{II}}) = \sum_{i=1}^k t_i(y_i^2 -\sigma_i).
    \]
    Therefore, Proposition \ref{p : well known}, along with Corollary \ref{c:Cdualgoren}, allows us to relate the irregular Hodge numbers of the Landau--Ginzburg model $(T(\Sigma_{A_\mathrm{I}}), w(\Sigma_{A_{\mathrm{II}}}))$ to the Hodge numbers of the iterated double cover $\widehat{T}_{\Delta} \rightarrow T_\Delta$.
    
    Applying the parallel construction, we obtain an unimodular, Gorenstein Clarke mirror pair (cf. \cite[Proposition 6.12]{HL2025})
    \[(T(\Sigma_{A_\mathrm{I}}), w(\Sigma_{A_{\mathrm{II}}})), \qquad (T(\Sigma_{A_\mathrm{II}}), w(\Sigma_{A_{\mathrm{I}}})).
    \]
    We have the following identity of Hodge numbers: For $p,q \in \mathbb{Z}$, 
    \[
    \begin{aligned}
        h^{p,q}(\widehat{T}_{\Delta})&=h^{d-p,d-q}_{\widehat{T}_{\Delta}}(T_{\Delta_\mathrm{I}}) \quad  & \text{(Poincar\'e--Lefschetz duality)} \\
        &= h^{d-p,d-q}(T({\Sigma}_{A_{\mathrm{I}}}),w({\Sigma}_{A_{\mathrm{II}}})) \quad &(\text{Proposition \ref{p : well known}})\\
        &= h^{k+p,d-q}(T({\Sigma}_{A_{\mathrm{II}}}),w({\Sigma}_{A_{\mathrm{I}}})) \quad & \text{(Corollary \ref{c:Cdualgoren})} \\
        &= h^{k+p,d-q}_{\widehat{T}_{\check{\Delta}}}(T_{\Delta_{\mathrm{II}}}) \quad &(\text{Proposition \ref{p : well known}}) \\
        &= h^{d-k-p,q}(\widehat{T}_{\check{\Delta}}). \quad  & \text{(Poincar\'e--Lefschetz duality)}
    \end{aligned}
    \]
\end{proof}

\begin{remark}
        It is natural to ask whether the proofs above work for cyclic covers of higher order. It turns out that they do not, as the nef partitions constructed in the proof of Theorem \ref{t:toricmirror} depend on the double covering property. We do not expect a mirror symmetry for higher order cyclic covers.
    \end{remark}

    \section{A proof of HLLY conjecture}\label{s:HLLY}
	
	We prove Theorem \ref{t:HLLY} (HLLY conjecture) by combining Theorem \ref{t:toricmirror} and some other Clarke dualities. 
	
As before, let's start with a pair of reflexive dual polytopes $(\Delta,\check{\Delta})$ with a nef partition $\Delta=\Delta_1\cup\cdots\cup\Delta_k$ and its dual nef partition $\check{\Delta}=\check{\Delta}_1\cup\cdots\cup\check{\Delta}_k$. We choose a projective unimodular triangulation of each $\Delta_i$ and $\check{\Delta}_i$ whose integral points are denoted by $A_i$ and $\check{A}_i$, respectively. We keep the same notation as in the previous sections. 

For $a_i \in \mathbb{Z}_{\geq 0}$, we define a stacky fan $\Sigma_{\{a_iA_i|i=1, \dots, k\}} \subset M \times \mathbb{Z}^k$ whose ray generators are of the form $\{\rho \times a_ie_i|\rho \in A_i \}$. For simplicity, once a subset $J \subseteq \{1, \cdots, k\}$ is taken, we write $\Sigma_{2A_{J}, A_{J^c}}=\Sigma_{\{2A_i |i \in J\} \cup \{A_i | i \notin J\}}$. For example, when $J = \emptyset$, then $\Sigma_{2A_{J}, A_{J^c}}$ is the fan for the total space of the vector bundle $ \mathcal{O}(-E_{\Delta_1}) \oplus \cdots \oplus \mathcal{O}(-E_{\Delta_k})$ over $T_{\Delta}$. The general case is described below. By applying the parallel construction for $\check{A}_i$'s, we obtain several Clarke dual pairs.

\begin{lemma}
    Fix a subset $J \subseteq \{1, \cdots, k\}$. Then the pair $(\Sigma_{2A_{J}, A_{J^c}}, \Sigma_{\check{A}_{J}, 2\check{A}_{J^c}})$ forms a simplicial, quasiprojective, Clarke dual pair. 
\end{lemma}
\begin{proof}
   It is straightforward to verify that the fans are simplicial and quasiprojective. The regularity of the pair follows directly from the definition of nef partitions. Convexity follows from a similar argument as in \cite[Proposition 6.12]{HL2025}.
\end{proof}

To such a Clarke pair $(\Sigma_{2A_{J}, A_{J^c}}, \Sigma_{\check{A}_{J}, 2\check{A}_{J^c}})$, the associated LG models are given as follows:
\[(T(\Sigma_{2A_{J}, A_{J^c}}), w(\Sigma_{\check{A}_{J}, 2\check{A}_{J^c}})), \qquad (T(\Sigma_{\check{A}_{J}, 2\check{A}_{J^c}}), w(\Sigma_{{2A}_{J}, {A}_{J^c}}))
\]
where
\begin{itemize}
	\item $T(\Sigma_{2A_{J}, A_{J^c}})=\sqrt{\mathrm{Tot}(-2E_{\Delta_J}-E_{\Delta_{J^c}})/0_{J}}$ and $w(\Sigma_{\check{A}_{J}, 2\check{A}_{J^c}})=g_{1,J}+g_{2,J^c}$. 
	\item $T(\Sigma_{\check{A}_{J}, 2\check{A}_{J^c}})=\sqrt{\mathrm{Tot}(-2E_{\check{\Delta}_{J^c}}-E_{\check{\Delta}_{J}})/0_{J^c}}$ and $w(\Sigma_{2{A}_{J}, {A}_{J^c}})=\check{g}_{1,J^c}+\check{g}_{2,J}$. 
\end{itemize} 
We simplify notation by dropping  $\mathcal{O}$ when referring to the total space of the corresponding line bundle. For example, 
\[
\mathrm{Tot}(-2E_{\Delta_J}-E_{\Delta_{J^c}}):=\mathrm{Tot}\left(\bigoplus_{j \in J}\mathcal{O}(-2E_{\Delta_j})\oplus \bigoplus_{j \notin J} \mathcal{O}(-E_{\Delta_j})\right).\]
Also, if $J = \{j_1,\dots, j_m\}$, then   $\sqrt{\mathrm{Tot}(-2E_{\Delta_J} - E_{\Delta_{J^c}})/0_{J}}$ is the fiber product 
\[
\sqrt{\mathrm{Tot}(-2E_{j_1})/0_{E_{j_1}}} \times_{T_\Delta} \dots \times_{T_\Delta}\sqrt{\mathrm{Tot}(-2E_{j_m})/0_{E_{j_m}}}\times_{T_\Delta} \mathrm{Tot}(-E_{\Delta_{J^c}})
\]
where $0_{E_i}$ denotes the zero section of the projection map. An analogous description applies to $\sqrt{\mathrm{Tot}(-2E_{\check{\Delta}_{J^c}}-E_{\check{\Delta}_{J}})/0_{J^c}}$.



	By Theorem \ref{t : multiple direct sum decomposition} and Theorem \ref{t:poddar}, the orbifold cohomology 
		\[H_{\mathrm{orb}}^*\left(\sqrt{\mathrm{Tot}(-2E_{\Delta_J}-E_{\Delta_{J^c}})/0_{J}}, g_{1,J}+g_{2,J^c}\right)\]
		can be decomposed based on both age grading and the grading in Corollary \ref{cor:another description} by coinvariant degree. Namely, this orbifold cohomology is decomposed as follows:
        \[
        \begin{aligned}
            &\bigoplus_{I^{(0)} \subseteq J} H^{*-|I^{(0)}|}\left(\mathrm{Tot}(-2E_{\Delta_{J\setminus I^{(0)}}}-E_{\Delta_{J^c}}), g_{1,J\setminus I^{(0)}}+g_{2,J^c}\right)\left(\frac{|I^{(0)}|}{2}\right)=\\
            & \bigoplus_{I^{(0)} \subseteq J, I^{(-)} \subseteq J^c}H^{*-|I^{(0)}|}\left(\mathrm{Tot}(-2E_{\Delta_{J\setminus I^{(0)} \cup J^c \setminus I^{(-)}}}-E_{\Delta_{I^{(-)}}}), g_{1,J\setminus I^{(0)} \cup J^c \setminus I^{(-)}}+g_{2, I^{(-)}}\right)^{(I^{(-)})}\left(\frac{|I^{(0)}|}{2}\right)\\
            & =\bigoplus_{I^{(0)} \subseteq J, I^{(-)} \subseteq J^c}H^{*-|I^{(0)}|}\left(\mathrm{Tot}(-2E_{\Delta_{I^{(+)}}}-E_{\Delta_{I^{(-)}}}), g_{1,I^{(+)}}+g_{2, I^{(-)}}\right)^{(I^{(-)})}\left(\frac{|I^{(0)}|}{2}\right).
        \end{aligned}
        \]
		where $I^{(0)}\subseteq J$ consists of the indices corresponding to the twisted sector (=number of shifting), and $I^{(-)}$ parametrizes the indices of the coinvariant part and $I^{(+)}:=\{1, \cdots, k\} \setminus  I^{(-)} \sqcup I^{(0)}$ that parametrizes the number of invariant parts. See also Remark \ref{rem:notation change} for the superscript notation $(I^{(-)})$. We define $B^{\lambda, \mu}_{I^{(0)}, I^{(-)}, I^{(+)}}$ to be the dimension of the direct summand indexed by $(I^{(0)}, I^{(-)})$ in the above cohomology group. Namely
        \begin{equation}\label{e:B def}
            B^{\lambda, \mu}_{I^{(-)}, I^{(0)}, I^{(+)}}:=\dim H^{\lambda-\frac{|I^{(0)}|}{2}, \mu-\frac{|I^{(0)}|}{2}}\left(\mathrm{Tot}(-2E_{\Delta_{I^{(+)}}}-E_{\Delta_{I^{(-)}}}), g_{1,I^{(+)}}+g_{2, I^{(-)}}\right)^{(I^{(-)})}.
        \end{equation}
        We will use the notation $B_{I^{(-)},I^{(0)},I^{(+)}}$ to denote the element of $\mathbb{Z}^{\tfrac{1}{2}\mathbb{Z} \times \tfrac{1}{2}\mathbb{Z}}$ whose $(\lambda,\mu)$ component is $B_{I^{(-)},I^{(0)},I^{(+)}}^{\lambda,\mu}$. Later, we will use a shifting operation on such objects, defined by letting $B_{I^{(-)},I^{(0)},I^{(+)}}(a)$ be such that
        \begin{equation}\label{e: shifting degrees}
            B_{I^{(-)},I^{(0)},I^{(+)}}^{\lambda ,\mu}(a) = B_{I^{(-)},I^{(0)},I^{(+)}}^{\lambda -a,\mu-a}
        \end{equation}
        for any $a \in \tfrac{1}{2}\mathbb{Z}$.
        \begin{remark}
            If we let $\widehat{T}_{\Delta,I^{(-)}}$ denote the $(\mathbb{Z}/2)^{|I^{(-)}|}$ cover of $T_\Delta$ with branch divisors $D_i, i \in I^{(-)}$, constructed as in Section \ref{s : statement}, then we may rewrite 
            \begin{equation}\label{e : restatement of definition of B}
            B^{\lambda,\mu}_{I^{(-)},I^{(0)},I^{(+)}} = h^{\lambda-\frac{|I^{(0)}|-|I^{(-)}|}{2},\mu-\frac{|I^{(0)}|-|I^{(-)}|}{2}}_{Z^{I^{(+)}}}(\widehat{T}_{\Delta,I^{(-)}})^{(I^{(-)})}.
            \end{equation}
            where $Z^{I^{(+)}}$ is the preimage of $\bigcap_{i \in I^{(+)}}D_i$ in $\widehat{T}_{\Delta,I^{(-)}}$.
        \end{remark}

		 Due to Theorem \ref{t:Cdual}, we have the identity of Hodge numbers: For $\lambda, \mu \in \mathbb{Q}$, 
		 \[h_{\mathrm{orb}}^{\lambda, \mu}\left(\sqrt{\mathrm{Tot}(-2E_{\Delta_J}-E_{\Delta_{J^c}})/0_{J}}, g_{1,J}+g_{2,J^c}\right)=h_{\mathrm{orb}}^{d+k-\lambda, \mu}\left(\sqrt{\mathrm{Tot}(-2E_{\check{\Delta}_{J^c}}-E_{\check{\Delta}_{J}})/0_{J^c}}, \check{g}_{1,J}+\check{g}_{2,J^c}\right).\]
		 The degree shifting of each summand in the both is determined by $|I^{(-)} \sqcup I^{(0)}|$. Therefore, to acquire symmetry between $I^{(-)}$ and $I^{(0)}$, we consider another Clarke dual pair $(\Sigma_{A_J, 2A_{J^c}}, \Sigma_{2\check{A}_{J}, \check{A}_{J^c} })$ that is obtained by simply changing the role of $J$ and $J^c$ from the previous pair. Then we get a similar identity of Hodge numbers: For $\lambda, \mu \in \mathbb{Q}$, 
		 \[h_{\mathrm{orb}}^{\lambda, \mu}\left(\sqrt{\mathrm{Tot}(-2E_{\Delta_{J^c}}-E_{\Delta_{J}})/0_{J^c}}, g_{1,J^c}+g_{2,J}\right)=h_{\mathrm{orb}}^{d+k-\lambda, \mu}\left(\sqrt{\mathrm{Tot}(-2E_{\check{\Delta}_{J}}-E_{\check{\Delta}_{J^c}})/0_{J}}, \check{g}_{1,J^c}+\check{g}_{2,J}\right).\]
		 Now we combine these two Hodge number dualities to deduce the weaker one. In other words, we pair up the orbifold cohomology 
		 \begin{equation}\label{e:grping}
         H_{\mathrm{orb}}^*\left(\sqrt{\mathrm{Tot}(-2E_{\Delta_J}-E_{\Delta_{J^c}})/0_{J}}, g_{1,J}+g_{2,J^c}\right)\bigoplus H_{\mathrm{orb}}^*\left(\sqrt{\mathrm{Tot}(-2E_{\Delta_{J^c}}-E_{\Delta_{J}})/0_{J^c}}, {g}_{1,J^c}+{g}_{2,J}\right).
		 \end{equation}
		 The mirror counterpart can be simply obtained by replacing $\Delta$ and $g$ by $\check{\Delta}$ and $\check{g}$, respectively. Also, the degree shifting of each summand in the both is determined by $|I^{(-)} \sqcup I^{(0)}|$. 
         Then the weaker version of Hodge number duality can be written as  
		 \begin{equation}\label{e:mir}
		 \sum_{I^{(0)}\subseteq J, I^{(-)}\subseteq J^c} B^{\lambda, \mu}_{I^{(-)}, I^{(0)}, I^{(+)}}+B^{\lambda, \mu}_{I^{(0)}, I^{(-)}, I^{(+)}}= \sum_{I^{(0)}\subseteq J, I^{(-)}\subseteq J^c} \check{B}^{d+k-\lambda, \mu}_{I^{(-)}, I^{(0)}, I^{(+)}}+\check{B}^{d+k-\lambda, \mu}_{I^{(0)}, I^{(-)}, I^{(+)}}
		 \end{equation}
		where $\check{B}^{d+k-\lambda, \mu}_{I^{(-)}, I^{(0)}, I^{(+)}}$'s are defined in the same way as \eqref{e:B def}. Since the degree of the mirror dual part is canonically determined once the Hodge grading $(\lambda, \mu)$ of $B_{I^{(-)}, I^{(0)}, I^{(+)}}$ is specified, we simply say that the sum 
		\[
		 \sum_{I^{(0)}\subseteq J, I^{(-)}\subseteq J^c} B_{I^{(-)}, I^{(0)}, I^{(+)}}+B_{I^{(0)}, I^{(-)}, I^{(+)}}
		 \]
		satisfies the \emph{mirror relation} because \eqref{e:mir} holds for all $(\lambda, \mu)$. We will say that the mirror relation is satisfied for any finite sum of $B_{I^{(-)},I^{(0)},I^{(+)}}$ if the analogue of \eqref{e:mir} holds. 

        Note that 
        \[
        B^{\lambda,\mu}_{[k],\emptyset,\emptyset} = h^{\lambda-k/2,\mu-k/2}(\widetilde{T}_\Delta)^{(-)},\qquad B^{\lambda,\mu}_{\emptyset,[k],\emptyset} = h^{\lambda-k/2,\mu-k/2}(T_\Delta).
        \]
        So, by Proposition \ref{p : double cover interpretation of the first factor}, Theorem \ref{t:HLLY} reduces to showing that 
        \begin{equation}\label{e : satisfies the mirror relation}
            B_{[k],\emptyset,\emptyset} + B_{\emptyset,[k],\emptyset} \text{ satisfies the mirror relation.}
        \end{equation}
        The reader should also observe that, following Proposition \ref{p : double cover interpretation of the first factor} and \eqref{e : restatement of definition of B}, we have 
        \[
        \sum_{I^{(0)}\cup I^{(-)} = [k]} B^{\lambda,\mu}_{I^{(-)},I^{(0)},\emptyset} = h^{\lambda - k,\mu-k}(\widehat{T}_\Delta).
        \]
        Therefore, Proposition \ref{p : double cover interpretation of the first factor} and Theorem \ref{t:toricmirror} tell us that:
        \begin{equation}
            \sum_{I^{(0)}\cup I^{(-)} = [k]} B_{I^{(-)},I^{(0)},\emptyset} \text{ satisfies the mirror relation.}
        \end{equation}
        
        Note that if $k-|I^{(+)}| \in 2\mathbb{Z}$ then $B^{\lambda,\mu}_{I^{(-)}, I^{(0)}, I^{(+)}} = 0$ unless $\lambda,\mu \in \mathbb{Z}$ and similarly if $k - |I^{(+)}| \in 2\mathbb{Z}+1$ then $B^{\lambda,\mu}_{I^{(-)}, I^{(0)}, I^{(+)}} = 0$ unless $\lambda,\mu \in (\tfrac{1}{2}) + \mathbb{Z}$. Since we are largely interested in the case where $|I^{(+)}| = 0$, we can distinguish between the cases where $|I^{(+)}|$ is even or odd. We focus only on the case where $|I^{(+)}|$ is even, as this is the equality relevant to \eqref{e : satisfies the mirror relation} and hence the proof of  Conjecture \ref{conj:HLLY}. Also, note that the mirror relation \eqref{e:mir} is preserved under scalar multiplication and addition. 
        
		Before getting into the proof of Theorem \ref{t:HLLY}, we examine the cases $k=2$ and $k=3$ to illustrate how this description is used in the proof of Conjecture \ref{conj:HLLY}.
		  
		\begin{example}\label{eg:k=2}
			Let's examine the case when $k=2$. Consider the following Clarke dual pair 
			\[
			(\Sigma_{2A_1, 2A_2}, \Sigma_{\check{A}_1, \check{A}_2}).
			\]
			The induced LG model is given by 
			\[
			(T(\Sigma_{2A_1, 2A_2}), w(\Sigma_{\check{A}_1, \check{A}_2})), \qquad (T(\Sigma_{\check{A}_1, \check{A}_2}), w(\Sigma_{2A_1, 2A_2}))
			\]
			where 
			\begin{itemize}
				\item $T(\Sigma_{2A_1, 2A_2})=\sqrt{\mathrm{Tot}(-2E_{\Delta_1}-2E_{\Delta_2})/0_{\{1,2\}}}$ and $w(\Sigma_{\check{A}_1, \check{A}_2})=g_{1,\{1\}}+g_{1,\{2\}}$.
				\item $T(\Sigma_{\check{A}_1, \check{A}_2})=\mathrm{Tot}(-E_{\check{\Delta}_1}-E_{\check{\Delta}_2})$, and $w(\Sigma_{2A_1, 2A_2})=\check{g}_{2,\{1\}}+\check{g}_{2,\{2\}}$.
			\end{itemize}
			Let's compute the orbifold cohomology. For the LG model $ (T(\Sigma_{2A_1, 2A_2}), w(\Sigma_{\check{A}_1, \check{A}_2}))$, we have 
			\[
			\begin{aligned}
				H_{\mathrm{orb}}^{\lambda, \mu}(T(\Sigma_{2A_1, 2A_2}), w(\Sigma_{\check{A}_1, \check{A}_2}))&=H^{\lambda, \mu}(\mathrm{Tot}(-2E_{\Delta_1}-2E_{\Delta_2}), g_{1,\{1\}}+g_{1,\{2\}}) \\
				&\oplus H^{\lambda-\frac{1}{2}, \mu-\frac{1}{2}}(\mathrm{Tot}(-2E_{\Delta_2}), g_{1,\{2\}}) \\
				&\oplus H^{\lambda-\frac{1}{2}, \mu-\frac{1}{2}}(\mathrm{Tot}(-2E_{\Delta_1}), g_{1,\{1\}}) \\
				&\oplus H^{\lambda-1, \mu-1}(T_\Delta).
			\end{aligned}
			\]
			Following the previous notation, the dimension can be written as $B^{\lambda, \mu}_{\emptyset, \emptyset, \{1,2\}}+B^{\lambda, \mu}_{\emptyset, \{1\}, \{2\}}+B^{\lambda, \mu}_{\emptyset, \{2\}, \{1\}}+B^{\lambda, \mu}_{\emptyset, \{1,2\}, \emptyset}$.
			On the other hand, for the LG model $ (T(\Sigma_{\check{A}_1, \check{A}_2}), w(\Sigma_{2A_1, 2A_2}))$, we have 
			\[
			\begin{aligned}
				H_{(\mathrm{orb})}^{\lambda, \mu}(T(\Sigma_{\check{A}_1, \check{A}_2}), w(\Sigma_{2A_1, 2A_2}))&=H^{\lambda, \mu}(\mathrm{Tot}(-2E_{\check{\Delta}_1}-2E_{\check{\Delta}_2}), \check{g}_{1,\{1\}}+\check{g}_{1,\{2\}}) \\
				&\oplus H^{\lambda, \mu}(\mathrm{Tot}(-2E_{\check{\Delta}_1}-E_{\check{\Delta}_2}), \check{g}_{1,\{1\}}+\check{g}_{2,\{2\}})^{\{2\}^{(-)}} \\
				&\oplus H^{\lambda, \mu}(\mathrm{Tot}(-E_{\check{\Delta}_1}-2E_{\check{\Delta}_2}), \check{g}_{2,\{1\}}+\check{g}_{1,\{2\}})^{\{1\}^{(-)}} \\
				&\oplus H^{\lambda, \mu}(\mathrm{Tot}(-E_{\check{\Delta}_1}-E_{\check{\Delta}_2}), \check{g}_{2,\{1\}}+\check{g}_{2,\{2\}})^{\{1,2\}^{(-)}}.
			\end{aligned}
			\]
			Again, the dimension becomes  $\check{B}^{\lambda, \mu}_{\emptyset, \emptyset, \{1,2\}}+\check{B}^{\lambda, \mu}_{\{2\}, \emptyset, \{1\}}+\check{B}^{\lambda, \mu}_{\{1\}, \emptyset, \{2\}}+\check{B}^{\lambda, \mu}_{\{1,2\}, \emptyset, \emptyset}$.
By looking at the part where $|I^{(+)}|$ is even (equivalently $\lambda, \mu \in \mathbb{Z}$ in this case), for every $\lambda, \mu \in \mathbb{Q}$, we get
			\[
			B^{\lambda, \mu}_{\emptyset, \emptyset, \{1,2\}}+B^{\lambda, \mu}_{\emptyset, \{1,2\}, \emptyset}= \check{B}^{d+2-\lambda, \mu}_{\emptyset, \emptyset, \{1,2\}}+\check{B}^{d+2-\lambda, \mu}_{\{1,2\}, \emptyset, \emptyset}.
			\]
In fact, more symmetrically, 
			\[	B_{\emptyset, \emptyset, \{1,2\}}+B_{\emptyset, \{1,2\}, \emptyset}+{B}_{\emptyset, \emptyset, \{1,2\}}+{B}_{\{1,2\}, \emptyset, \emptyset}\]
			satisfies the mirror relation \eqref{e:mir}.
			
			One may also consider another Clarke dual pair $(\Sigma_{2A_1, A_2}, \Sigma_{\check{A}_1, 2\check{A}_2})$ and the induced LG models:
			\[
			(T(\Sigma_{2A_1, A_2}), w(\Sigma_{\check{A}_1, 2\check{A}_2})), \qquad (T(\Sigma_{\check{A}_1, 2\check{A}_2}), w(\Sigma_{2A_1, A_2})).
			\]
			Then it is easy to see that the sum
			\[
			B_{\emptyset, \emptyset, \{1,2\}}+B_{\{2\}, \{1\}, \emptyset}+{B}_{\emptyset, \emptyset, \{1,2\}}+{B}_{\{1\}, \{2\}, \emptyset}
			\]
			satisfies the mirror relation \eqref{e:mir}. By substracting one from the other, we eliminate the terms $B_{\emptyset, \emptyset, \{1,2\}}$, and get the sum $B_{\emptyset, \{1, 2\},\emptyset}-B_{\{2\}, \{1\}, \emptyset}-{B}_{\{1\}, \{2\}, \emptyset}+{B}_{\{1,2\},\emptyset, \emptyset}$ which satisfies the mirror relation \eqref{e:mir}. Combining with the Hodge number duality for $(\widehat{T}_\Delta, \widehat{T}_{\check{\Delta}})$ (Theorem \ref{t:toricmirror}), this proves Conjecture \ref{conj:HLLY}.
		\end{example}
		
		\begin{example}\label{eg:k=3}
			Let's examine the case when $k=3$. First, consider the mirror pair \[
            (\Sigma_{2A_1, 2A_2, 2A_3}, \Sigma_{\check{A}_1, \check{A}_2, \check{A}_3}).\] 
            Performing the similar computation as in the case $k=2$, we obtain the following sum that satisfies the mirror relation \eqref{e:mir}: 
\begin{equation}\label{e:k=3,p=0}
\begin{aligned}
    &B_{\emptyset, \{1,2,3\}, \emptyset}+B_{\emptyset, \{1\}, \{2,3\}}+B_{\emptyset, \{2\}, \{1,3\}}+B_{\emptyset, \{3\}, \{1,2\}}\\
    &+B_{ \{1\}, \emptyset,\{2,3\}}+B_{\{2\},\emptyset,  \{1,3\}}+B_{\{3\}, \emptyset, \{1,2\}}+B_{\{1,2,3\},\emptyset,  \emptyset}.
\end{aligned}
\end{equation}
	Note that this is the sum with $|I^{(+)}|$ being even, which requires the odd number of shiftings (i.e. $\lambda, \mu \in \mathbb{Z}+1/2$). On the other hand, from the Clarke pair $(\Sigma_{A_1, 2A_2, 2A_3}, \Sigma_{2\check{A}_1, \check{A}_2, \check{A}_3})$, we get
	\begin{equation}\label{e:k=3, p=1}
			\begin{aligned}
&B_{\{2,3\}, \{1\}, \emptyset}+B_{\emptyset, \{1\}, \{2,3\}}+B_{\emptyset, \{2\}, \{1,3\}}+B_{\emptyset, \{3\}, \{1,2\}}\\
&+B_{ \{1\}, \emptyset,\{2,3\}}+B_{\{2\},\emptyset,  \{1,3\}}+B_{\{3\}, \emptyset, \{1,2\}}+B_{\{1\}, \{2,3\}, \emptyset}.			 
			\end{aligned}
	\end{equation}
		By taking $J=\{2\}$ and $J=\{3\}$, we also obtain the similar sums where only the first and the last terms are different. Then by subtracting there three sums from the sum \eqref{e:k=3,p=0} multiplied by $3$, we will get the sum
			\[
            \begin{aligned}
                &B_{\emptyset, \{1,2,3\}, \emptyset}+B_{\{1,2,3\},\emptyset,  \emptyset}\\
                &-(B_{\{2,3\}, \{1\}, 0}+B_{\{1\}, \{2,3\}, 0})-(B_{\{1,3\}, \{2\}, 0}+B_{\{2\}, \{1,3\}, 0})-(B_{\{1,2\}, \{3\}, 0}+B_{\{3\}, \{1,2\}, 0})
            \end{aligned}
			\]
            that satisfies the mirror relation \eqref{e:mir}. Now, comparing with the Hodge number duality of $(\widehat{T}_\Delta, \widehat{T}_{\check{\Delta}})$, we can prove Conjecture \ref{conj:HLLY}: $B_{\emptyset, \{1,2,3\}, \emptyset}+B_{\{1,2,3\},\emptyset,  \emptyset}$ satisfies the mirror relation \eqref{e:mir}. Furthermore, subtracting this sum from \eqref{e:k=3,p=0}, we obtain the middle terms of \eqref{e:k=3,p=0} satisfies the mirror relation \eqref{e:mir}, so does $B_{\{2,3\}, \{1\}, \emptyset}+B_{\{1\}, \{2,3\}, \emptyset}$. Since the role of a partition of $\{1,2,3\}$ doesn't matter, we conclude that for $\{i,j,k\}=\{1,2,3\}$, the following sum satisfies the mirror relation \eqref{e:mir}:
			\begin{equation}\label{e:k=3, special}
				B_{\{i,j\}, \{k\}, \emptyset}+B_{\{k\}, \{i,j\}, \emptyset}.
			\end{equation}
        Plugging it back to \eqref{e:k=3,p=0}, the result follows. 
		\end{example}

		For the cases when $k \geq 4$, we do not need to keep track of all the indices of the summands for each $J$. Instead, we group all $J$'s with the same size together. For $|J| = p$ with $0 \leq p \leq \lfloor \frac{k}{2} \rfloor$, we group the summands $B$'s and denote
        $B_{a,b,c}=\sum_{|I^{(-)}|=a, |I^{(0)}|=b, |I^{(+)}|=c} B_{I^{(-)}, I^{(0)}, I^{(+)}}$ for any $a,b,c \geq 0$.

		\begin{lemma}
        The sum of all $B_{I^{(-)},I^{(0)},I^{(+)}}$ with $|J| = p$ and $|I^{(+)}| = |J \setminus (I^{(-)}\cup I^{(0)})| \in 2\mathbb{Z}$ is given by 
		\[
			T(k,p):=\sum_{n=0}^{p} \sum_{\substack{p-n \leq i \leq k-n, \\ i \in 2\mathbb{Z}}}\binom{i}{p-n}\left(B_{k-n-i,n,i}+B_{n,k-n-i,i}\right).\]
		\end{lemma}
		
		\begin{proof}
        Fix $I^{(-)},I^{(+)}$ so that $|I^{(-)}| = n, |I^{(+)}| = k-n-i$. The coefficient of $B_{I^{(-)},I^{(0)},I^{(+)}}$ in $T(k,p)$ is equal to the number of subsets $J$ of $[k]$ of size $p$ containing $I^{(-)}$ but not $I^{(0)}$ or equivalently, the number of subsets of $[k]\setminus I^{(-)}\cup I^{(0)} = [k] \setminus I^{(+)}$ of size $p-n$, which is just ${i \choose p-n}$. The same argument holds with $n$ and $k-n-i$ exchanged.\end{proof}

		\begin{lemma}\label{l:Clrk identity}
			There are the binomial identities:
			\begin{enumerate}
				\item When $k$ is even, 
					\begin{equation}\label{e:Cl1}
					\sum_{p=0}^{\frac{k}{2}}(-1)^pT(k,p)-(-1)^{\frac{k}{2}}\frac{1}{2}T(k,k/2)=\sum_{a=0}^k(-1)^aB_{k-a,a,0}.
				\end{equation}
				\item When $k$ is odd, 
					\begin{equation}\label{e:Cl2}
					\sum_{p=0}^{\lfloor\frac{k}{2}\rfloor}(-1)^p(k-2p)T(k,p)=\sum_{a=0}^{\lfloor\frac{k}{2}\rfloor}(-1)^a(k-2a)(B_{k-a,a,0}+B_{a,k-a,0}).
				\end{equation}
			\end{enumerate}
		\end{lemma}
		 
		 \begin{proof}
         Both identities can be obtained by the well-known binomial identities. 
		 	\begin{enumerate}
		 	    \item Let $k=2m$ and fix $0 \leq n_0 \leq p \leq m$ and $0 \leq i_0 \leq 2m-n_0$. Then the coefficient of $B_{2m-n_0-i_0, n_0, i_0}$ of $T(2m, p)$ is given as 
                \[
                \binom{i_0}{2m-n_0-p}+\binom{i_0}{p-n_0}
                \]
                Therefore, the coefficient of  $B_{k-n_0-i_0, n_0, i_0}$ of the LHS in \eqref{e:Cl1} becomes
                \[
                (-1)^m\binom{i_0}{m-n_0}+\sum_{p=0}^{m-1}(-1)^p \left\{ \binom{i_0}{2m-n_0-p}+\binom{i_0}{p-n_0}\right\}
                \]
                When $i_0>0$, we replace the sign of $\binom{i_0}{2m-n_0-p}$ by $(-1)^p+(2m-2p)$. Then the sum becomes $(-1)^{n_0}\sum_{p=0}^{2m-n_0}(-1)^p \binom{i_0}{p}$. Since $i_0 \leq 2m-n_0$, this sum must vanish. When $i_0=0$, the only non-trivial summand is $\binom{i_0}{p-n_0}$ when $p=n_0$. 
                \item Let $k=2m+1$ and fix $0 \leq n_0 \leq p \leq m$ and $0 \leq i_0 \leq 2m-n_0$. Then the coefficient of $B_{2m+1-n_0-i_0, n_0, i_0}$ of $T(2m, p)$ is given as 
                \[
                \binom{i_0}{2m+1-n_0-p}+\binom{i_0}{p-n_0}
                \]
                 Therefore, the coefficient of  $B_{2m+1-n_0-i_0, n_0, i_0}$ of the LHS in \eqref{e:Cl2} becomes
                \[
               \sum_{p=0}^{m}(-1)^p(2m+1-2p) \left\{ \binom{i_0}{2m+1-n_0-p}+\binom{i_0}{p-n_0}\right\}
                \]
                Note that this summation is symmetric with respect to $i_0=2m-2n_0+1$. Since we only consider $i_0$ is even, it is enough to compute the case $i_0<2m-2n_0+1$. In fact, when $i_0>0$, this summation can be simplified to 
                \[(2m-2n_0-i_0)\sum_{a=0}^{i_0/2-1}(-1)^a\binom{i_0}{a} +\frac{2m-2n_0-i_0}{2}(-1)^{i_0/2}\binom{i_0}{\frac{i_0}{2}},\]
                which vanishes for the same reason in the first assertion. When $i_0=0$, the only non-trivial summand is $\binom{i_0}{p-n_0}$ when $p=n_0$. 
		 	\end{enumerate}
		 \end{proof}
		 
		 Using these binomial identities, we  prove the following theorem.
		 \begin{theorem}\label{t:hdual}
		 	For any $a,b \geq 0$, the sum 
		 	\[B_{a,b,0}+B_{b,a,0} 
		 	\]
		 	satisfies the mirror relation \eqref{e:mir}. In particular when $a=0, b=k$ or equivalently $a=k, b=0$, this implies Theorem \ref{t:HLLY}.
		 \end{theorem}

         Recall that an element $i$ in $I_1$ or $I_2$ corresponds to a set of integral points $A_i$. For $j=1,2$, we introduce a singleton $\sum I_j$ that corresponds to the union of $A_i$'s over all $i \in I_j$. For example, $B_{\sum I_2, \sum I_1, \emptyset}+B_{\sum I_1, \sum I_2, \emptyset}$ corresponds to the coarser nef partition $\Delta=\Delta_1' \cup \Delta_2'$ and $\check{\Delta}=\check{\Delta}_1' \cup \check{\Delta}_2'$ where $\Delta'_1=\cup_{i \in I_1} \Delta_i$,$\Delta'_2=\cup_{i \in I_2}\Delta_i$, $\check{\Delta}'_1=\cup_{i \in I_1} \check{\Delta}_i$ and $\check{\Delta}'_2=\cup_{i \in I_2}\check{\Delta}_i$, with the previous choice of projective unimodular triangulations. A crucial observation, which can be deduced easily, e.g. from \eqref{e : restatement of definition of B}, is that 
         \begin{equation}\label{e : contraction relation}
         B_{I^{(-)},I^{(0)},I^{(+)}} = B_{I^{(-)},\sum I^{(0)},I^{(+)}}((|I^{(0)}|-1)/2).
         \end{equation}
         Here, the $(-)$ notation is as in \eqref{e: shifting degrees}. 
		 \begin{lemma}\label{l:clrk mir}
		 	Suppose that Theorem \ref{t:hdual} holds for all nef partitions of all reflexive polytopes $\Delta = \Delta_1\cup\dots \cup \Delta_N$ and all $a+b \leq N$. Then the sum $B_{a,b,0}+B_{b,a,0}$ satisfies the mirror relation \eqref{e:mir} for $a+b \leq N+1$ and $a,b \geq 2$. 
		 \end{lemma}
		 
		 \begin{proof}
		 	 Consider the sum $B_{I_1, I_2, \emptyset}+B_{I_2, I_1, \emptyset}$ for $|I_1|=a, |I_2|=b$ and $a+b \leq N+1$. By adding $(B_{\sum I_2, \sum I_1, \emptyset}+B_{\sum I_1, \sum I_2, \emptyset})(|I_1\cup I_2|/2 -1)$ to $B_{I_1, I_2, \emptyset}+B_{I_2, I_1, \emptyset}$ and then resumming, we obtain the following:
		 	 \begin{equation}\label{eq: summation formula}
             \left(B_{\sum I_2, \sum I_1, \emptyset}(|I_1\cup I_2|/2 -1) + B_{I_1, I_2, \emptyset}\right)+\left(B_{I_2, I_1, \emptyset}+B_{\sum I_1, \sum I_2, \emptyset}(|I_1\cup I_2|/2 -1)\right).
		 	 \end{equation}
		 	 Since $B_{I_1, I_2, \emptyset}$ can be viewed as $B_{I_1, \sum I_2, \emptyset}$ up to the degree shift as in \ref{e : contraction relation}, we can reduce the sum $(B_{\sum I_2, \sum I_1, \emptyset}+ B_{I_1, I_2, \emptyset})$ to the case $N=a+1$. The parallel argument applies to the other term $(B_{I_2, I_1, \emptyset}+B_{\sum I_1, \sum I_2, \emptyset})$. Therefore, \eqref{eq: summation formula} is equal to
             \begin{equation}
             \left(B_{\sum I_2, I_1, \emptyset}+ B_{I_1, \sum I_2, \emptyset}\right)(|I_2|/2 -1/2)+\left(B_{I_2,\sum I_1, \emptyset}+B_{\sum I_1, I_2, \emptyset}\right)(|I_1|/2 -1/2).
             \end{equation}
             The hypothesis of the Lemma implies that 
             \[
\sum_{|I_1| =a,|I_2|=b}\left[\left(B_{\sum I_2, I_1, \emptyset}+ B_{I_1, \sum I_2, \emptyset}\right)(|I_2|/2 -1/2)+\left(B_{I_2,\sum I_1, \emptyset}+B_{\sum I_1, I_2, \emptyset}\right)(|I_1|/2 -1/2)\right]
             \]
             satisfies the mirror relation \eqref{e:mir}. We may also take the sum of expressions in \eqref{eq: summation formula} over all $|I_1| = a, |I_2| = b$. This sum is equal to $B_{a,b} + B_{b,a} + \binom{a+b}{a}B_{1,1,0}$. We showed that  $\binom{a+b}{a}B_{1,1,0}$ satisfies the mirror relation \eqref{e:mir} in Example \ref{eg:k=2}, so it follows that $B_{a,b,0}+B_{b,a,0}$ satisfies the mirror relation. 
		 \end{proof}

         \begin{example}
             Consider the sum $B^{\lambda, \mu}_{\{1,2\}, \{3,4\}, \emptyset}+B^{\lambda, \mu}_{\{3,4\}, \{1,2\}, \emptyset}$. From Example \ref{eg:k=2} and Example \ref{eg:k=3}, we get 
             \[
             \begin{aligned}
                 &\left(B^{\lambda, \mu}_{\{1,2\}, \{3,4\}, \emptyset}+B^{\lambda, \mu}_{\{3,4\}, \{1,2\}, \emptyset}\right)+\left(B^{\lambda-1, \mu-1}_{\{1+2\}, \{3+4\}, \emptyset}+B^{\lambda-1, \mu-1}_{\{3+4\}, \{1+2\}, \emptyset}\right) \\
                 &=\left(B^{\lambda-\frac{1}{2}, \mu-\frac{1}{2}}_{\{1,2\}, \{3+4\}, \emptyset}+B^{\lambda-\frac{1}{2}, \mu-\frac{1}{2}}_{\{3+4\}, \{1,2\}, \emptyset}\right)+\left(B^{\lambda-\frac{1}{2}, \mu-\frac{1}{2}}_{\{1+2\}, \{3,4\}, \emptyset}+B^{\lambda-\frac{1}{2}, \mu-\frac{1}{2}}_{\{3,4\}, \{1+2\}, \emptyset}\right)
             \end{aligned}
             \]
             where $\{1+2\}=\sum\{1,2\}$ and $\{3+4\}=\sum\{3,4\}$ are the singletons introduced in the proof of Lemma \ref{l:clrk mir}; the expressions $1+2$ and $3+4$ are to be understood as formal sums. Applying the mirror relations for $k=3$ (see \eqref{e:k=3, special}) and a similar computation, this sum is equal to 
             \[
             \begin{aligned}
                 &\left(\check{B}^{d+3-\left(\lambda-\frac{1}{2}\right), \mu-\frac{1}{2}}_{\{1,2\}, \{3+4\}, \emptyset}+\check{B}^{d+3-\left(\lambda-\frac{1}{2}\right), \mu-\frac{1}{2}}_{\{3+4\}, \{1,2\}, \emptyset}\right)+\left(\check{B}^{d+3-\left(\lambda-\frac{1}{2}\right), \mu-\frac{1}{2}}_{\{1+2\}, \{3,4\}, \emptyset}+\check{B}^{d+3-\left(\lambda-\frac{1}{2}\right), \mu-\frac{1}{2}}_{\{3,4\}, \{1+2\}, \emptyset}\right) \\
                 &= \left(\check{B}^{d+4-\lambda, \mu}_{\{1,2\}, \{3,4\}, \emptyset}+\check{B}^{d+4-\lambda, \mu}_{\{3,4\}, \{1,2\}, \emptyset}\right)+\left(\check{B}^{d+2-(\lambda-1), \mu-1}_{\{1+2\}, \{3+4\}, \emptyset}+\check{B}^{d+2-(\lambda-1), \mu-1}_{\{3+4\}, \{1+2\}, \emptyset}\right)
             \end{aligned}
             \]
             Since we have the mirror relation \eqref{e:mir} for $B^{\lambda-1, \mu-1}_{\{1+2\}, \{3+4\}, \emptyset}+B^{\lambda-1, \mu-1}_{\{3+4\}, \{1+2\}, \emptyset}$, we conclude that the mirror relation \eqref{e:mir} holds for the sum $B^{\lambda, \mu}_{\{1,2\}, \{3,4\}, \emptyset}+B^{\lambda, \mu}_{\{3,4\}, \{1,2\}, \emptyset}$.
         \end{example}
        The proof of Theorem \ref{t:hdual} then reduces to proving that the mirror relation holds for $B_{N+1,0,0} + B_{0,N+1,0}$ and $B_{N,1,0} + B_{1,N,0}$ for all $N$. This can be deduced from Lemma \ref{l:Clrk identity} as argued below.
         
		 \begin{proof}[Proof of Theorem \ref{t:hdual}]
		 	We proceed by induction on $a+b$. When $a + b \leq 3$, the result has already been proved in Examples~\ref{eg:k=2} and~\ref{eg:k=3}. Suppose that the mirror relation~\eqref{e:mir} holds for all $a + b \leq N$. Then Theorem~\ref{t:toricmirror} provides the mirror relation for
		 	\begin{equation}\label{e:compl}
		 		B_{\widehat{T}}:=\sum_{a=0}^{N+1} B_{N+1-a,a,0}.
		 	\end{equation}
When $N+1$ is even, it follows from Lemma \ref{l:Clrk identity} that the sum $\Sigma_{a=0}^{N+1}(-1)^a B_{N+1-a,a,0}$ satisfies the mirror relation \eqref{e:mir}. By adding and subtracting $B_{\widehat{T}}$ in \eqref{e:compl} to this, one can see that the following sums satisfy the mirror relation \eqref{e:mir}:
		 		\[2\sum_{a=0;even}^k B_{N+1-a,a,0},\qquad 2\sum_{a=0;odd} B_{N+1-a,a,0}.
		 		\]
		 		By applying Lemma \ref{l:clrk mir} and the inductive hypothesis, we see that the sums $B_{N+1, 0,0}+B_{0,N+1,0}$ and $B_{N,1,0}+B_{1,N,0}$ satisfy the mirror relation \eqref{e:mir}.

		 		 When $N+1$ is odd, it follows from Lemma \ref{l:Clrk identity} that the sum $\Sigma_{a=0}^{\lfloor \frac{N+1}{2} \rfloor}(-1)^a(N+1-2a)(B_{N+1-a,a,0}+B_{a,N+1-a,0})$ satisfies the mirror relation \eqref{e:mir}. By Lemma \ref{l:clrk mir} and the inductive hypothesis, we may deduce that the mirror relation \eqref{e:mir} holds for 
		 		\[(N+1)(B_{N+1,0,0}+B_{0,N+1,0})-(N-1)(B_{N,1,0}+B_{1,N+1,0}).
		 		\]
		  Applying the same argument to the sum $B_{\widehat{T}}$ in \eqref{e:compl}, the mirror relation \eqref{e:mir} holds for 
		 	\[(B_{N+1,0,0}+B_{0,N+1,0})+(B_{N,1,0}+B_{1,N+1,0}).
		 	\]
		 Combining these two sums, we obtain the mirror relation for both $B_{N+1,0,0} + B_{0,N+1,0}$ and $B_{N,1,0} + B_{1,N,0}$.
		 \end{proof}
    
    \begin{remark}
    As one may notice, the Hodge number identities used in the proof of Theorem~\ref{t:HLLY} are not optimal: We only consider the sum over terms for which $|I^{(+)}|$ is even and we work with the weaker form given in~\eqref{e:grping}. This suggests that several interesting identities remain unexplored. To keep the paper concise and well-organized, we do not pursue these further and leave them to the reader.
    \end{remark}

\bibliography{hom}
\bibliographystyle{abbrv}
\end{document}